\documentclass[11pt,letterpaper]{article}
\usepackage{tikz}
\usetikzlibrary{shapes.geometric}
\usepackage{amsfonts, amsmath, amssymb, amscd, amsthm, color, graphicx, mathabx, mathrsfs, wasysym, setspace, mdwlist, calc, float, mathtools, tikz-cd}

\usepackage{tocloft}
\usepackage{enumerate}

\usepackage{xcolor}

\usepackage{hyperref,cancel}

\hypersetup{colorlinks,
    linkcolor={red!50!black},
    citecolor={blue!80!black},
    urlcolor={blue!80!black}}
\usepackage{float}

\numberwithin{equation}{section}

\newcommand{\La}{\Lambda}

\newcommand{\Ga}{\Gamma}

\newcommand{\ra}{\rightarrow}
\newcommand{\ca}{\curvearrowright}

\newcommand{\A}{\mathcal A}
\newcommand{\B}{\mathcal B}
\newcommand{\C}{\mathcal C}
\newcommand{\D}{\mathcal D}

\renewcommand{\L}{\mathcal L}
\newcommand{\M}{\mathcal M}
\newcommand{\N}{\mathcal N}
\renewcommand{\P}{\mathcal P}
\newcommand{\Q}{\mathcal Q}
\newcommand{\R}{\mathcal R}

\newcommand{\T}{\mathcal T}

\newcommand{\Z}{\mathcal Z}
\newcommand{\W}{\mathcal W}

\newcommand{\sC}{\mathscr C}
\newcommand{\sD}{\mathscr D}
\newcommand{\sE}{\mathscr E}
\newcommand{\sF}{\mathscr F}
\newcommand{\sG}{\mathscr G}
\newcommand{\sH}{\mathscr H}

\newcommand{\sP}{\mathscr P}

\newcommand{\sS}{\mathscr S}
\newcommand{\sT}{\mathscr T}
\newcommand{\sU}{\mathscr U}
\newcommand{\sV}{\mathscr V}

\newcommand{\oo}{\bar\otimes}

\theoremstyle{plain}
\newtheorem{main}{Theorem}

\newtheorem{mcor}[main]{Corollary}

\newtheorem{thm}{Theorem}[section]
\newtheorem*{thm*}{Theorem}
\newtheorem{cor}[thm]{Corollary}
\newtheorem{lem}[thm]{Lemma}
\newtheorem{claim}[thm]{Claim}
\newtheorem{prop}[thm]{Proposition}

\theoremstyle{definition}
\newtheorem{defn}[thm]{Definition}

\newtheorem{ex}[thm]{Example}
\theoremstyle{plain}
\newtheorem{rem}[thm]{Remark}

\begin{document}

\title{Rigidity for von Neumann algebras of  graph product groups II. Superrigidity results}
\author{Ionu\c t Chifan, Michael Davis, and Daniel Drimbe}
\date{}

\maketitle
\begin{abstract} \noindent In \cite{CDD22} we investigated the structure of $\ast$-isomorphisms between von Neumann algebras $L(\Gamma)$ associated with graph product groups  $\Gamma$ of flower-shaped graphs and property (T) wreath-like product vertex groups as in \cite{CIOS21}. In this follow-up we continue the structural study of these algebras by establishing that these graph product groups $\Gamma$  are entirely recognizable from the category of all von Neumann algebras arising from an \emph{arbitrary} non-trivial graph product group with infinite vertex groups. A sharper $C^*$-algebraic version of this statement is also obtained. In the process of proving these results we also extend the main $W^*$-superrigidity result from \cite{CIOS21} to direct products of property (T) wreath-like product groups.   
 \end{abstract}
\section{Introduction}

In \cite{MvN43} Murray and von Neumann associated in a natural way a von Neumann
algebra, denoted by $L(\Gamma)$, to every countable discrete group $\Gamma$. Precisely, $L(\Gamma)$ is defined as the weak operator closure of the complex group algebra $\mathbb C[\Gamma]$ acting by left convolution on the Hilbert
space $\ell^2(\Gamma)$ of square-summable functions on $\Gamma$. The classification of group von Neumann algebras has since been a central theme in operator algebras driven by the following fundamental question: what aspects of the group $\Gamma$ are remembered by $L(\Gamma)$? This is a challenging problem as von Neumann algebras tend to forget a lot of the information about the groups from which they were constructed. An excellent illustration of this is Connes’ theorem which asserts that II$_1$ factors arising from amenable groups are isomorphic to the hyperfinite II$_1$ factor \cite{Co76}. Hence, group von Neumann algebras of icc amenable groups have no memory of the algebraic structure of the underlying group. 

In sharp contrast, the non-amenable case is far more complex. The emergence of Popa’s deformation/rigidity theory \cite{Po06} has led to the discovery of groups with certain canonical algebraic properties which are completely retained by their von Neumann algebra. We highlight here only a few of these developments and refer the reader to the surveys
\cite{Po06,Va10b,Io12, Io17} for a more complete account in this direction. Popa's strong rigidity theorem \cite{Po03,Po04} asserts that wreath product groups $\mathbb Z/2\mathbb Z\wr\Gamma$, where $\Gamma$ is an icc property (T) group, are completely recognizable from the category of all von Neumann algebras arising from arbitrary wreath product groups with abelian base and icc acting group. Several years later  Ioana, Popa and Vaes discovered in \cite{IPV10} the first examples of W$^*$-superrigid groups, i.e.\ groups that can be entirely reconstructed from their von Neumann algebras. Several additional classes of examples were unveiled subsequently \cite{BV12,Be14,CI17,CD-AD20,CD-AD21,CIOS21}.

In this paper we study superrigidity aspects of von Neumann algebras of graph product groups \cite{Gr90}. These groups are natural generalizations of right-angled Artin and Coxeter groups and play an important role in several subareas of topology and group theory. These groups display a rich structure and have been investigated intensively over the last two decades using deep methods in geometric group theory. In this direction, several landmark results have been discovered, see \cite{HW08, AM10, Wi11, MO13,Ag13}. More recently, certain classes of graph product groups, including many right-angled Artin groups, have been studied through the lens of measured group theory and led to strong rigidity results in the orbit equivalence setting \cite{HH20,HH21}.

Graph product groups have also been considered in the analytic framework of von Neumann algebras \cite{CF14, Ca16, CdSS17, DK-E21,CK-E21, CDD22} where various structural and rigidity results of von Neumann algebras of graph product groups have been obtained. 
In \cite{CDD22} we further developed some of Popa's powerful deformation/rigidity theoretic methods  \cite{Po06} which enabled us to completely describe the structure of all $\ast$-isomorphisms between von Neumann algebras  arising from the fairly large class of graph product groups associated with flower-shaped graphs (see class $\rm{CC}_1$ in Definition \ref{definition.CC1} below) and vertex groups which are property (T) wreath-like product groups introduced in \cite{CIOS21}. In essence, our result can be viewed as a von Neumann algebraic counterpart of the group theoretic result of Genevois-Martin \cite{GM19}.


This paper continues the study initiated in \cite{CDD22}, the main goal being to establish strong rigidity results for the aforementioned graph product groups. For example, we show that any graph product group associated with an asymmetric $\rm{CC}_1$ graph and vertex groups that are property (T) wreath-like product groups is completely  recognizable from the category of all von Neumann algebras arising from an \emph{arbitrary} non-trivial graph product group with infinite vertex groups (see Theorem \ref{grprodsuper'}). Along the way we also extend the main W$^*$-superrigidity result from \cite{CIOS21} to direct product groups  consequently obtaining new examples of groups satisfying Connes Rigidity Conjecture (see Corollary \ref{superprod'}).


\subsection{Statements of the main results}\label{main.results}
 
Before stating our results, we recall the construction of graph product groups \cite{Gr90}. Let  $\mathscr G=(\mathscr V,\mathscr E)$ be a finite graph without loops or multiple edges. The \emph{graph product group} $\Ga={\mathscr G}\{\Ga_v\}$ of a given  family  of \emph{vertex groups} $\{\Ga_v\}_{v\in \mathscr V}$ is the quotient of the free product $\ast_{v\in \mathscr V} \Ga_v$ by the relations $[\Ga_u,\Ga_v]=1$, whenever  $(u,v)\in \mathscr E$. Note that graph products can be seen as groups that ``interpolate'' between the direct product $\times_{v\in \mathscr V} \Ga_v$ (when $\mathscr G$ is complete) and the free product $\ast_{v\in \mathscr V} \Ga_v$ (when $\mathscr G$ has no edges).
For any subgraph $\sH= (\sU,\sF)$ of $\sG$ we  denote by $\Gamma_\sH$ the subgroup generated by $\langle \Gamma_u \,:\,u\in \sU\rangle $ and we call it the \emph{full subgroup} of $\sG\{\Gamma_v\}$ corresponding to $\sH$. 
A \emph{clique} $\sC$ of $\sG$ is a maximal, complete subgraph of $\sG$. The set of cliques of $\sG$ will be denoted by ${\rm cliq}(\sG)$.     The full subgroups $\Ga_\sC$ for $\sC\in \rm cliq(\sG)$ are called the clique subgroups of $\sG\{\Ga_v\}$.  

We are now ready to recall the class of graphs  ${\rm CC}_1$ that was introduced in \cite{CDD22} and which is used for our main results. 

\begin{defn}\label{definition.CC1}
A graph  $\sG$ is called a \emph{simple cycle of cliques (abbrev.\ ${\rm CC}_1$)} if there is an enumeration of its clique set $ {\rm cliq}(\sG)= \{\sC_1, ..., \sC_n\} $ with $n\geq 4$ such that the subgraphs $\sC_{i,j} :=\sC_i \cap \sC_j$ satisfy: 
\begin{equation*} \begin{split}
    &\sC_{i,j}=\begin{cases} \emptyset,  \text{ if }  \hat i-\hat j\in \mathbb Z_{n} \setminus \{ \hat 1, \widehat {n-1}\}\\
\neq \emptyset, \text{ if }  \hat i-\hat j\in \{\hat 1, \widehat {n-1}\} 
\end{cases} \text{ and }\\
&\mathscr{C}^{\rm int}_i:=\sC_i \setminus  \left(\sC_{i-1,i} \cup \sC_{i,i+1}\right) \neq \emptyset\text{, for all }1\leq k\leq n, \text{ with conventions } 0=n \text{ and } n+1=1.\end{split}
  \end{equation*}

\end{defn}

Note this automatically implies the cardinality $|\sC_i |\geq 3$ for all $i$.  Also such an enumeration $ {\rm cliq}(\sG)= \{\sC_1,..., \sC_n\} $ is called \emph{a consecutive cliques enumeration}.

\begin{ex}
A basic example of such a graph is any simple, length $n$, cycle of triangles $\sF_n=(\sV_n,\sE_n)$, which essentially looks like a flower shaped graph with $n$ petals:    

\begin{equation}\label{flower}
\begin{tikzpicture}[
regular polygon colors/.style 2 args={
    append after command={%
        \pgfextra
        \foreach \i [count=\ni, remember=\ni as \lasti (initially #1)] in {#2}{
            \draw[thick,\i] (\tikzlastnode.corner \lasti) --(\tikzlastnode.corner \ni);}
        \endpgfextra
    }
},
]

\node[thick, minimum size={2*3cm},regular polygon,regular polygon
 sides=12, rotate=11.25, regular polygon colors={12}{black, black, black, black, black, black, black, black, black, black, black, white}] at
 (3,3) (16-gon)[text=blue] {};

\node[thick, draw=white,minimum size={2*4cm},regular polygon,regular polygon
 sides=12, rotate=-4] at
 (3,3) (sixteengon)[text=blue] {};

\node[thick, draw=white,minimum size={2*4cm},regular polygon,regular polygon
 sides=12, rotate=26.5] at
 (3,3) (6teengon)[text=blue] {};

\path (16-gon.corner 12) -- node[auto=false, rotate=-45]{\ldots} (16-gon.corner 11);

\draw[thick] (6teengon.corner 12) -- (16-gon.corner 12);
 \node [fill=black,circle,inner sep=2pt] at (6teengon.corner 10){};
  
 \foreach \p [count=\n] in {1,...,12}{
 \node [fill=black,circle,inner sep=2pt] at (16-gon.corner \n){};
}
 \foreach \p [count=\n] in {3,...,12}{
 \node [fill=black,circle,inner sep=2pt] at (sixteengon.corner \n){};
\draw[thick] (6teengon.corner \n) -- (16-gon.corner \n);
}
 \foreach \p [count=\n] in {1,...,9}{
 \node [fill=black,circle,inner sep=0pt] at (6teengon.corner \n){};
}
\foreach \p [count=\n] in {1,...,11}{
 \draw[thick] (sixteengon.corner \n) -- (16-gon.corner \n);
}

\end{tikzpicture}
\end{equation}
\end{ex}

\noindent In fact any graph from ${\rm CC}_1$ is a two-level clustered graph that is a specific retraction of $\mathscr F_n$; for more details the reader may consult \cite[Section 2]{CDD22}.



In this paper we investigate various superrigidity aspects for group von Neumann algebras of the aforementioned graph product groups. Since the underlying  groups are essentially built-up from collections of large clique groups that are just direct product groups, it is natural to first tackle the superrigidity question for these types of groups. One could think of this as being the degenerate case. In this direction we were able to establish a product rigidity result for property (T) wreath-like product groups in the same spirit with \cite[Theorem A]{CdSS15} or the more recent results \cite{CD-AD20,Dr20}. Specifically, we have the following:

\begin{main}\label{prodrigid1}\label{A} For every $1\leq k\leq n$, let  $\Ga_k\in \mathcal W\mathcal R (A_k, B_k \ca I_k)$ be property (T) groups where $A_k$ is abelian, $B_k$ is an icc  subgroup of  a hyperbolic group, $B_k \curvearrowright I_k$ has amenable stabilizers and denote  $\Ga= \Ga_1\times  \dots \times \Ga_n$. 

\noindent Assume that $t>0$ is a scalar and $\La$ is an arbitrary group satisfying $\M=\L(\Ga)^t=\L(\La)$. 

\noindent Then one can find a direct product decomposition $\La= \La_1\times\dots\times \La_n$, some scalars $t_1,\dots,t_n>0$ with $t_1\cdots t_n=t$ and a unitary $u\in \M$ satisfying $\L(\Gamma_{i})^{t_i}=u\L(\La_i) u^*$, for any $1\leq i \leq n.$

\end{main}

This theorem in conjunction with the W$^*$-superrigidity result from  \cite[Theorem 9.9]{CIOS21} immediately yields that essentially all finite direct products of property (T) wreath-like product groups covered by the prior theorem are completely recognizable  from their von Neumann algebras. Before we introduce the result, we recall the notion of group-like $*$-isomorphism between von Neumann algebras. Let $\Gamma$ and $\Lambda$ be countable groups.  Let $\eta:\Gamma \ra \mathbb T$ be a multiplicative character and $\delta : \Gamma\ra \Lambda$ a group isomorphism. Consider the group von Neumann algebras $\mathcal L(\Gamma)$ and $\mathcal L(\Lambda)$ and denote by $\{u_g \,:\, g\in \Gamma\}$ and $\{v_\lambda \,:\, \lambda\in \Lambda\}$ their corresponding canonical group unitaries. Then the map $\Gamma\ni u_g\ra \eta(g)v_{\delta(g)}\in\Lambda$ canonically extends to a $\ast$-isomorphism denoted by $\Psi_{\eta,\delta}: \mathcal L(\Gamma)\ra \mathcal L(\Lambda)$.

\begin{mcor}\label{superprod'} For every $1\leq k\leq n$, let  $\Ga_k\in \mathcal W\mathcal R (A_k, B_k \ca I_k)$ be property (T) groups where $A_k$ is abelian, $B_k$ is an icc  subgroup of  a hyperbolic group, $B_k \curvearrowright I_k$ has amenable stabilizers and the set $\{i\in I \; | \; g\cdot i\neq i\}$ is infinite for any $g\in B_k\setminus \{1\}$. Denote $\Ga= \Ga_1\times  \dots \times \Ga_n$.

\noindent Let $t>0$ be any scalar,  $\La$ be an arbitrary group and $\theta: \L(\Ga)^t\ra \L(\La)$ be any $\ast$-isomorphism.  

\noindent Then $t=1$ and one can find a character $\eta:\Gamma\to\mathbb T$, a group isomorphism $\delta:\Gamma\to\Lambda$ and a unitary $u\in \L(\La)$ satisfying $\theta={\rm ad}(u)\circ \Psi_{\eta,\delta}$.

\end{mcor}

We mention in passing that since  property (T) passes to (finite) direct products of groups, Corollary \ref{superprod'} provides new examples of property (T) groups which satisfy Popa's strengthening of Connes Rigidity Conjecture \cite{Co82,Po06}. In  particular, the result also shows that these property (T) factors have trivial fundamental group and also verify Jones'  outer automorphism problem \cite{Jo00} providing additional examples to recent similar results \cite{CDHK20,CIOS21,CIOS23}.

To tackle the problem of reconstructing the above graph product groups $\Ga$ from their factors $\L(\Ga)$ we develop new aspects of a more conceptual principle from \cite{CI17,CD-AD20} called \emph{peripheral reconstruction $W^*$-method}; this  consists of exploiting the natural tension that occurs between ``a peripheral structure'' and a ``direct product structure'' in the group.  In our specific situation this means that if $\La$ is any group such that $\L(\Ga)=\L(\La)$, then the main strategy is to first identify in $\Lambda$ collections of subgroups that play the same role as the ``peripheral structure'' of $\Gamma$ given by its full subgroups associated to the cliques in the underlying graph. 

In this direction, using an approach combining the comultiplication map \cite{Io10, IPV10,PV11}, the ultrapower methods from \cite{Io11} (see also \cite{DHI16,CdSS15,KV15}), a method for reconstructing malnormal groups from \cite{CD-AD20,CI17}, and Corollary \ref{superprod'} we are able to show that the clique subgroups of our graph product groups are in fact completely recognizable under the W$^*$-equivalence. For the precise statement, the reader may consult Theorem \ref{theorem.almost.Wsuperrigidity}. While this result does not show complete reconstruction of these graph products from the category of all group von Neumann algebras, it can be used effectively to show these groups are recognizable from the von Neumann algebras associated with an extensive family of groups---all non-trivial graph products with infinite vertex groups.

\begin{main}\label{grprodsuper'} Let $\sG\in{\rm CC}_1$, let ${\rm cliq}(\sG)=\{ \sC_1, \ldots , \sC_n\}$ be its consecutive cliques enumeration and assume that $|\sC_i|\neq |\sC_j|$ whenever $i\neq j$.   Let $\Gamma= \sG\{\Gamma_v\}$ be a graph product group where all vertex groups $\Gamma_v$ are property (T) wreath-like product groups $\Gamma_v\in  \mathcal W\mathcal R(A_v,B_v)$ with $A_v$ abelian and $B_v$ an icc subgroup of a hyperbolic group.

\noindent Then for any nontrivial graph product group $\La$ with infinite vertex groups  satisfying $\L(\Ga)\cong \L(\La)$, we have  $\Ga\cong \Lambda$. \end{main}

In fact, we have obtained a more precise version of the above result (see Theorem \ref{superwithincat}). Not only can we derive an isomorphism of the underlying groups but we can in fact completely describe all isomorphisms between $\L(\Gamma)$ and $\L(\La )$. Namely, they appear as compositions between the canonical group-like isomorphisms $\Psi_{\eta, \delta }$ induced by a group isomorphism $\delta:\Gamma\to\Lambda$, a character $\eta:\Gamma\to\mathbb T$ and the local automorphisms of graph product group von Neumann algebras introduced in \cite{CDD22} (see also Section \ref{local.isom}). 

The result yields new applications towards rigidity in the $C^*$-algebraic framework. Since these graph product groups have trivial amenable radical (see \cite[Lemma 4.3]{CDD22}) using \cite[Theorem 1.3]{BKKO14} it follows that their reduced C$^*$-algebras have unique trace. Therefore  the conclusion of Theorem \ref{grprodsuper'} still holds true for reduced group $C^*$-algebras. Moreover, if one assumes in addition these groups are torsion free, then the previous theorem actually yields a formally stronger reconstruction statement allowing us to drop the hypothesis assumption that the vertex groups of the target $\La$ are infinite.

\begin{mcor}Let $\sG\in{\rm CC}_1$, let ${\rm cliq}(\sG)=\{ \sC_1,\ldots , \sC_n\}$ be a consecutive cliques enumeration and assume that $|\sC_i|\neq |\sC_j|$ whenever $i\neq j$.  Let $\Gamma= \sG\{\Gamma_v\}$ be a graph product group where all vertex groups $\Gamma_v$ are torsion free, property (T) wreath-like product groups $\Gamma_v\in  \mathcal W\mathcal R(A_v,B_v)$ with $A_v$ abelian and $B_v$ an icc subgroup of a hyperbolic group.

\noindent Then for any nontrivial graph product group $\La$ satisfying $C^*_r(\Ga)\cong C^*_r(\La)$, we have  $\Ga\cong \Lambda$. 
\end{mcor}

To see this just notice that the graph product groups $\Ga$ covered by this corollary are torsion free and also satisfy the Baum-Connes conjecture; see Proposition \ref{bcpreserv}, \cite[Theorem 7.7]{O-O01b} and \cite[Theorem 20]{MY02}. Hence, their reduced $C^*$-algebras  $C^*_r(\Ga)$ are projectionless and so are $C^*_r(\La)$. This further entails that $\La$ is also torsion free; in particular, the vertex subgroups of $\La$ are automatically infinite. Thus, the conclusion follows from Theorem \ref{grprodsuper'}.  

{\bf Acknowledgements.} We are grateful to the referee for all the comments and remarks that significantly improved the exposition of the paper. I.C. was partially supported by NSF FRG Grant \#1854194 and NSF Grant \#DMS-2154637; D.D. was supported by the postdoctoral fellowship fundamental research 12T5221N of the Research Foundation Flanders.

\section{Preliminaries}

\subsection{Terminology}
Throughout this article all von Neumann algebras are denoted by calligraphic letters e.g. $\M$, $\N$, $\P$, $\Q$, etc. All von Neumann algebras $\M$ will be tracial, i.e. endowed with a unital, faithful, normal linear functional $\tau:\M\rightarrow \mathbb C$  satisfying $\tau(xy)=\tau(yx)$, for all $x,y\in \M$. This induces a norm on $\M$ given by the formula $\|x\|_2=\tau(x^*x)^{1/2}$, for any $x\in \M$. The $\|\cdot\|_2$-completion of $\M$ will be denoted by $L^2(\M)$. 

Given a von Neumann algebra $\M$, we will denote by $\mathscr U(\M)$ its unitary group, by $\mathcal P(\M)$ its projections set and by $\mathcal Z(\M)$ its center.  Given a unital inclusion $\N\subset \M$ of von Neumann algebras, we denote by $E_{\N}:\M\to \N$ the unique $\tau$-preserving {\it conditional expectation} from $\M$ onto $\N$, by $e_{\N}:L^2(\M)\to L^2(\N)$ the orthogonal projection onto $L^2(\N)$ and by $\langle \M,e_\N\rangle$ the Jones' basic construction of $\N\subset \M$.
Also, we denote by $\N'\cap \M =\{ x\in \M \,:\, [x, \N]=0\}$ the relative commmutant of $\N$ inside $\M$ and by $\mathscr N_\M(\N)=\{ u\in \mathscr U(\M)\,:\, u\N u^*=\N\}$ the normalizer of $\N$ inside $\M$. We say that the inclusion $\N$ is regular in $\M$ if $\mathscr N_{\M}(\N)''=\M$ and
irreducible if $\N'\cap \M=\mathbb C 1$.	

For a group inclusion $\Sigma< \Gamma$  we denote by $C_{\Gamma}(\Sigma)=\{ g\in \Gamma\,|\, [g,\Sigma]=1\}$ its {\it centralizer} in $\Gamma$ and by $vC_{\Gamma}(\Sigma)=\{ g\in \Gamma\,|\, |g^\Sigma|<\infty\}$ its {\it virtual centralizer}. Note that $vC_\Gamma(\Gamma)=1$ precisely when $\Gamma$ is icc. 
We denote by $N_\Gamma(\Sigma)=\{ g\in\Gamma\,|\, g\Sigma g^{-1}=\Sigma \}$ the {\it normalizer} of $\Sigma$ in $\Gamma$.

\subsection{Graph product groups}
We now recall the notion of graph product groups introduced by E. Green  \cite{Gr90} while also highlighting some of their properties that are relevant to our work. Let $\sG=(\sV,\sE)$ be a finite simple graph, where $\sV$ and $\sE$  denote its vertex and edge sets, respectively. Let $\{\Gamma_v\}_{v\in\sV}$ be a family of groups called vertex groups. The graph product group associated with this data, denoted by $\sG\{\Gamma_v,v \in \sV\}$ or simply $\sG\{\Gamma_v\}$, is the group generated  by $\Gamma_v$, $v\in \sV$  with the only relations being $[\Gamma_u, \Gamma_v] = 1$, whenever $(u,v)\in \sE$. 
Given any subset $\sU\subset \sV$, the subgroup $\Gamma_\sU =\langle \Gamma_u \,:\,u\in \sU\rangle $ of $\sG\{\Gamma_v,v\in \sV\}$ is called a full subgroup. This can be identified  with the graph product $\sG_\sU\{\Gamma_u,u \in \sU\}$ corresponding to the subgraph $\sG_\sU$ of $\sG$, spanned by the vertices of $\sU$. 
For every $v \in \sV$ we denote by ${\rm lk}(v)$ the subset of vertices $w\neq v$ so that $(w,v)\in \sE$. Similarly, for every $\sU 
\subseteq \sV$ we denote by ${\rm lk}(\sU) = \cap _{u\in \sU}{\rm lk}(u)$. Also, we make the convention that ${\rm lk}(\emptyset) = \sV$. Notice that $\sU \cap  {\rm lk}(\sU) = \emptyset$.

Graph product groups naturally admit many amalgamated free product decompositions as follows (see \cite[Lemma 3.20]{Gr90}). For any $w \in \sV$ we have
\begin{equation}\label{afpdesc}\sG\{\Gamma_v\} = \Gamma_{\sV\setminus \{w\}} \ast_{ \Gamma_{\rm lk}(w)} \Gamma_{{\rm st}(w)},\end{equation}
where ${\rm st} (w) = \{w\} \cup {\rm lk} (w)$. Notice that $\Gamma_{{\rm lk}(w)}\lneqq \Gamma_{{\rm st}(w)}$, but it could be the case that $\Gamma_{{\rm lk}(w)}=\Gamma_{\sV\setminus \{w\}} $, when $\sV={\rm st}(w)$. In this case, the amalgam decomposition is called degenerate.

Similarly, for every subgraph $\sU\subset \sG$ we denote by ${\rm st}(\sU)= \sU \cup {\rm lk}(\sU )$. A maximal complete subgraph $\sC\subseteq \sG$ is called a clique and the collections of all cliques of $\sG$ will be denoted by  ${\rm cliq}(\sG)$. Below we highlight various properties of full subgroups that will be useful in this paper. 

\begin{prop}\label{proposition.AM10}
Let $\Gamma=\sG \{\Gamma_v\}$ be any graph product of groups, $g\in\Gamma$ and let $\sS,\sT\subseteq \sG$ be any subgraphs. Then the following hold.
\begin{enumerate}
    \item \emph{\cite[Lemma 3.9]{AM10}} If $g \Gamma_{\sT} g^{-1}\subset \Gamma_{\sT}$, then $g \Gamma_{\sT} g^{-1}=\Gamma_{\sT}$.
    
    \item \emph{\cite[Proposition 3.13]{AM10}} $N_{\Gamma}(\Gamma_{\sT})=\Gamma_{\sT\cup {\rm link}(\sT)}$.
    
    \item \emph{\cite[Proposition 3.4]{AM10}} There exist $\sD\subseteq \sS\cap \sT$ and $h\in \Gamma_\sT$ such that $g \Gamma_{\sS} g^{-1}\cap \Gamma_{\sT}=h \Gamma_{\sD} h^{-1}$.
\end{enumerate}

\end{prop}

For further use we also record the following easy consequence of results in \cite{O-O01b,O-O98}.

\begin{prop}\label{bcpreserv} Let $\Gamma=\sG \{\Gamma_v\}$ be any graph product of groups. Then the following hold:
\begin{enumerate}
    \item If $\Ga_v$ is torsion free for all $v\in \mathscr V$, then $\Ga$ is also torsion free;
    \item If $\Ga_v$ is torsion free and satisfies Baum-Connes conjecture for all $v\in \mathscr V$, then so is $\Ga$.
\end{enumerate}

\end{prop}
\begin{proof}  Torsion free passes to both direct products and free product with amalgamation. Thus using either product decomposition or \eqref{afpdesc} the statement follows by induction on the number of vertices in $\sG$.

  To see part 2.\ notice that the class of torsion free groups that satisfies Baum-Connes property is closed under taking finite direct products \cite{O-O01b} and also under taking free product with amalgamation \cite{O-O98}. Once again, using either product decomposition or \eqref{afpdesc} the statement follows by induction on the number of vertices in $\sG$.  
\end{proof}

\subsection{Cycles of cliques graphs and their von Neumann algebras}\label{local.isom}

In the first part of this section we recall from \cite[Section 7]{CDD22} a canonical family of $*$-isomorphisms between graph product group von Neumann algebras when the underlying graphs belong to ${\rm CC}_1$. Let $\sG,\sH \in {\rm CC}_1$ be isomorphic graphs and fix $\sigma : \sG \ra \sH$ an isometry. Let ${\rm cliq}( \sG) =\{ \sC_1,  \ldots ,\sC_n\}$ be an enumeration of consecutive cliques. Let $\Gamma_\sG$ and $\Lambda_\sH$ be graph product groups and assume that for every $1\leq i\leq n $ there are $\ast$-isomorphisms $\theta_{i-1,i}: \L(\Gamma_{\sC_{i-1,i}})\ra \L(\Lambda_{\sC_{\sigma (\sC_{i-1,i})}}) $,   $\xi_{i}:\L(\Gamma_{\sC^{\rm int}_i})\ra \L(\Lambda_{\sigma(\sC^{\rm int}_i)})$ and $\theta_{i,i+1}:\L(\Gamma_{\sC_{i,i+1}})\ra \L(\Lambda_{\sC_{\sigma (\sC_{i,i+1})}}) $; here and in what follows we convene as before that $n=0$ and $n+1=1$.  By \cite[Theorem 7.1]{CDD22} (see also \cite{CF14})  these $\ast$-isomorphisms induce a unique $\ast$-isomorphism $\phi_{\theta,\xi, \sigma}:\L(\Gamma_\sG)\ra  \L(\Lambda_\sH)$ defined as 
\begin{equation}\label{branchaut'} \phi_{\theta,\xi, \sigma}(x)=\begin{cases}
\theta_{i-1,i}(x), \text{ if } x\in \L(\Gamma_{\sC_{i-1,i}})\\
\xi_i(x), \qquad \text{if }  x\in \L(\Gamma_{\sC^{\rm int}_i })
\end{cases} 
\end{equation}
for all $1\leq i\leq n$. 

When $\Gamma_\sG =\La_\sH$ this construction yields a group of $\ast$-automorphisms of $\L(\Gamma_\sG)$, denoted by ${\rm Loc}_{\rm c,g}(\L(\Gamma_\sG))$. We also denote by ${\rm Loc}_{\rm c}(\L(\Gamma_\sG))$ the subgroup of all local automorphisms satisfying $\sigma={\rm Id}$.  Next, we highlight a class of automorphisms in ${\rm Loc}_{\rm c}(\L(\Gamma_\sG))$ needed to state Theorem \ref{superwithincat}. Consider $n$-tuples $a =(a_{i,i+1})_i$ and $b= (b_i)_i$ of nontrivial unitaries  $a_{i,i+1} \in \L(\Gamma_{\sC_{i-1,i}})$ and $b_i\in \L(\Gamma_{\sC^{\rm int}_i})$, for every $1\leq i\leq n$. If in  \eqref{branchaut'} we let $\theta_{i,i+1}= {\rm ad} (a_{i,i+1}) $ and $\xi_i ={\rm ad} (b_i)$, then the corresponding local automorphism $\phi_{\theta,\xi, {\rm Id}}$ is in general an outer automorphism of $\L(\Gamma)$ (see \cite[Proposition 7.4]{CDD22}) and will be denoted by $\phi_{a,b}$ throughout the paper. 

We conclude this subsection by recording an important result from \cite{CDD22} that is essential for deriving Theorem \ref{grprodsuper'}.
\begin{thm}\emph{\cite[Theorem 5.2]{CDD22}}\label{cyclerel1} Let $\sG=\{\sV,\sE\}$ be a  graph in the class ${\rm CC}_1$ and let $\sC_1,...,\sC_n$ be an enumeration of its consecutive cliques. Let \{$\Gamma_v|\;v\in \sV\}$ be a collection of icc groups and let $\Gamma_\sG$ be the corresponding graph product group.   For each $1\leq i\leq n$ let $a_{i,i+1}\in \sU(\L(\Gamma_{\sC_{i,i+1}}))$, $b_{i,i+1}\in \sU(\L(\Gamma_{\sC_i \cup \sC_{i+1}\setminus \sC_{i,i+1}})$ and denote $x_{i,i+1}= a_{i,i+1} b_{i,i+1}$.

\noindent If $x_{1,2} x_{2,3}\cdots x_{n-1,n}x_{n,1}=1$, then for each $1\leq i\leq n $ one can find $a_i \in \sU(\L(\Gamma_{\sC_{i-1,i}}))$, $b_i \in \sU(\L(\Gamma_{\sC^{\rm int}_i}))$ and $c_i \in \sU(\L (\Gamma_{\sC_{i,i+1}}))$  so that $x_{i,i+1}= a_i b_i c_i b^*_{i+1} a^*_{i+2}c^*_{i+1}$. Here we convene that $n+1=1$, $n+2=2$, etc.     

\end{thm}

\subsection{Wreath-like product groups}

In \cite{CIOS21} a new category of groups called \emph{wreath-like product groups} was introduced.  To recall their construction let $A$ and $B$ be any countable groups and let $B \ca I$ be an action on a countable set. One says $W$ is a wreath-like product of $A$ and $B\ca I$  if it can be realized as a group  extension 

\begin{equation}\label{regwreathlike1'''}
    1 \ra \bigoplus_{i\in I} A_i\hookrightarrow W \overset{\varepsilon}{\twoheadrightarrow} B \ra 1 
 \end{equation}
 which satisfies the following properties:
 \begin{enumerate}
     \item [a)] $A_i\cong A$ for all $i\in I$, and 
     \item [b)] the action by conjugation of  $W$ on $\bigoplus_{i\in I} A_i$  permutes the direct summands according to the rule \begin{equation*}w A_i w^{-1}= A_{\varepsilon(w)i}\text{ for all }w\in W, i\in I.\end{equation*}
 \end{enumerate}
 The class of all such wreath-like groups is denoted by $\W\R(A, B\ca I)$. When $I= B$ and the action $B\ca I$ is by translation this consists of so-called regular wreath-like product groups and we simply denote their class by $\W\R(A, B)$.
 
 Notice that every classical generalized wreath product $A\wr_I B \in \W\R (A,B\ca I)$. However, building examples of non-split wreath-like products is far more involved. Indeed, using deep methods in group theoretic Dehn filling and Cohen-Lyndon subgroups it was shown in \cite{CIOS21} that large classes of such wreath-like products exist, including many with property (T).



 \begin{thm}[\text{\cite{CIOS21}}]\label{AHQ}
Let $G$ be a hyperbolic group. For every finitely generated group $A$, there exists a quotient $W$ of $G$ such that $W\in \W\R (A,B)$ for some hyperbolic group $B$. In particular, when $G$ has property (T) (e.g.\ any uniform lattice in $Sp(n,1)$, $n\geq 2$) then so does $W$.
\end{thm}





\subsection {Popa's intertwining-by-bimodules techniques} We next recall from  \cite [Theorem 2.1, Corollary 2.3]{Po03} Popa's {\it intertwining-by-bimodules} technique.
Let $\Q\subset \M$ be a  von Neumann subalgebra. {\it Jones' basic construction} $\langle \M,e_\Q\rangle$ is defined as the von Neumann subalgebra of $\mathbb B(L^2(\M))$ generated by $\M$ and the orthogonal projection $e_\Q$ from $L^2(\M)$ onto $L^2(\Q)$. The basic construction $\langle \M,e_\Q\rangle$ has a faithful semi-finite trace given by $\text{Tr}(xe_\Q y)=\tau(xy)$, for every $x,y\in \M$. We denote by $L^2(\langle \M,e_\Q\rangle)$ the associated Hilbert space and endow it with the natural $\M$-bimodule structure.

\begin{thm}[\cite{Po03}]\label{corner} Let $(\M,\tau)$ be a tracial von Neumann algebra and $\P\subset p\M p,\Q\subset q\M q$  be von Neumann subalgebras. 
Then the following  are equivalent:

\begin{enumerate}

\item There exist projections $p_0\in \P, q_0\in \Q$, a $*$-homomorphism $\theta:p_0\P p_0\rightarrow q_0\Q q_0$  and a non-zero partial isometry $v\in q_0\M p_0$ such that $\theta(x)v=vx$, for all $x\in p_0\P p_0$.


\item There is no sequence $(u_n)_{n\ge 1}\subset\mathcal U(\P)$ satisfying $\|E_\Q(x^*u_ny)\|_2\rightarrow 0$, for all $x,y\in p\M$.
\end{enumerate}

If one of these equivalent conditions holds true,  we write $\P\prec_{\M}\Q$, and say that {\it a corner of $\P$ embeds into $\Q$ inside $\M$.}
Moreover, if $\P p'\prec_{\M}\Q$ for any non-zero projection $p'\in \P'\cap p\M p$, then we write $\P\prec^{s}_{\M}\Q$.
\end{thm}

We continue by recording several elementary facts about intertwining results in group von Neumann algebras (of graph product groups).

\begin{lem}\emph{\cite[Lemma 2.2]{CI17}}\label{lemma.CI17} Let $\Ga_1, \Ga_2 < \Ga$ be countable groups such that $\L(\Ga_1) \prec _{\L(\Ga)} \L(\Ga_2)$.
Then one can find $g \in \Ga$ such that $[\Ga_1 : \Ga_1 \cap  g \Ga_2g^{-1}] < \infty$.

\end{lem}

\begin{cor}\emph{\cite[Lemma 2.3]{CDD22}}\label{corollary.graph.CI17}
Let $\Gamma=\sG \{\Gamma_v\}$ be any graph product of infinite groups  and let $\sS,\sT\subseteq \sG$ be any subgraphs. 
\noindent If $\L(\Gamma_{\sS})\prec_{\L(\Gamma)}  \L(\Gamma_\sT)$, then $\sS\subset\sT$.
\end{cor}

We also recall the following result which classifies all rigid subalgebras of von Neumann algebras associated to graph product groups.

\begin{thm}\emph{\cite[Theorem 6.1]{CDD22}}\label{proptcliques}\label{theorem.embed.clique} Let $\Gamma= \sG\{\Gamma_v\}$ be a graph product group, let $\Gamma\ca \P$ be any trace preserving action and denote by $\M=\P\rtimes \Gamma$ the corresponding crossed product von Neumann algebra. Let $r\in \M$ be a projection and let $\Q\subset r\M r$ be a property (T) von Neumann subalgebra. 

Then one can find a clique $\sC\in {\rm cliq}(\sG)$ such that $\Q\prec_\M \P\rtimes \Gamma_\sC$. Moreover, if $Q\nprec \P \rtimes \Gamma_{\sC \setminus \{c\}}$ for all $c\in \sC$, then one can find projections $q\in \Q$, $q'\in \Q'\cap r\M r$ with $qq'\neq 0$ and a unitary $u\in \M$ such that $u q\Q q q'u^{*}\subseteq \P \rtimes \Gamma_\sC$.   In particular, if $\P\rtimes \Gamma_\sC$ is a factor then one can take $q=1$ above.  

\end{thm}



The following result is a direct corollary of  \cite[Theorem 1.2.1]{IPP05}. For completeness, we provide all the details.

\begin{cor}\label{controlintertw5} Let $\Gamma=\sG \{\Gamma_v\}$ be any graph product of infinite groups and let $\M=\L(\Gamma)$. Let $\mathcal A\subset\L(\Gamma_v)$ be a diffuse von Neumann subalgebra, for some $v\in \mathscr G$. Then  $\mathcal A '\cap \M = \L(\Ga_{{\rm link}(v)})\bar\otimes (\A'\cap \L(\Ga_v))$.

\end{cor} 
\begin{proof} 
From definitions we have  $\mathcal A '\cap \M \supset \L(\Ga_{{\rm link}(v)})\bar\otimes (\A'\cap \L(\Ga_v))$. 
For proving the reverse containment, we note that
we can write $\M$ as an amalgamated free product $\M= \L(\Gamma_{{\rm star} (v)})\ast_{\L(\Gamma_{{\rm link} (v)})} \L(\Gamma_{\mathscr G\setminus \{v\}})$. Since $\mathcal A$ is diffuse, it follows that $\mathcal A\nprec_{\L(\Gamma_{{\rm star} (v)})} \L(\Gamma_{{\rm link} (v)}) $.
By \cite[Theorem 1.2.1]{IPP05}, it follows that $\mathcal A '\cap \M\subset \A '\cap \L(\Gamma_{{\rm star} (v)}) =\L(\Ga_{{\rm link}(v)})\bar\otimes (\A'\cap \L(\Ga_v)).$
\end{proof}

\begin{lem}\label{control.relative.commutants}
Let $\Sigma<\Gamma$ be countable groups and denote $\M=\L (\Gamma)$. Assume that $\P\subset p\M p$ and $\Q\subset q\M q$ are von Neumann subalgebras satisfying $\P\prec_{\M} \L(\Sigma)$ and $\L(\Sigma)\prec_{\M} \Q$.

\noindent If $\P'\cap p\M p$ is amenable, then $\Q '\cap q\M q$ has an amenable direct summand.
\end{lem}

\begin{proof} Note that one can find an increasing sequence of groups $\Omega_1 \leqslant \Omega_{2}\leqslant \cdots\leqslant {\rm vC}_\Gamma(\Sigma)$ normalized by $\Sigma$ with $\cup_{n\ge 1} \Omega_n={\rm vC}_\Gamma(\Sigma)$ whose centralizers form a descending sequence $\Sigma \geqslant C_\Sigma(\Omega_1)\geqslant \cdots \geqslant C_\Sigma(\Omega_n )\geqslant C_\Sigma(\Omega_{n+1})\geqslant\cdots $ of finite index subgroups. Indeed, recall that ${\rm vC}_\Gamma(\Sigma)=\{g\in\Gamma|\; g^\Sigma \text{ is finite} \}$ and let $\{\mathcal O_n\}_{n\ge 1}$ be a countable enumeration of all the finite orbits of the action by conjugation of $\Sigma$ on $\Gamma$. Note that $\Omega_n:=\langle \cup_{1\leq k\leq n} \mathcal O_k \rangle$ satisfies the assumption.

Next, since $[\Sigma: C_\Sigma(\Omega_n )]<\infty$, we get from the assumption that $\P\prec_{\M} \L(C_\Sigma(\Omega_n ))$, for all $n\ge 1$. By passing to relative commutants, we derive from \cite[Lemma 3.5]{Va08}  that $\L(\Omega_n)\prec_{\M} \P'\cap p\M p$, which implies that $L(\Omega_n)$ has an amenable direct summand for any $n\ge 1$. This further implies that there exists a non-zero projection $z_n\in \mathcal Z(L(\Omega_n))$ such that $L(\Omega_n)z_n$ is amenable (see \cite[Remark 2.2]{Io12b}). Standard arguments now imply that
$\Omega_n$ is amenable for any $n\ge 1$.  This shows that ${\rm vC}_\Gamma(\Sigma)$ is amenable.
Finally, since $\L(\Sigma)\prec_{\M} \Q$, we pass to relative commutants and apply \cite[Lemma 3.5]{Va08} to get that $\Q'\cap q\M q\prec_{\M} \L({\rm vC}_\Gamma(\Sigma))$. Since ${\rm vC}_\Gamma(\Sigma)$ is amenable, the conclusion follows.
\end{proof}

\subsection{Relative amenability}

A tracial von Neumann algebra $(\M,\tau)$ is {\it amenable} if there exists a positive linear functional $\Phi:\mathbb B(L^2(\M))\to\mathbb C$ satisfying $\Phi_{|\M}=\tau$ and $\Phi$ is $\M$-{\it central}, meaning $\Phi(xT)=\Phi(Tx),$ for all $x\in \M$ and $T\in \mathbb B(L^2(\M))$. By Connes' celebrated theorem \cite{Co76}, it follows that $\M$ is amenable if and only if $\M$ is approximately finite dimensional.
 
We continue by recalling the notion of relative amenability introduced by 
Ozawa and Popa in \cite{OP07}. Let $(\M,\tau)$ be a tracial von Neumann algebra. Let $p\in \M$ be a projection and $\P\subset p\M p,\Q\subset \M$ be von Neumann subalgebras. Following \cite[Definition 2.2]{OP07}, we say that $\P$ is {\it amenable relative to $\Q$ inside $\M$} if there exists a positive linear functional $\Phi:p\langle \M,e_\Q \rangle p\to\mathbb C$ such that $\Phi_{|p\M p}=\tau$ and $\Phi$ is $\P$-central. 
We say that $\P$ is {\it strongly non-amenable relative to} $\Q$ inside $\M$ if $\P p'$ is non-amenable relative to $\Q$ inside $\M$ for any non-zero projection $p'\in \P '\cap p\M p$.  

In this section we prove a result (Proposition \ref{reductionamenable}) that is inspired by \cite[Proposition 3.2]{PV12}; its proof is similar to the approach from \cite[Lemma 2.10]{Dr19a}. 
First we show the following well-known result which computes the basic construction of natural inclusions of  subalgebras arising from trace preserving actions of countable groups (see also \cite[Lemma 2.5]{Be14}).

\begin{lem}\label{lemma.basic.construction}
 Let $\Ga\curvearrowright \D$ be a trace preserving action of a countable group. Let $\Sigma<\Ga$ be a subgroup and denote $\M=\D\rtimes\Gamma$ and $\N=\D\rtimes\Sigma$. Denote by $\Q=(\D\bar\otimes \ell^\infty(\Gamma/\Sigma))\rtimes\Gamma$ the semifinite von Neumann algebra arising from the diagonal action $\Gamma \curvearrowright \D\bar\otimes \ell^\infty(\Gamma/\Sigma)$.
 
 Then there exists a $*$-isomorphism $\theta: \langle \M,e_{\N} \rangle\to \Q$ satisfying $\theta(x)=x$, for all $x\in \M$, and $\theta(e_{\N})=1\otimes \delta_{e\Sigma}\in \D\bar\otimes \ell^\infty(\Gamma/\Sigma)$.
\end{lem}

\begin{proof}
We denote by $\tau$ the trace on $\M$ and by ${\rm Tr}$ the natural faithful normal semifinite trace on $\Q$, so that the restriction of $(\Q, {\rm Tr})$ to $\M$ is $(\M,\tau)$. By letting $f=1\otimes \delta_{e\Sigma}$, note that the following properties hold:
\begin{enumerate}
    \item [(i)] $\Q$ is the weak closure of the $*$-subspace $\M f\M$,

    \item [(ii)] ${\rm Tr}(f)=1$ and ${\rm Tr}(xf)=\tau(x)$, for any $x\in \M$,

    \item [(iii)] $f\Q f=f\N=\N f$.
\end{enumerate}

 By using \cite[Theorem 4.3.15]{SS10} we derive from properties (i)-(iii) that there is a $*$-isomorphism $\theta: \langle \M,e_{\N} \rangle\to \Q$ with $\theta(x)=x$, for all $x\in \M$, and $\theta(e_{\N})=f$.
\end{proof}

\begin{prop}\label{reductionamenable} Let $\varepsilon:\Ga \ra \La$ be an epimorphism between countable groups and let $\Ga\curvearrowright \D$ be a trace preserving action. Let $\M = \D\rtimes\Ga$ and $\N= \L(\La)$ and denote by $\Delta: \M\ra \M\oo \N$ be the $\ast$-embedding given by $\Delta(au_g)=au_g\otimes v_{\varepsilon(g)}$ for all $a\in \D,g\in \Ga$. Here, we denoted by $\{u_g\}_{g\in\Gamma}$ the canonical unitaries that implement the action $\Gamma\curvearrowright \D$ and by $\{v_\lambda\}_{\lambda\in\Lambda}$ the canonical generating unitaries of $\L(\Lambda)$.

If $\P\subseteq p \M p$ is a von Neumann subalgebra such that $\Delta(\P)$ is amenable relative to $\M\otimes 1$ inside $\M \oo \N$, then $\P$ is amenable relative to $\D\rtimes \ker(\varepsilon)$ inside $\M$.

\end{prop}

\begin{proof}
First, we claim that there is an injective $*$-homomorphism $\theta: \langle \M, e_{\D\rtimes \ker(\varepsilon)}  \rangle  \to \langle \M\bar\otimes\N, e_{\M\otimes 1}   \rangle  $ which satisfies  $\theta(x)=\Delta(x)$, for any $x\in \M$, and $\theta(e_{\D\rtimes \ker(\varepsilon)})=e_{\M\otimes 1}$.
To see that the above claim holds, we use Lemma \ref{lemma.basic.construction} to obtain a $*$-isomorphism $\theta_1: \langle \M, e_{\D\rtimes \ker(\varepsilon)}  \rangle \to (\D\bar\otimes \ell^{\infty}(\Gamma/\ker(\varepsilon)))\rtimes \Gamma$ which satisfies $\theta_1(au_g)=(a\otimes 1)u_g$, for all $a\in \D,g\in\Gamma$ and $\theta_1(e_{\D\rtimes \ker(\varepsilon)})=1\otimes \delta_{e \ker(\varepsilon)}$.  By applying once again Lemma \ref{lemma.basic.construction}, it follows that there is a $*$-isomorphism $\theta_2: \langle \M\bar\otimes\N, e_{\M\otimes 1}   \rangle  \to (\M\bar\otimes \ell^{\infty}(\Lambda))\rtimes \Lambda$ which satisfies $\theta_2(au_g\otimes v_\lambda)=(au_g\otimes 1)v_\lambda$, for all $a\in \D,g\in\Gamma,\lambda\in\Lambda$ and $\theta_2(e_{\M\otimes 1})=1\otimes\delta_e$. Here, $(\M\bar\otimes \ell^{\infty}(\Lambda))\rtimes \Lambda$ arises from the diagonal action of $\Lambda$ on $\M\bar\otimes \ell^{\infty}(\Lambda)$, where $\Lambda$ acts trivially on $\M$ and by left translations on $\ell^\infty(\Lambda)$. 

Next, since the epimorphism $\varepsilon:\Gamma\to\Lambda$ naturally define a group isomorphism between $\Gamma/\ker(\varepsilon)$ and $\Lambda$, there is a $*$-isomorphism between $\ell^{\infty}(\Gamma/\ker(\varepsilon))$ and $\ell^\infty(\Lambda)$ which sends $\delta_{e {\rm ker}(\epsilon)}$ to $\delta_e$. Therefore, we can define an injective $*$-homomorphism $\varphi: (\D\bar\otimes \ell^{\infty}(\Gamma/\ker(\varepsilon)))\rtimes \Gamma\to (\M\bar\otimes \ell^{\infty}(\Lambda))\rtimes \Lambda$ by letting $\varphi((a\otimes f)u_g)=(au_g\otimes f)v_{\varepsilon(g)}$, for all $a\in D,g\in \Gamma$ and $f\in \ell^{\infty}(\Gamma/\ker(\varepsilon))$. Altogether, it implies that we can define an injective $*$-homomorphism $\theta: \langle \M, e_{\D\rtimes \ker(\varepsilon)}  \rangle  \to \langle \M\bar\otimes\N, e_{\M\otimes 1}   \rangle  $
by letting $\theta=\theta_2^{-1}\circ \varphi\circ\theta_1$. In this way, 
$\theta(au_g)=\theta_2^{-1} (\varphi ((a\otimes 1)u_g))=\theta_2^{-1}((au_g\otimes 1)v_{\varepsilon(g)})=\Delta(au_g)$, for all $a\in\D$ and $g\in\Gamma$. Note that we also have $\theta(e_{\D\rtimes \ker(\varepsilon)})=\theta_2^{-1}(\varphi(1\otimes\delta_{e \ker(\varepsilon)}))=\theta_2^{-1}(1\otimes \delta_e)=
e_{\M\otimes 1}$, hence proving the above claim.


Finally, denote $\tilde p=\Delta(p)$. The assumption implies that there exists a $\Delta(\P)$-central positive linear functional $\Phi:  \tilde p\langle \M\bar\otimes\N, e_{\M\otimes 1}   \rangle  \tilde p\to \mathbb C$ such that the restriction of $\Phi$ to $\tilde p (\M\bar\otimes \N) \tilde p$ equals the trace on $\tilde p (\M\bar\otimes \N) \tilde p$.
Define now the positive linear functional $\Psi:  p\langle \M, e_{\D\rtimes \ker(\varepsilon)}  \rangle p \to \mathbb C$ by $\Psi(x)=(\Phi\circ\theta)(x)$, for all $x\in p\langle \M, e_{\D\rtimes \ker(\varepsilon)}  \rangle p$ and note that the restriction of $\Psi$ to $p\M p$ equals the trace on $p\M p$ and $\Psi$ is $\P$-central. This shows that $\P$ is amenable relative to $\D\rtimes \ker(\varepsilon)$ inside $\M$.
\end{proof}

\subsection{Quasinormalizers of von Neumann algebras}

\noindent Given a group inclusion $\Sigma<\Gamma$, the one-sided quasi-normalizer ${\rm QN}^{(1)}_\Gamma(\Sigma)$ is the semigroup of all $g\in \Gamma$ for which there exists a finite set $F\subset \Gamma$ such that $\Sigma g\subset F \Sigma$ \cite[Section 5]{FGS10}; equivalently, $g\in {\rm QN}^{(1)}_\Gamma(\Sigma)$ if and only if $[\Sigma: g \Sigma g^{-1}\cap \Sigma]<\infty$. The quasi-normalizer ${\rm QN}_\Gamma(\Sigma)$ is the group of all $g\in \Gamma$ for which exists a finite set $F\subset \Gamma$ such that $\Sigma g\subset F\Sigma$ and $g \Sigma\subset \Sigma F$.

\noindent Given an inclusion $\P \subseteq \M$ of finite von Neumann algebra we define the quasi-normalizer  $\mathscr {QN}_{\M}(\P)$ as the set of all elements $x\in \M$ for which there exist $x_1,...,x_n\in \M$ such that $\P x\subseteq \sum x_i \P$ and $x\P \subseteq \sum \P x_i$ (see \cite[Definition 4.8]{Po99}). Also the one-sided quasi-normalizer $\mathscr {QN}^{(1)}_{\M}(\P)$ is defined as the set of all elements $x\in \M$ for which there exist $x_1,...,x_n\in \M$ such that $\P x\subseteq \sum x_i \P$ \cite{FGS10}.
We continue by recalling from \cite{Po03,FGS10} some basic properties of (one-sided) quasi-normalizing algebras.

\begin{lem}[\cite{Po03,FGS10}]\label{QN1}
Let $\P\subset \M$ be tracial von Neumann algebras. For any projection $p\in\P$, the following hold:
\begin{enumerate}
    \item $W^*({\mathscr{QN}^{(1)}_{p\M p}}(p\P p))=pW^*({\rm \mathscr{QN}^{(1)}_{\M}}(\P))p$. 
   
    \item $W^*({ \mathscr{QN}_{p\M p}}(p\P p))=pW^*({\rm \mathscr{QN}_{M}}(\P))p$. 
\end{enumerate}

\end{lem}

We continue by recording the following remark which can be deduced directly from the definition.

\begin{rem}\label{QN.remark}
    Let $\P\subset \M$ be tracial von Neumann algebras. For any projection $p\in\P'\cap\M$, we have $W^*({\rm \mathscr{QN}_{p'\M p'}}(\P p'))=p'W^*({\rm \mathscr{QN}_{\M}}(\P))p'$.
\end{rem}

\noindent Finally, we record the following result which provides a relation between the group theoretical quasinormalizer and the von Neumann algebraic one.

\begin{lem}[Corollary 5.2 in \cite{FGS10}]\label{QN2}
Let $\Sigma<\Gamma$ be countable groups. Then the following hold:
\begin{enumerate}
    \item $W^*(\mathscr{QN}^{(1)}_{\L(\Gamma)}(\L(\Sigma)))=\L(K)$,  where $K<\Gamma$ is the subgroup generated by ${\rm QN}^{(1)}_\Gamma(\Sigma)$. In particular, if ${\rm QN}^{(1)}_\Gamma(\Sigma)=\Sigma$, then $\mathscr{QN}^{(1)}_{\L(\Gamma)}(\L(\Sigma))=\L(\Sigma)$.
    
    \item $W^*(\mathscr{QN}_{\L(\Gamma)}(\L(\Sigma)))=\L({\rm QN}_\Gamma(\Sigma))$.
\end{enumerate}

\end{lem}

For further use, we briefly recall some results from \cite{CDD22} that provide control of quasinormalizers of certain von Neumann subalgebras in von Neumann algebras of general graph product groups.

\begin{thm}\emph{\cite[Theorem 2.7]{CDD22}}\label{controlquasinormalizer1}
Let $\Gamma=\sG \{\Gamma_v\}$ be any graph product of groups  and let $\sS,\sT\subseteq \sG$ be any subgraphs. Denote by $\M= \L(\Gamma)$ and assume there exist $x,x_1,x_2,...,x_n \in \M$ such that $\L(\Gamma_\sS)x\subseteq \sum^n_{k=1} x_k \L(\Gamma_\sT)$. Then $\sS\subseteq \sT$ and $x\in \L(\Gamma_{\sT \cup {\rm lk}(\sS) })$.     
\end{thm}


\begin{cor}\label{controlquasinormalizer2}Let $\Gamma=\sG \{\Gamma_v\}$ be any graph product of groups  and let $\sC\in {\rm cliq}( \sG)$ be a clique with at least two vertices. Fix a vertex $v\in \sC$ such that ${\rm lk}(\sC \setminus \{v\})=\{v\}$.  Denote by $\M= \L(\Gamma)$ and assume there exist $x,x_1,x_2,...,x_n \in \M$ such that $\L(\Gamma_{\sC \setminus \{v\}})x\subseteq \sum^n_{k=1} x_k \L(\Gamma_\sC)$. Then  $x\in \L(\Gamma_{\sC })$. 

\begin{proof} Follows applying Theorem \ref{controlquasinormalizer1} for $\sS =\sC \setminus \{v\} $ and $\sT =\sC$.\end{proof}

\end{cor}

\begin{lem}\emph{\cite[Lemma 2.9]{CDD22}}\label{lemma.control.qn}\label{controlonesidednorm3}
Let $\Gamma=\sG \{\Gamma_v\}$ be a graph product of groups and let $\sC\in {\rm cliq}(\sG)$ be a clique. Let $\P\subset p\L(\Gamma_{\sC})p$ be a von Neumann subalgebra such that $\P\nprec_{\L(\Gamma_\sC)} L(\Gamma_{\sC \setminus \{v\}})$, for any $v\in \sC$. If $x\in \L(\Gamma)$ satisfies $x\P\subset \sum_{i=1}^n \L(\Gamma_{\sC})x_i$ for some $x_1,\dots,x_n\in \L(\Gamma)$, then $xp\in \L(\Gamma_\sC)$.
\end{lem}

\section{A version of Popa-Vaes dychotomy for normalizers of subalgebras in crossed-product von Neumann algebras}

In \cite[Theorem 1.4]{PV12} Popa and Vaes established a remarkably deep dichotomy for normalizers of von Neumann subalgebras in crossed-product von Neumann algebras, $\A\rtimes \Gamma$, arsing from trace preserving actions $\Gamma \ca \D$ of biexact weakly amenable groups $\Gamma$ on tracial von Neumann algebras $\D$. For further use,  we present next a version of this result for actions of groups which surject onto biexact, weakly amenable groups.  

\begin{thm} Let $\Sigma$ be a biexact weakly amenable group and let $\Gamma$ be a group which admits an epimorphism $\varepsilon: \Gamma \ra \Sigma$.  Let $\D$ be a tracial von Neumann algebra and let $\Gamma \ca ^{\sigma} \D$ be a trace-preserving action. Denote by $\M =\D \rtimes_{\sigma} \Gamma$ the corresponding crossed-product von Neumann algebra and let $0\neq p\in \M$ be a projection. Then for any von Neumann subalgebra $\P \subset p\M p$ that is amenable relative to $\D$ inside $\M$ one of the following must hold:

\begin{enumerate}
    \item $\P \prec_\M \D\rtimes \ker(\varepsilon)$, or
    \item  $\mathscr N_{p\M p}(\P)''$ is amenable relative to $\D\rtimes \ker(\varepsilon) $ inside $\M$.
\end{enumerate}
\end{thm}
\begin{proof} Denote  $\Q=\mathscr N_{p\M p}(\P)''$. Following \cite[Section 3]{CIK13}, consider the $\ast$-embedding 
$\Delta : \M \ra \M \bar\otimes \L(\Sigma)$ given by $\Delta(au_g)=(au_g) \otimes v_{\varepsilon(g)}$ for all $a\in \D$ and $g\in \Gamma$, where $(v_g)_{g\in \Gamma}$ are the canonical group unitaries in $\L(\Sigma)$.  Then using \cite[Theorem 1.4]{PV12} one of the following must hold
\begin{enumerate}
    \item [a)] $\Delta(\P)\prec_{\M\bar\otimes \L(\Sigma)} \M\otimes 1$, or 
    \item [b)] $\Delta(\Q)$ is amenable relative to $\M \otimes 1$ inside $\M\bar \otimes \L(\Sigma)$.
\end{enumerate}
Using \cite[Proposition 3.4]{CIK13} one can see that case a) implies 1. Moreover, using Proposition \ref{reductionamenable}, case b) yields 2.  \end{proof}

\section{Solidity results for von Neumann algebras of wreath-like product groups}

We start this section by recording the following result which follows from Popa and Vaes' structure theorem \cite[Theorem 1.4]{PV12} for normalizers inside crossed products of hyperbolic groups as in \cite{CIK13}.

\begin{prop}\label{prop.dichotomy.wreath.like}
Let $\Ga \in \mathcal W\mathcal R (A, B \ca I)$ where $A$ is abelian and $B$ is a subgroup of a hyperbolic group. Let $\Ga\curvearrowright \D$ be a trace preserving action and denote $\M=\D\rtimes\Ga.$ Let $\A,\B\subset p \M p$ be commuting von Neumann subalgebras.

Then either $\A\prec_{\M} \D\rtimes A^{I}$ or $\B$ is amenable relative to $\D$ inside $\M$. 
\end{prop}

\begin{proof}
Let $\varepsilon:\Ga\to B$ be the quotient homomorphism given by $\Ga \in \mathcal W\mathcal R (A, B \ca I)$ and let $\{u_g\}_{g\in\Ga}$ and $\{v_h\}_{h\in B}$ be the canonical generating unitaries of $\L(\Ga)$ and $\L(B)$, respectively. Denote by $\Delta:\M\to \M\bar\otimes \L(B)$ the $*$-homorphism given by $\Delta(au_g)=au_g\otimes v_{\varepsilon(g)}$, for all $a\in \D,g\in\Ga$.
From \cite[Lemma 5.2]{KV15}, we derive that either $\Delta (\A)\prec_{\M\bar\otimes \L(B)} \M\otimes 1$ or $\Delta(\B)$ is amenable relative to $\M\bar\otimes 1$. The first possibility implies by \cite[Proposition 3.4]{CIK13} that $\A\prec_{\M} \D\rtimes A^{I}$, while the second one implies by Proposition \ref{reductionamenable} that $\B$ is amenable relative to $\D\rtimes A^{I}$ inside $\M$. Since $A$ is amenable, the conclusion follows from \cite[Proposition 2.4]{OP07}.
\end{proof}


We recall that a II$_1$ factor $\M$ is called {\it solid} if for any projection $p\in \M$ and any diffuse  von Neumann subalgebra  $\A \subset p \M p$ the relative commutant $\A'\cap p\M p$ is amenable. We note that if $\M$ is solid, then  any amplification $\M^t$, where $t>0$, is also solid.

\begin{thm}\label{prodrigid1}Let $\Ga\in \mathcal W\mathcal R (A, B \ca I)$ where $A$ is abelian, $B$ is a subgroup of a hyperbolic groups and the action $B\ca I$ has amenable stabilizers. 
Then $\L(\Ga)$ is solid.
\end{thm}

\begin{proof} 
Let $\A \subset p\L(\Ga)p$ be a diffuse  von Neumann subalgebra and
assume by contradiction that $\B:=\A'\cap p\L(\Ga)p$ is non-amenable. Thus, there exists a non-zero projection $z\in \Z(\B)$ such that $ q \B q$ is nonamenable for every projection $q \in \B z$. By applying Proposition \ref{prop.dichotomy.wreath.like} we deduce that $\A z\prec^s_{\L(\Ga)} \L(A^{I})$, and hence,  by using \cite[Corollary 4.7]{CIOS21} (for $m=1$) and also \cite[Lemma 4.11]{CIOS21} we get that $\B z$ is amenable, contradiction.
\end{proof}

\begin{rem} Notice that this result generalizes some of the results obtained in \cite[Section 9.2]{CIOS21}.

\end{rem}




\begin{prop}\label{lemma.counting}\label{prop.counting}
 For any $1\leq i\leq n$, we consider $\Ga_i\in \mathcal W\mathcal R (A_i, B_i \ca I_i)$ where $A_i$ is abelian, $B_i$ is a  subgroup of a hyperbolic group and denote $\Ga= \Ga_1\times \cdots \times \Ga_n$. Let $\Lambda$ be a countable group and denote $\M=\L(\Lambda)$. Assume that $\L(\Gamma)\subset \M$ is a von Neumann subalgebra.

Let $\P_1,\dots,\P_{m}\subset p\M p$ be commuting non-amenable factors and $\Lambda_0<\Lambda$ a subgroup such that $\P_1\vee\dots\vee \P_{m}\prec_{\M} \L(\Lambda_0)$ and $\L(\Lambda_0)\prec_{\M} \L(\Gamma_1\times \dots \times \Gamma_n)$. Then $m\leq n$. 

\end{prop}

Throughout the proof, we use the following notation. If $\P\subset p\M p, \Q\subset q \M q$ are subalgebras such that $\P\prec_{\M} \Q q'$, for all projections $0\neq q'\in \Q'\cap q\M q$,  we write $\P\prec_{\M}^{s'} \Q$. 

\begin{proof}
Denote $\Gamma=\Gamma_1\times\dots\times \Gamma_n$ and $\P=\P_1\vee\dots\vee \P_{m}$.
Following the augmentation technique from \cite[Section 3]{CD-AD20}, we consider a Bernoulli action with abelian base $\Lambda\curvearrowright \D^\Lambda$ and let $\tilde \M=\D^\Lambda\rtimes \Lambda$. Let $\Delta:\tilde\M\to \tilde\M\bar\otimes \M$ be the $*$-homomorphism given by $\Delta(dv_h)=dv_h\otimes v_h$, for all $d\in \D^\Lambda, h\in \Lambda$. The assumption implies that $\P\prec^{s'}_{\tilde \M} \D^\Lambda\rtimes \Lambda_0$, and hence, by applying \cite[Lemma 2.3]{Dr19b} we get that $\Delta(\P)\prec_{\tilde \M\bar\otimes \M}^{s'} \tilde\M \bar\otimes \L(\Lambda_0)$. Next, from \cite[Lemma 2.4(2)]{Dr19b} we derive that $\Delta(\P)\prec_{\tilde \M\bar\otimes \M} \tilde\M \bar\otimes \L(\Gamma)$. 

Thus, there exist projections $r\in \P, q\in \tilde\M \bar\otimes \L(\Gamma)$, a non-zero partial isometry $v\in q(\tilde\M\bar\otimes \M)\Delta(r)$ and a $*$-homomorphism $\theta: \Delta(r\P r)\to q (\tilde\M \bar\otimes \L(\Gamma)) q$ such that $\theta(x)v=vx$, for any $x\in \Delta(rPr)$, and the support projection of $E_{\tilde\M \bar\otimes \L(\Gamma)}(vv^*)$ equals $q$. Using \cite[Lemma 4.5]{CdSS17}, we can assume that $r\in \P_1$. Denote $\N=\tilde\M \bar\otimes \L(\Gamma)$ and $\Q_i=\Delta(r\P_i r)$ for any $1\leq i\leq m$.
From \cite[Lemma 10.2]{IPV10} it follows that $\Delta(\P_i)$ is strongly non-amenable relative to $\tilde\M\otimes 1$ inside $\tilde\M\bar\otimes \M$, for any $1\leq i\leq m$. Therefore, $\Q_i$ is non-amenable relative to $\tilde\M\otimes 1$ inside $\N$ for any $1\leq i\leq m$. In particular, we derive from \cite[Proposition 2.7]{PV11} that there exists $1\leq j\leq n$ such that  $\Q_m$ is non-amenable relative to $\tilde\M\otimes \L(\Gamma_{\widehat {j}})$ inside $\N$. Without loss of generality, we may assume that $j=n$. By applying  Proposition \ref{prop.dichotomy.wreath.like}, it follows that $\bigvee_{i=1}^{m-1} Q_i\prec_{\N} \tilde\M\otimes \L(\Gamma_{\widehat {n}}\times A_{n}),$ where $A_n<\Gamma_{n}$ is an amenable subgroup. Thus,
$
\Delta(\bigvee_{i=1}^{m-1} \P_i)\prec_{\tilde\M\bar\otimes \M} \tilde\M\bar\otimes \L(\Gamma_{\widehat {n}}\times A_{n}).
$

Assume by contradiction that $m>n$. By repeating the previous argument finitely many times, we obtain that there exists an amenable subgroup $A<\Gamma$ with the property that $\Delta(\bigvee_{i=1}^{m-n} \P_i)\prec_{\tilde\M\bar\otimes \M}\tilde\M\bar\otimes \L(A)$. By applying \cite[Lemma 10.2]{IPV10} it follows that $\P_1$ is amenable, contradiction. 
\end{proof}

\begin{prop}\label{prop.counting2}
For any $1\leq i\leq n$, we consider $\Ga_i\in \mathcal W\mathcal R (A_i, B_i \ca I_i)$ where $A_i$ is abelian, $B_i$ is a  subgroup of a hyperbolic group and denote $\Ga= \Ga_1\times \cdots \times \Ga_n$. 
Let $\P_0,\P_1,..., \P_n\subseteq p \L(\Ga) p$ be commuting von Neumann subalgebras such that  $\P_n$ has no amenable direct summand and $\P_i$ is a non-amenable factor for any $1\leq i\leq n-1$.

Then $\P_0$ is completely atomic.
\end{prop}

\begin{proof} Throughout the proof we denote by $\Gamma_S=\times_{i\in S}\Gamma_i$ the subproduct supported on a subset $S\subset \{1,\dots, n\}$. We also denote by $\widehat S$ the complement of $S$ inside $\{1,\dots, n\}$.
Denote $\M=\L(\Ga)$. Since $\P_n$ has no amenable direct summand, it follows that $\P_n z$ is non-amenable for any non-zero projection $z\in (\bigvee_{i=0}^n \P_i)'\cap pMp$. By \cite[Proposition 2.7]{PV11} we get that there exists $1\leq j\leq n$ such that $\P_n z$ is non-amenable relative to $\L(\Ga_{\widehat j})$ inside $\M$. Without loss of generalization, assume $j=n$. Using Proposition \ref{prop.dichotomy.wreath.like} we get that $(\bigvee_{i=0}^{n-1} \P_i)z\prec_{\M} \L(\Ga_{\widehat n})\bar\otimes \L(A_n^{(I_n)})$. Thus, there exist projections $r\in \bigvee_{i=0}^{n-1} \P_i, q\in \L(\Ga_{\widehat n})\bar\otimes \L(A_n^{(I_n)})$ with $rz\neq 0$, a non-zero partial isometry $v\in q\M r$ and a $*$-homomorphism $\theta: r(\bigvee_{i=0}^{n-1} \P_i)rz \to q(\L(\Ga_{\widehat n})\bar\otimes \L(A_n^{(I_n)}))q$ such that $\theta(x)v=vx$ and the support projection of $E_{\L(\Ga_{\widehat n})\bar\otimes \L(A_n^{(I_n)})}(vv^*)$ equals $q$. Since $\P_{n-1}$ is a II$_1$ factor, by using \cite[Lemma 4.5]{CdSS17} we can assume that $r\in \P_{n-1}$. Denote $\N=q(\L(\Ga_{\widehat n})\bar\otimes \L(A_n^{(I_n)}))q$ and $\Q_i=\theta(r\P_irz)$ for any $0\leq i\leq n-1$. Since $\Q_{n-1}$ is non-amenable, we use \cite[Proposition 2.7]{PV11} and derive that exists $1\leq k\leq n-1$ such that $\Q_{n-1}$ is non-amenable relative to $\L( \Ga_{\widehat{k,n}})\bar\otimes \L(A_k^{(I_k)}\times A_n^{(I_n)})$ inside $\N$. Without loss of generalization, we assume $k=n-1$. By applying Proposition \ref{prop.dichotomy.wreath.like} we derive that 
$
\bigvee_{i=0}^{n-2} \Q_i\prec_{\N} \L( \Ga_{\widehat{n-1,n}})\bar\otimes \L(A_{n-1}^{(I_{n-1})}\times A_n^{(I_n)}).
$
Using that $\Q_i v=v r\P_i rz$, we derive that $(\bigvee_{i=0}^{n-2} \P_i)z\prec_{\M} \L(\Ga_{\widehat {n-1,n}})\bar\otimes \L(A_{n-1}^{(I_{n-1})}\times A_n^{(I_n)})$. By using the previous argument finitely many times, we obtain $\P_0z\prec_{\M}\L(\times_{i=1}^n A_i^{(I_i)})$.

Assume by contradiction that $\P_0z\nprec_{\M} \mathbb C1$. The previous paragraph implies that there exists a non-empty subset $F\subset \{1,\dots,n\}$ such that
\begin{equation}\label{v1}
 \P_0z\prec_{\M}\L(\times_{i\in F} A_i^{(I_i)}) \text{ and }  \P_0z\nprec_{\M}\L(\times_{i\in F\setminus\{j\}} A_i^{(I_i)}) \text{ for any }j\in F.
\end{equation}
Therefore, there exist projections $p_0\in P_0, s\in \L(\times_{i\in F} A_i^{(I_i)})$ with $p_0z\neq 0$,  a non-zero partial isometry $w\in s\M p_0$, a $*$-homomorphism $\Psi: p_0P_0p_0z\to s\L(\times_{i\in F} A_i^{(I_i)})s$ such that $\Psi(x)w=wx$, for any $x\in p_0P_0p_0z$, and the support projection of $E_{\L(\times_{i\in F} A_i^{(I_i)})}(ww^*)$ equals $s$. By letting $T=\Psi(p_0P_0z)$ and $A=\times_{i\in F} A_i^{(I_i)}$, one can check that \eqref{v1} implies that 
$$
T\subset s\L(A)s \text{ and }  T\nprec_{\L(A)}\L(\times_{i\in F\setminus\{j\}} A_i^{(I_i)}) \text{ for any }j\in F.
$$
By applying \cite[Corollary 4.7 and Lemma 4.11]{CIOS21}, we have $T'\cap s \L(\Ga_{F}) s$ is amenable. Note that
the $*$-isomorphism ${\rm Ad}(w): w^*w \M w^*w\to ww^* \M ww^*$ sends $w^*w (p_0P_0p_0z'\cap p_0z\M p_0z) w^*w$ onto $ww^* (T'\cap s\M s) ww^*$. Hence,
\begin{equation*}\label{v2}
    w ( p_0P_0p_0z'\cap p_0z\M p_0z  ) w^*=ww^* ((T'\cap s \L(\Ga_{F}) s)\bar\otimes \L(\Ga_{\widehat F})s ) ww^*. 
\end{equation*}
Since $\bigvee _{i=1}^n \P_i\subset \P_0'\cap p\M p$, 
it follows that $\bigvee_{i=1}^n \P_i\prec_{\M} (T'\cap s \L(\Ga_{F}) s)\bar\otimes \L(\Ga_{\widehat F})s $. Since $\widehat F$ has at most $n-1$ elements and $T'\cap s \L(\Ga_{F}) s$ is amenable, we can apply the previous arguments finitely many times and derive that there exists an amenable subalgebra $S\subset \M $ such that $\P_n\prec_{\M} S$. This shows that $\P_n$ has an amenable direct summand, contradiction.

Finally, we conclude that $\P_0z\prec_{\M} \mathbb C1$, for all non-zero projections $z\in (\bigvee_{i=0}^n \P_i)'\cap p\M p$. Using \cite[Lemma 2.4]{DHI16}, we conclude that  $\P_0\prec^s_{\M} \mathbb C1$, which implies that $\P_0$ is completely atomic.
\end{proof}

\section{Product rigidity results for von Neumann algebras of wreath-like product groups}


The goal of this section is to prove Theorem \ref{A} and consequently derive Corollary \ref{superprod'}.
We first prove the following proposition.

\begin{prop}[\cite{CU18,CdSS15}]\label{almostprod} 
Let $\Gamma$ be an icc group, denote $\M=L(\Gamma)$ and assume that $\M=\A \bar\otimes \B$ where $\B$ is solid. 
Let $ \Omega<\Gamma$ be a subgroup with nonamenable centralizer $C_\La(\Omega)$ such that $\A\prec \L(\Omega)$.

Then there exist projections $e\in \L(\Omega)$ and $r \in \L(\Omega)'\cap \M$ with $e r\neq 0$ such that $ e(\L(\Omega)e r \vee r(\L(\Omega)'\cap \M)r e\subseteq er\M er $ is a finite index inclusion of II$_1$ factors.

\end{prop}

\begin{proof} Let $\Q=\L(\Omega)$  As $\A\prec\Q$ then using \cite[Proposition 2.4]{CKP14} and its proof, one can find nonzero projections $a \in \A$, $q \in \Q$, a partial
isometry $v\in q \M a$, a subalgebra $\D \subseteq  q\Q q$, and a $\ast$-isomorphism $\phi : a\A a \ra  \D $ such
that
\begin{enumerate}
    \item $\D \vee (\D'\cap q\Q q) \subseteq  q\Q q$ has finite index, and 
    
    \item $\phi(x)v = vx$ for all  $x \in a\A a$. 
\end{enumerate}

The intertwining relation 2.\ implies that  $vv^* \in \D' \cap  q\M
q$ and $v^*v \in  (a\A a)' \cap  a\M a = a \otimes \B.$  Hence there is a projection
$b \in  \B$ satisfying $v
^*v = a \otimes b$. Choosing $u \in \mathscr U(\M
)$ such that $v = u(a \otimes b) $ then relation 2.\ entails that  
\begin{equation}\label{equalcorners1}\D vv^* = v\A v^* = u(a\A a \otimes b)u^*.\end{equation}

Using this relation and passing to the relative commutants, we obtain $vv^* 
(\D' \cap  q\M
q)vv^* = u(a \otimes b\B b)u
^*$. 
Thus one can find scalars $s_1,s_2 > 0$ satisfying 
\begin{equation}\label{commutant1}(\D' \cap  q\M
q)z = u(a \otimes b\B b)
^{s_1} u
^* \cong \B 
^{s_2}.\end{equation} 
Here, $z$ denotes the central support projection of $vv^*$
in $\D'\cap q\M q$. Now notice that 
\begin{equation}\label{commutant2} 
\D' \cap q \M
q \supseteq (q\Q q)'
\cap  q\M
q = (\Q' \cap  \M
)q \supseteq  \L(C_\La(\Omega))q,\end{equation}
where $\L(C_\La(\Omega))$ has no amenable direct summand since $C_\La(\Omega)$ is a non-amenable
group. Moreover, we also have $\D' \cap q\M
q \supseteq  \D' \cap q \Q q$. Thus $(\Q
' \cap \M
)z$ and
$(\D' \cap  q\Q q)z$ are commuting subalgebras of $(\D' 
\cap q \M
q)z$ where $(\Q' \cap \M
)z$ has
no amenable direct summand. Since $\B^{s_2}$ is solid then $(\D' \cap  q\Q q)z$ must be purely atomic. Thus, cutting by a
central projection $0\neq r
\in \D' \cap  q\Q q$,  we may assume that $\D \subseteq  q\Q q$ is a
finite index inclusion of algebras. Since $\D$ is a factor, shrinking $r$ if necessary, we can actually assume that $\D \subseteq q\Q q$ is an irreducible inclusion of finite index II$_1$ factors. Moreover, one can check
that if one replaces $v$ by the partial isometry of the polar decomposition of $r
v \neq  0$ then
all relations \eqref{equalcorners1}, \eqref{commutant1} and  \eqref{commutant2} are still satisfied. In addition, we can assume without any
loss of generality that the support projection satisfies $s(E_{\Q}(vv^*)) = q$. 

Using \cite[Lemma 3.1]{Jo81} one can find a projection $e\in q\Q
q$ and a subfactor $\R \subseteq \D \subseteq q \Q q$ such that $e \Q e =\R e$ and the index $[\D:\R]=[q \Q q: \D]$. Now notice the restriction $\phi^{-1}: \R \ra a\A a$ is an injective $\ast$-homorphism such that $\T:= \phi^{-1}(\R)\subseteq a\A a$ is finite index and \begin{equation}\label{invintertwining}
    \phi^{-1}(y)v^*=v^*y \text{ for all } y\in \R.
\end{equation}   

Let $\phi': \R e \ra \R$ be the $\ast$-homorphism given by  $\phi'(y e)=y$, for any $y\in \R$. Since $e$ has full support in $\langle \D, e\rangle =q\Q q$ we have $ev\neq 0$. Letting $w_0$ be a partial isometry such that $w_0 |v^*e |=v^*e$ the \eqref{invintertwining} gives that $\theta= \phi^{-1} \circ \phi': e \Q e \ra a\A a$ is an injective $\ast$-isomorphism satsfying $\theta(e \Q e)= \T$ and 
\begin{equation}\label{intertwiningcomp}
    \theta(y) w_0^*=w_0^* y \text{ for all } y\in e \Q e. 
\end{equation}

Notice that $w_0^*w_0\in \T'\cap (a \A a \oo \B)=(\T'\cap a \A a) \oo \B $. Since we have $\T'\cap (a\A a\oo \B )\supseteq\mathscr Z(\T'\cap a \A a) \oo \B$ and these von Neumann algebras have the same center then proceeding as in the proof of \cite[Proposition 12]{OP03} one can see that $w_0^*w_0$ is equivalent in 
$\T'\cap (a \A a \oo \B)$ with projection in $\mathscr Z(\T'\cap a \A a) \oo \B$. Thus one can assume without any loss of generality that $w_0^*w_0\in \mathscr Z(\T'\cap a \A a) \oo \B$.  As $[a\A a:\T]<\infty$ then $\T'\cap a\A a$ is finite dimensional. Thus replacing $w_0$ by $w=w_0r_0$ for a minimal projection $r_0\in \mathscr Z(\T'\cap a \A a)$ with $r_0 w_0^*|v^*e|\neq 0$ we see all previous relations are satisfied including \eqref{intertwiningcomp}. Moreover we can assume that $w^*w= z_1\otimes z_2$ for some $z_1\in \mathscr Z(\T'\cap a \A a)$ and $z_2\in B$. Hence \eqref{intertwiningcomp} implies 
\begin{equation}\label{ii1}
w^* \Q w= \theta(e \Q e) w^*w= \T z_1 \otimes z_2.   
\end{equation}  

Since $\T \subseteq a \A a$ is a finite index inclusion of II$_1$ factors, using the local index formula \cite[Lemma 2.2.1]{Jo81} we have that $\T z_1\subseteq z_1 \A z_1$ is a finite index inclusion of II$_1$ factors as well. In addition, we have

\begin{equation}\label{ii2}
    (w^* \Q w)'\cap (z_1\otimes z_2) \M (z_1\otimes z_2)= ((\T z_1)' \cap z_1 \A z_1) \oo z_2 \B z_2.  
\end{equation}
Altogether, the previous relations \eqref{ii1} and \eqref{ii2} imply that \begin{equation}\label{finiteindex4}
    \begin{split}
        \T z_1 \oo z_2 \B z_2& \subset \T z_1 \vee (\T z_1'\cap z_1 \A a)\oo z_2 \B z_2\\
        &=  w^* \Q w\vee (w^* \Q w)'\cap (z_1\otimes z_2) \M (z_1\otimes z_2)\\
        &\subseteq z_1 \A z_1\oo z_2 \B z_2  = (z_1\otimes z_2) \M (z_1\otimes z_2)
    \end{split}
\end{equation}
 
 Since $\T z_1\subseteq z_1 \A z_1$ is a finite index inclusion of II$_1$ factors then so is $\T z_1 \oo z_2 \B z_2\subseteq z_1 \A z_1\oo z_2 \B z_2$. Let $f= ww^*$ and note that $f=re$ for some projection $r\in \Q'\cap \M$. Now notice  \eqref{finiteindex4} implies that 
 \begin{equation}
     e\Q e r \vee r (\Q '\cap \M)r e  = e\Q e r \vee ((e\Q er) '\cap er\M er)\subseteq  f \M f. 
     \end{equation}  
is a finite index inclusion of von Neumann algebras. Since $\M$ is a factor, the finite index condition implies that the center $\mathscr Z(r(\Q'\cap \M )r)$ is completely atomic (see \cite[Proposition 2.1(3)]{CdSS17}). Thus, compressing $r$ more if necessary we further obtain that $e\Q e r \vee r (\Q '\cap \M)r e  \subseteq  f \M f$ is a finite index inclusion of II$_1$ factors, as desired.  \end{proof}

In this section we show that various direct products of wreath-like groups give rise to $W^*$-superrigid groups. To show this, we first establish a product rigidity result in the same spirit of \cite{CdSS15}. See also \cite{CD-AD20,Dr20} for more recent similar results.  

\begin{thm}\label{prodrigid1}
 For every $1\leq k\leq n$, let  $\Ga_k\in \mathcal W\mathcal R (A_k, B_k \ca I_k)$ be property (T) groups where $A_k$ is abelian, $B_k$ is an icc  subgroup of  a hyperbolic group and $B_k \curvearrowright I_k$ has amenable stabilizers.
Denote $\Ga= \Ga_1\times  \dots \times \Ga_n$ and assume that $t>0$ is a scalar and $\La$ is an arbitrary group satisfying $\M=\L(\Ga)^t=\L(\La)$. 

Then one can find a product decomposition $\La= \La_1\times\dots\times \La_n$, some scalars $t_1,\dots,t_n>0$ with $t_1\cdots t_n=t$ and a unitary $u\in \mathscr U(\M)$ so that $\L(\Gamma_{i})^{t_i}=u\L(\La_i) u^*$ for any $1\leq i\leq n $.

\end{thm}

\begin{proof}Without any loss of generality we can assume that $t=1$ as the other cases do not hide any difficulties. Throughout the proof we denote by $\Gamma_S=\times_{i\in S}\Gamma_i$ the subproduct supported on a subset $S\subset \{1,\dots, n\}$. We also denote by $\widehat S$ the complement of $S$ inside $\{1,\dots, n\}$.

Let $\Delta: \M\ra \M\oo \M$ be the $\ast$-embedding given by $\Delta(v_h)=v_h\otimes v_h$ for all $h\in \La$. Let $1\leq i,j\leq n $ and observe that $\Delta(\L(\Ga_{\hat i})), \Delta(\L(\Ga_i))\subset \M \oo \L(\Ga_{\hat j})\oo \L(\Ga_j)$ are commuting subalgebras. Using Proposition \ref{prop.dichotomy.wreath.like} (see also \cite[Theorem 6.15]{CIOS21}), we have either $\Delta(\L(\Ga_{\hat i}))\prec_{\M\bar\otimes \M} \M\oo \L(\Ga_{\hat j})\oo L(A^{(I_j)}_j)$
or $\Delta(\L(\Ga_{i}))\prec_{\M\bar\otimes \M} \M\oo \L(\Ga_{\hat j})\oo \L(A^{(I_j)}_j)$. 
Since $\L(A^{(I_j)}_j)$ is abelian and  $\Delta(\L(\Ga_{\hat i}))$, $\Delta(\L(\Ga_{i}))$ have property (T), the prior intertwining relations further imply that either $\Delta(\L(\Ga_{\hat i}))\prec_{\M\bar\otimes \M} \M\oo \L(\Ga_{\hat j})$ or $\Delta(\L(\Ga_{i}))\prec_{\M\bar\otimes \M} \M\oo \L(\Ga_{\hat j})$. Moreover, using \cite[Lemma 2.4(3)]{DHI16}, we have either $\Delta(\L(\Ga_{\hat i}))\prec^s_{\M\oo \M} \M\oo \L(\Ga_{\hat j})$ or $\Delta(\L(\Ga_{i}))\prec^s_{\M\oo \M} \M\oo \L(\Ga_{\hat j})$. If the former would hold for all $1\leq j\leq n$, then by using \cite[Lemma 2.8(2)]{DHI16} we would get $\Delta(\L(\Ga_{i}))\prec_{\M\bar\otimes \M} \bigcap^n_{j=1} (\M\oo \L(\Ga_{\hat j})) =\M\otimes 1$. This contradicts \cite[Lemma 7.2]{IPV10}. Hence, for every $1\leq i\leq n$ there is $1\leq j\leq n$ so that $\Delta(\L(\Ga_{\hat i}))\prec_{\M\bar\otimes \M} \M\oo \L(\Ga_{\hat j})$. Furthermore, using \cite[Theorem 4.3]{Dr20} we actually have that $\Delta(\L(\Ga_{\hat i}))\prec_{\M\bar\otimes \M} \M\oo \L(\Ga_{\hat i})$ for all $1\leq i\leq n$.  Now, since $\L(\Ga_i)$ has property (T), Theorem \ref{Th:ultrapower} shows that there is a subgroup $\Sigma<\La$ such that 
\begin{equation}\label{aa10'}
\L(\Ga_{\hat i})\prec_\M \L(\Sigma)\text{ and }\L(\Ga_{i})\prec_\M \L(C_\La(\Sigma)).
\end{equation}

Since $\L(\Ga_i)$ is solid, \eqref{aa10'} allows to apply
Proposition \ref{almostprod} and find a  projection $0\neq p= qr$ with $q\in \L(\Sigma), r \in \L(\Sigma )'\cap \M$ such that $p(\L(\Sigma)\vee (\L(\Sigma)'\cap \M))p \subseteq p\M p$ is a finite index inclusion of II$_1$ factors. In particular, $p\L(\Sigma {\rm vC}_\La(\Sigma))p \subseteq p\M p$ is also a finite index inclusion of von Neumann algebras, and thus, $[\La: \Sigma {\rm vC}_\La(\Sigma)]<\infty$. Since $\La$ is icc property (T) then so is $\Sigma {\rm vC}_\La(\Sigma)$. Now, observe one can find an increasing sequence of groups $\cdots \leqslant \Omega_n \leqslant \Omega_{n+1}\leqslant \cdots\leqslant {\rm vC}_\La(\Sigma)$ normalized by $\Sigma$ with $\cup_n \Omega_n={\rm vC}_\La(\Sigma)$ whose centralizers form a descending sequence $\Sigma \geqslant C_\Sigma(\Omega_1)\geqslant \cdots \geqslant C_\Sigma(\Omega_n )\geqslant C_\Sigma(\Omega_{n+1}) \cdots $ of finite index subgroups. Therefore,  $\Sigma \Omega_n \nearrow \Sigma {\rm vC}_\La(\Sigma)$ and using property (T) there is $n$ so that $\Sigma \Omega_n = \Sigma {\rm vC}_\La(\Sigma)$.  Since $C_\Sigma(\Omega_n) \Omega_n \leqslant \Sigma \Omega_n$ has finite index, we conclude that $C_\Lambda(\Omega_n) \Omega_n\leq \Lambda$ has finite index as well. Denote $\Sigma_0=C_\Lambda(\Omega_n)$ and note that $C_\Lambda(\Sigma)\leq C_\Lambda(\Sigma_0)$. From \eqref{aa10'}, it follows that
\begin{equation}\label{aa15'}
\L(\Ga_{\hat i})\prec_\M \L(\Sigma_0)\text{ and }\L(\Ga_{i})\prec_\M \L(C_\La(\Sigma_0)).
\end{equation}



By passing to relative commutants in \eqref{aa15'}, we get that $\L(\Sigma_0)\prec_\M \L(\Ga_{\hat i})$ and $\L(C_{\Lambda}(\Sigma_0))\prec_\M \L(\Gamma_i)$. Since $[\La:\Sigma_0 C_\La(\Sigma_0)]<\infty$, we derive that $\L(\Sigma C_\La(\Sigma))'\cap \M=\mathbb C 1$. Therefore, using  \cite[Lemma 2.4(3)]{DHI16} we see that 
\begin{equation}\label{aa11'}
    \L(\Sigma_0)\prec^s_\M \L(\Ga_{\hat i}) \text{ and } \L(C_{\Lambda}(\Sigma_0))\prec^s_\M \L(\Gamma_i).
\end{equation}

By using \eqref{aa15'}, \eqref{aa11'} together with \cite[Theorem 6.1]{DHI16} and a standard inductive argument, we obtain the desired conclusion. \end{proof}

\begin{cor}\label{superprod}  For every $1\leq k\leq n$, let  $\Ga_k\in \mathcal W\mathcal R (A_k, B_k \ca I_k)$ be property (T) groups where $A_k$ is abelian, $B_k$ is an icc  subgroup of  a hyperbolic group, $B_k \curvearrowright I_k$ has amenable stabilizers and the set $\{i\in I \; | \; g\cdot i\neq i\}$ is infinite for any $g\in B_k\setminus \{1\}$. Denote  $\Ga= \Ga_1\times \dots \times \Ga_n$ and assume that $t>0$ is any scalar,  $\La$ is an arbitrary group and $\theta: \L(\Ga)^t\ra \L(\La)$ is any $\ast$-isomorphism.  

Then $t=1$ and one can find a character $\eta :\Gamma\to\mathbb T$, a group isomorphism $\delta:\Gamma\to\Lambda$ and a unitary $u\in \L(\La)$ such that $\theta(u_g)= \eta(g) u v_{\delta(g)} u^*$, for all $g\in \Ga$.     

\end{cor}


\begin{proof}
This follows directly from Theorem \ref{prodrigid1} and \cite[Theorem 9.1]{CIOS21}.
\end{proof}

\section{Reconstruction of clique subgroups under $W^*$-equivalence}

Towards establishing superrigidity results for graph products the first major step is to identify in the mystery group $\Lambda$ collections of subgroups that play the same role as the full subgroups associated to clique subgraphs in the source group $\Gamma$. This will be achieved using the commultiplication map \cite{Io10, PV11} in combination with an ultrapower method from \cite{Io11} (see also \cite{DHI16,CdSS15}) and a technique for reconstructing malnormal group structure developed in \cite{CD-AD20,CI17}. Our result is a new manifestation of a more conceptual principle called \emph{peripheral reconstruction $W^*$-method} which consists of exploiting the natural tension that occurs between ``a peripheral structure'' and a ``direct product structure'' in the group. For completeness we include all the details.

We start by recalling an ultrapower technique which is essentially contained in the proof of \cite[Theorem 3.1]{Io11} (see also \cite[Theorem 3.3]{CdSS15}) and the statement that will be used is a particular case of \cite[Theorem 4.1]{DHI16}.

\begin{thm}[\cite{Io11}]\label{Th:ultrapower}
Let $\Lambda$ be a countable icc group and denote by $\M=\L(\Lambda)$. Let $\Delta:\M\to \M\bar\otimes \M$ be the $*$-homomorphism given by $\Delta (v_\lambda)=v_\lambda\otimes v_\lambda$, for all $\lambda\in\Lambda.$ Let $\P,\Q\subset \M$ be von Neumann subalgebras such that $\Delta(\P)\prec_{\M\bar\otimes \M}\M\bar\otimes \Q$.

Then there exists a decreasing sequence of subgroups $\Sigma_k<\Lambda$ such that $\P\prec_\M \L(\Sigma_k)$, for every $k\ge 1$, and $\Q'\cap \M\prec_\M \L(\cup_{k\ge 1} C_\Lambda(\Sigma_k)).$
\end{thm}

We continue with the following result which contains a consequence of Theorem \ref{Th:ultrapower} and solidity type result to certain von Neumann algebras of graph product groups .
Throughout this article if $\sC\in {\rm cliq} (\sG)$ is a clique of a graph then for every $v\in \sC$ we denote by $\hat v= \sC\setminus\{v\}$.

\begin{thm}\label{commutingsubgroups1} Let $\Gamma= \mathscr G \{\Gamma_v\}$ be a graph product of groups such that $\mathscr G\in {\rm CC}_1$ and denote $\M=\L(\Gamma)$. 
Assume that for any $v\in\mathscr V$, $\Gamma_v\in \mathcal W\mathcal R (A_v, B_v \ca I_v)$ where $A_v$ is abelian, $B_v$ is an icc  subgroup of a hyperbolic group.

Let $\Lambda$ be an arbitrary group such that $\M=\L(\Lambda)$ and let  $\mathscr C \in \rm cliq (\mathscr  G)$ and $v\in \mathscr C$. Then the following hold:

\begin{enumerate}
    \item There is a subgroup $\Lambda_{\hat v} < \Lambda $ such that  $C_\Lambda(\Lambda_{\hat v})$ is non-amenble and $\L(\Gamma_{\hat v})\prec_{\M} L(\Lambda_{\hat v})$.
    
    \item If $\P, \Q, \R\subset p\L(\Gamma_{\mathscr C}) p$ are commuting von Neumann subalgebras such that $\Q$ has no amenable direct summand and $\R$ is isomorphic to a corner of $\L(\Gamma_{\hat v})$, then $\P$ is completely atomic.
\end{enumerate}

\end{thm}

\begin{proof} 1. Following \cite{IPV10,Io10} consider the $\ast$-embedding $\Delta : \M \ra \M \bar\otimes \M$ given by $\Delta(v_h)=v_h\otimes v_h$ for all $h\in \Lambda$.  Fix  $\mathscr C \in \rm cliq (\mathscr  G)$. From assumptions we have  $\Delta(\L(\Gamma_\sC))\subseteq \M \bar\otimes \M = \L(\Gamma\times \Gamma)$ is a property (T) von Neumann algebra. Thus applying Theorem \ref{proptcliques} one can find cliques $\sC_1,\sC_2\in {\rm cliq}(\sG)$ such that $\Delta(\L(\Gamma_\sC))\prec_{\M \bar\otimes \M} \L(\Gamma_{\sC_1}\times \Gamma_{\sC_2})$. This implies that one find projections $q\in \Delta(\L(\Gamma_\sC))$ $p\in \L(\Gamma_{\sC_1}\times \Gamma_{\sC_2})$, partial isometry $w\in \M\bar\otimes \M $ and a $\ast$-isomorphism onto its image $\theta: q\Delta(\L(\Gamma_\sC)) q\ra  \R:= \theta(q\Delta(\L(\Gamma_\sC)) q)\subseteq p \L(\Gamma_{\sC_1}\times \Gamma_{\sC_2}) p$ such that $\theta(x)w=wx$ for all $x\in q \Delta(\L(\Gamma_\sC)) q$. Notice that $ww^*\in \R'\cap q(\M \bar \otimes \M)q $ and $w^*w\in (\Delta(\L(\Gamma_\sC))'\cap \M\bar\otimes \M )q$. Moreover we can assume without loss of generality that the support of $E_{\L(\Gamma_{\sC_1} \times \Gamma_{\sC_2})}(ww^*)$ equals $p$.

Since $\Gamma_t, t\in \sV$, are icc we can assume without loss of generality that $q\in \Delta (\L(\Gamma_{v_0}))$ for some $v_0\in \sC$. For every $\sD \subseteq \sC$, we denote  $\R_\sD= \theta(q\Delta(\L(\Gamma_\sD))q)$. Thus, $\R= \vee_{t\in \sC}  \R_v$ where $\R_t= \theta(q\Delta(\L(\Gamma_t))q)$ are mutually commuting non-amenable II$_1$ factors. 

Now, fix $\emptyset \neq \sD \subseteq \sC$ and notice that $\R_{\sD} \vee \R_{{\rm lk}(\sD)} = \R\subset  \L(\Gamma_{\sC_1}\times \Gamma_{\sC_2})=:\tilde \N$ are commuting nonamenable factors. Since $\mathcal R_\sD$ and $\R_{{\rm lk}(\sD)}$ are commuting property (T) algebras, then \cite[Theorem 6.15]{CIOS21} implies that for every $t \in \sC_2$ either  $\R_{\sD}\prec_{\tilde \N} \L(\Gamma_{\sC_1}\times \Gamma_{\sC_2 \setminus \{t\}}\times A^{(B_t)}_t ) $
or $\R_{{\rm lk}(\sD)}\prec_{\tilde \N} \L(\Gamma_{\sC_1}\times \Gamma_{\sC_2 \setminus \{t\}}\times A^{(B_t)}_t ) $. However, since $R_\sD, \R_{{\rm lk}(\sD)}$ have property (T) and $A_t^{(B_t)}$ is amenable we further conclude that either  $\R_{\sD}\prec_{\tilde \N} \L(\Gamma_{\sC_1}\times \Gamma_{\sC_2 \setminus \{t\}} )$
or $\R_{{\rm lk}(\sD)}\prec_{\tilde \N} \L(\Gamma_{\sC_1}\times \Gamma_{\sC_2 \setminus \{t\}} )$ and using factoriality we actually have either  a) $\R_{\sD}\prec^s \L(\Gamma_{\sC_1}\times \Gamma_{\sC_2 \setminus \{t\}} )$
or b) $\R_{{\rm lk}(\sD)}\prec^s \L(\Gamma_{\sC_1}\times \Gamma_{\sC_2 \setminus \{t\}} )$. Assume by contradiction that b) holds for all $t\in \sC_2$. By \cite[Lemma 2.6]{DHI16} this would imply that $\R_{{\rm lk}(\sD)}\prec \cap_{t\in \sC_2} \L(\Gamma_{\sC_1}\times \Gamma_{\sC_2 \setminus \{t\}} )= \L(\Gamma_{\sC_1})\otimes 1$. Using \cite[Proposition 7.2]{IPV10}, this entails that $\R_{{\rm lk}(\sD)}$ is atomic, a contradiction. Hence, there is $t\in \sC_2$ such that $\R_{\sD}\prec^s \L(\Gamma_{\sC_1}\times \Gamma_{\sC_2 \setminus \{v\}} )$. Combining this with the first part we get again that $\Delta(\L(\Gamma_\sD))\prec \L(\Gamma_{\sC_1}\times \Gamma_{\sC_2 \setminus \{t\}}) )$. The conclusion of the first part follows from Theorem \ref{Th:ultrapower} by letting $\sD=\sC\setminus\{v\}$.

2. The conclusion follows from Proposition \ref{prop.counting2}. 
\end{proof}


Next, by using similar methods to \cite[Theorem 9.1]{CD-AD20},  we identify up to corners the clique subgroups in the mystery subgroup.  

\begin{thm}\label{identificationcornergroups1} Let $\Gamma= \mathscr G \{\Gamma_v\}$ be a graph product of icc groups such that $\mathscr G\in {\rm CC}_1$ and denote $\M=\L(\Gamma)$. Let $\mathscr C \in \rm cliq (\mathscr  G)$ and $v\in \mathscr C$ with ${\rm lk}( \hat{v})=\{v\}$. Assume that whenever $\P, \Q, \R\subset p\L(\Gamma_{\mathscr C}) p$ are commuting von Neumann subalgebras such that $\Q$ has no amenable direct summand and $\R$ is isomorphic to a corner of $\L(\Gamma_{\hat v})$, then $\P$ is completely atomic.

Let $\Lambda$ be an arbitrary group such that $\M=\L(\Lambda)$ and assume that there exists a subgroup $\Lambda_{\hat v}< \Lambda$ with non-amenable centralizer $C_\Lambda(\Lambda_{\hat v})$ such that $\L(\Gamma_{\hat v})\prec_{\M} \L(\Lambda_{\hat v})$.      
 
Then there exist a subgroup $\Lambda_{\hat v} C_\Lambda(\Lambda_{\hat v})  \leqslant \Sigma_\sC<\Lambda$ with  $[\Sigma: \Lambda_{\hat v} {\rm vC}_\Lambda(\Lambda_{\hat v})]<\infty$ and ${\rm QN}^{(1)}_\La(\Sigma_\sC)=\Sigma_\sC$, a non-zero projection $c\in \mathcal Z(\L(\Sigma_\sC))$ and $w_0\in \mathscr U(\M)$ with $w_0c w^*_0=n\in\L(\Gamma_\sC)$  such that $w_0 \L(\Sigma_\sC)c w^*_0= n\L (\Gamma_\sC)n$.
\end{thm}

\begin{proof} Since $\L(\Gamma_{\hat v})\prec_\M  \L(\Lambda_{\hat v})$, one can find projections $a\in \L(\Gamma_{\hat v})$, $f\in \L(\Lambda_{\hat v})$, a non-zero partial isometry $v\in f\M a$ and a $\ast$-isomorphism onto its image $\phi:a \L(\Gamma_{\hat v})a \rightarrow \B:=\phi(a\L(\Gamma_{\hat v})a) \subseteq f\L(\Lambda_{\hat v})f $ such that \begin{equation}\label{inteq1}
    \phi(x)v=vx \text{ for all } x\in a\L(\Gamma_{\hat v})a.
\end{equation}
Notice that $vv^*\in \B'\cap f\M f$ and $v^*v\in a\L(\Gamma_{\hat v})a'\cap a\M a$. Since the virtual centralizer satisfies ${\rm vC}_\Gamma(\Gamma_{\hat v})= C_\Gamma(\Gamma_{\hat v}) = \Gamma_v$, then $\L(\Gamma_{\hat v})'\cap \M = \L(\Gamma_v)$. Thus, we can write  $v^*v=aa_0$ for a projection $a_0\in \L(\Gamma_v)$. Equation \eqref{inteq1}  implies that $\B vv^* = v \L(\Gamma_{\hat v})v^* = u_1 \L(\Gamma_{\hat v}) v^*v u^*_1$, where $u_1\in \M$ is a unitary extending $v$. Taking relative commutants we get $vv^*(\B' \cap f\M f)vv^*= u_1 v^*v( a\L(\Gamma_{\hat v})a'\cap a\M a )v^*v u^*_1= u_1 v^*v\L(\Gamma_{v} ) v^*v u^*_1$. Hence, $vv^*(\B \vee \B'\cap f\M f)vv^*=\B vv^*\vee vv^*(\B' \cap f\M f)vv^*= u_1 v^*v\L(\Gamma_\sC) v^*vu^*_1$. Therefore, since $\L(\Gamma_\sC)$ is a factor one can find a new unitary $u_2 \in \mathscr U(\M)$ such that 
\begin{equation}\label{k1}
(\B \vee \B'\cap f\M f) z_2 \subseteq u_2 \L(\Gamma_\sC) u^*_2,
\end{equation} 
here $z_2$ is the central support of $vv^*$ in $\B \vee \B'\cap f\M f$. In particular, we have $z_2\in \mathcal Z(\B'\cap f\M f )$ and $vv^*\leq z_2 \leq f$.

\noindent Now, observe that \eqref{k1} implies that
\begin{equation}\label{inclusions}
\L(C_\Lambda(\Lambda_{\hat v}))z_2\subseteq (f\L(\Lambda_{\hat v} )f'\cap f \M f )z_2\subseteq (\B'\cap f \M f)z_2\subseteq u_2 \L(\Gamma_{\sC})u_2^*.
\end{equation}  Next, we will show the following containment \begin{equation}\label{controlrelcom}z_2(\L(C_\Lambda(\Lambda_{\hat v}))f'\cap f\M f) z_2\subseteq u_2\L(\Gamma_\sC)u_2^*.\end{equation}

Since we have $\M= u_2\L(\Gamma_{\sV\setminus \{v\}} \ast_{\Gamma_{\hat v}} \Gamma_{\sC} ) u^*_2=u_2\L(\Gamma_{\sV\setminus \{v\}})u_2^* \ast_{u_2\L(\Gamma_{\hat v})u_2^*} u_2\L (\Gamma_{\sC} ) u^*_2$, then by \cite[Theorem 1.2.1]{IPP05} and \eqref{inclusions} to get \eqref{controlrelcom} it suffices to show that \begin{equation}\label{corenonintertwining}\L(C_\Lambda(\Lambda_{\hat v}))z_2\nprec_{u_2\L (\Gamma_{\sC} ) u^*_2} u_2\L(\Gamma_{\hat v})u_2^*.\end{equation}

Since $\Gamma_v$ is icc, it follows that   $vv^*(\B' \cap f\M f)vv^*= u_1 v^*v\L(\Gamma_{v} ) v^*v u^*_1= u_1 a_0(\L(\Gamma_{v} ) )a_0 a u^*_1$ is a factor. Since $z_2$ is the central support of $vv^*$ in $\B' \cap f\M f$, then $(\B' \cap f\M f) z_2$ is a factor as well. 

Assume that \eqref{corenonintertwining} does not hold. Thus one can find projections $t\in \L(C_\Lambda(\Lambda_{\hat v})), r\in u_2 \L(\Gamma_{\hat v})u_2^*$, a partial isometry $w\in ru_2\L(\Gamma_\sC)u_2^* tz_2$ and a $\ast$-isomorphism onto its image $\psi: t\L(C_\Lambda(\Lambda_{\hat v}))tz_2\ra  ru_2\L(\Gamma_{\hat v})u_2^* r$  such that \begin{equation}\label{intertwining2}
    \psi(x)w=wx \text{ for all } x\in t\L(C_\Lambda(\Lambda_{\hat v}))tz_2.
\end{equation}
Notice that $t\L(C_\Lambda(\Lambda_{\hat v}))tz_2\subseteq t(\B'\cap f \M f)tz_2\subseteq u_2 \L(\Gamma_{\sC})u_2^*$ and observe one can pick $t$ small enough so that $tz_2$ is subequivalent to $v^*v$ inside $(\B'\cap f \M f)z_2$. Using this one can find a unitary $u_3\in \M$ such that $t\L(C_\Lambda(\Lambda_{\hat v}))tz_2\subseteq u_3 v^*v \L(\Gamma_v)v^*v u_3^*$. Hence, using this in combination with relation \eqref{intertwining2}, we see that for all unitaries $x\in t\L(C_\Lambda(\Lambda_{\hat v}))tz_2$, we have that $\psi(x)$ is a unitary in $r u_2 \L(\Gamma_{\sC}) u_2^* r$, and therefore, 
\begin{equation}\label{inter2}
    \begin{split}
        \|E_{u_2 \L(\Gamma_{\hat v}) u_2^*}(ww^*)\|_2&=\|E_{u_2 \L(\Gamma_{\hat v}) u_2^*}(\psi(x)ww^*)\|_2= \| E_{u_2\L(\Gamma_{\hat v})u_2^*}(wxw^*)\|_2\\
        &=\| E_{u_2\L(\Gamma_{\hat v})u_2^*}(wE_{u_3 \L(\Gamma_v)u_3^*}(x)w^*)\|_2\\
        &=\| E_{\L(\Gamma_{\hat v})}(u_2^*w u_3E_{\L(\Gamma_v)}(u_3^*x u_3)u_3^*w^* u_2)\|_2.
    \end{split}
\end{equation}

Since $w\neq 0$ then \eqref{inter2} and basic approximations of $u_2^*w u_3$ and $u_3$ imply the existence of a finite subset $F\subset \Gamma$ and a constant $C>0$ such that for all unitaries $x\in t\L(C_\Lambda(\Lambda_{\hat v}))tz_2$ we have\begin{equation}\label{boundbelow}
  \sum_{g,h,k,l\in F }| E_{\L(\Gamma_{\hat v})}(u_gE_{\L(\Gamma_v)}(u_hx u_k) u_l)\|_2\geq C.  
\end{equation}  

\noindent However, we see that  for all $x\in t\L(C_\Lambda(\Lambda_{\hat v}))tz_2$ we have \begin{equation*}\begin{split}\sum_{g,h,k,l\in F }| E_{\L(\Gamma_{\hat v})}(u_gE_{\L(\Gamma_v)}(u_hx u_k) u_l)\|_2&= \sum_{g,h,k,l\in F }| P_{ g^{-1}\Gamma_{\hat v}l^{-1}\cap \Gamma_v}(u_hx u_k)\|_2
\\&= \sum_{g,h,k,l\in F, gl^{-1}\in \Gamma_v }| P_{ (g^{-1}\Gamma_{\hat v}g^{-1}\cap \Gamma_v)gl^{-1}}(u_hx u_k)\|_2\\&= \sum_{g,h,k,l\in F, gl^{-1}\in \Gamma_v }| E_{ \L(g^{-1}\Gamma_{\hat v}g^{-1}\cap \Gamma_v)}(u_hx u_{klg^{-1}})\|_2\\
&=\sum_{g,h,k,l\in F, gl^{-1}\in \Gamma_v }| \tau(x u_{klg^{-1}h})|.\end{split}\end{equation*}
Here, we denoted by $P_S:\ell^2(\Gamma)\to\ell^2(\Gamma)$ the orthogonal projection onto the $\|\cdot\|_2$-closure of the ${\rm span}\{u_g|g\in S\}$ for any subset $S\subset\Gamma$.
As $F$ is finite and $t\L(C_\Lambda(\Lambda_{\hat v}))tz_2$ is diffuse there is a sequence of unitaries $x_n \in t\L(C_\Lambda(\Lambda_{\hat v}))tz_2$ such that $\sum_{g,h,k,l\in F, gl^{-1}\in \Gamma_v }| \tau(x_n u_{klg^{-1}h})|\ra 0$ as $n\ra \infty$. This however contradicts  \eqref{boundbelow}. Hence \eqref{corenonintertwining} must hold.

Thus \eqref{controlrelcom} implies that  $z_2 (\L(C_\Lambda(\Lambda_{\hat v}))f\vee \L(C_\Lambda(\Lambda_{\hat v}))f'\cap f\M f) z_2\subseteq u_2\L(\Gamma_\sC)u_2^*$. Again since $\L(\Gamma_\sC)$ is a factor there is $u\in \mathscr U(\M)$ so that 
\begin{equation}\label{inteq2} (\L(C_\Lambda(\Lambda_{\hat v}))f\vee \L(C_\Lambda(\Lambda_{\hat v}))f'\cap f\M f) z\subseteq u\L(\Gamma_\sC)u^*,\end{equation} where $z$ is the central support of $z_2$ in $(\L(C_\Lambda(\Lambda_{\hat v}))f\vee \L(C_\Lambda(\Lambda_{\hat v}))f'\cap f\M f)$. In particular, we have $vv^*\leq z_2 \leq z \leq f$. Now since  $f\L(\Lambda_{\hat v} )f\subseteq  \L(C_\Lambda(\Lambda_{\hat v}))f'\cap f\M f$ then by \eqref{inteq2} we get $(f\L(\Lambda_{\hat v} )f\vee \L(C_\Lambda(\Lambda_{\hat v}))f)z\subseteq u\L(\Gamma_\sC)u^*$ and hence 
\begin{equation}\label{inteq3} u^*(\L(C_\Lambda(\Lambda_{\hat v}))f\vee f \L(\Lambda_{\hat v} )f) z u\subseteq \L(\Gamma_\sC).\end{equation}

In particular, \eqref{inteq3} implies that $u^*f\L(\Lambda_{\hat v})fzu\subseteq \L(\Gamma_\sC)$. 
Since $vv^*\leq z\in f\L(C_\Lambda(\Lambda_{\hat v}))f'\cap f\M f$ and $\B$ is a factor then the map $\phi': a\L(\Gamma_{\hat v})a \rightarrow u^* \B z  u \subseteq u^*f \L(\Lambda_{\hat v})fz u$ given by $\phi'(x)=u^* \phi(x)z u$ still defines a $\ast$-isomorphism that satisfies $\phi'(x) w=wx$, for any $x\in a\L(\Gamma_{\hat v})a$, where $w= u^*zv$ is a non-zero partial isometry. Hence, $\L(\Gamma_{\hat v}) \prec_{\M}  u^*f\L(\Lambda_{\hat v})fzu$. Thus, by Corollary \ref{controlquasinormalizer2} it follows that $\L(\Gamma_{\hat v})\prec_{\L(\Gamma_\sC)} u^*f\L(\Lambda_{\hat v})fzu$.

\noindent To this end using \cite[Proposition 2.4]{CKP14} and its proof we can find non-zero $p \in \mathscr P(\L(\Gamma_{\hat v}))$, $r=u^*ezu \in u^*f\L(\Lambda_{\hat v})fzu$ with $e\in \mathscr P(f\L(\Lambda_{\hat v})f)$, a von Neumann subalgebra $\C \subseteq  u^*e\L(\Lambda_{\hat v})ezu $, and a $\ast$-isomorphism $\theta : p\L(\Gamma_{\hat v})p \rightarrow \C$ such that:
\begin{enumerate}
\item [a)] the inclusion $\C \vee (\C' \cap  u^*e\L(\Lambda_{\hat v})ezu ) \subseteq  u^*e\L(\Lambda_{\hat v})ezu $  has finite index;
\item [b)]there is a non-zero partial isometry $y \in \L(\Gamma_\sC)$ such that $\theta(x) y =y x$ for all $x\in p\L(\Gamma_{\hat v})p$, where $y^*y\in p\L(\Gamma_{\hat v})p'\cap p\M p$ and $yy^*\in \C '\cap r\M r$. 
\end{enumerate}


Note that $r\in \L(\Gamma_{\mathscr C})$ and $\C$, $\C' \cap   u^* r\L(\Lambda_{\hat v})r z u$ and $u^*\L(C_\Lambda(\Lambda_{\hat v}))r z u$ are commuting von Neumann subalgebras of  $r \L(\Gamma_{\mathscr C})r$. Since $\C$ is isomorphic to a corner of $\L(\Gamma_{\hat v})$ and $u^*\L(C_\Lambda(\Lambda_{\hat v}))ezu$ has no amenable direct summand, then assumption implies that $\C' \cap  u^* e\L(\Lambda_{\hat v})ezu$ is purely atomic. 
Thus, one can find a non-zero projection $q \in \mathcal Z(\C' \cap   u^* e\L(\Lambda_{\hat v})e zu )$ such that after compressing the containment in a) by $q$ and replacing $\C$ by $\C q$, $y$ by $qy$  and  $\theta(x)$ by $\theta(x) q$ in b) we can assume in addition that  $\C \subseteq  u^*e\L(\Lambda_{\hat v})ezu $  is a finite index inclusion of non-amenable II$_1$ factors.
By \cite[Proposition 1.3]{PP86} it follows that 
$\C \subseteq  u^*e\L(\Lambda_{\hat v})ezu $ admits a finite Pimsner-Popa basis, which implies that there exist $x_1,\dots,x_m\in u^*e\L(\Lambda_{\hat v})ezu$ such that $u^*e\L(\Lambda_{\hat v})ezu=\sum_{i=1}^m x_i \C$. 
Note also that $u^*e\L(\Lambda_{\hat v})ezu \subset r\M r$ since $r=u^*ezu$. Hence,
\begin{equation}\label{jj}
 \mathcal {QN}_{ r\M r}(\C ) ''=\mathcal {QN}_{ r\M r}( u^*e\L(\Lambda_{\hat v})ezu )''.   
\end{equation}

Also, the intertwining relation in b) shows that $\C yy^*= y p\L(\Gamma_{\hat v})p y^*= l p\L(\Gamma_{\hat v})p y^*y l$ where $yy^*\in \C'\cap r \M r$ and $l\in \L(\Gamma_\sC)$ is a unitary extending $y$, meaning $y=ly^*y$. Therefore, using the quasi-normalizer formulas for group von Neumann algebras  and for compressions, Lemma \ref{QN2}  and Lemma \ref{QN1} (and Remark \ref{QN.remark}), respectively,  we deduce that
\begin{equation}\label{eqquasicorner}\begin{split}
  ly^*y \L(\Gamma_\sC) y^*yl^* & = ly^*y \L({\rm QN}_\Gamma (\Gamma_{\hat v})) y^*yl^*   \overset{L. \ref{QN2}}{=} l y^*y \mathcal{QN}_{\M}(\L(\Gamma_{\hat v}))''y^*yl^*\\
  & \overset{L. \ref{QN1}}{=}\mathcal {QN}_{l y^*y \M y^*yl}(l p\L(\Gamma_{\hat v})p y^*y l)''= \mathcal {QN}_{yy^*\M yy^*}(\C yy^*)''
  \\
  &\overset{R. \ref{QN.remark}}{=}yy^*\mathcal {QN}_{ r\M r}(\C ) ''yy^*\overset{\eqref{jj}}{=}yy^*\mathcal {QN}_{ r\M r}( u^* e\L(\Lambda_{\hat v})ezu )'' yy^*\\
  &\overset{L. \ref{QN1}}{=} yy^*  u^*z e\mathcal {QN}_{ \L(\Lambda)}(\L(\Lambda_{\hat v}) )''ez u yy^*\overset{L. \ref{QN2}}{=}yy^*  u^* z e\L ({\rm QN}_{\La}(\Lambda_{\hat v}) )e zu  yy^*.
\end{split}\end{equation}

\noindent Denote  $\Upsilon = {\rm QN}_\Lambda(\Lambda_{\hat v})$ and  $\Sigma = \langle {\rm QN}^{(1)}_\Lambda(\Upsilon))\rangle <\La$. As ${\rm QN}^{(1)}_\Gamma(\Gamma_\sC)= \Gamma_\sC$, then formula \eqref{eqquasicorner} together with the corresponding formulas for one-sided quasinormalizers, Lemma \ref{QN1} and Lemma \ref{QN2}, show that \begin{equation}\label{eqquasicorner2}yy^* u^* ze\L (\Upsilon)e zu  yy^* =yy^*  u^* z e\L (\Sigma)ez u yy^*= ly^*y \L(\Gamma_\sC) y^*yl^*.\end{equation}
 In particular, by \cite[Lemma 2.2]{CI17} we have $[\Sigma:\Upsilon]<\infty$ and one can also check that ${\rm QN}^{(1)}_\Lambda(\Sigma)=\Sigma$.

\noindent Notice the above relations also show that $yy^*= u^* d  u$ for some projection $d\in ze\L (\Sigma)ez$. Thus, relations \eqref{eqquasicorner2} entail that $ u^*  d\L (\Sigma) d u = ly^*y \L(\Gamma_\sC) y^*yl^*$ and letting $w_0:=ul \in \mathscr U(\M)$ and $t=y^*y$  we conclude that $w^*_0 d \L(\Sigma)d w_0= t \L(\Gamma_\sC) t$. Moreover, if we replace $w^*_0 \Sigma w_0$ by $\Sigma$ and we use $w^*_0 d w_0= t$ we have that $t \L(\Sigma)t = t \L(\Gamma_\sC)t$. As $\L(\Gamma_\sC)$ is a factor one can find a unitary $w_1 \in \M$  so that if $c$ denotes the central support of $t \in \L(\Sigma )$ we have that $\L(\Sigma) c \subseteq w_1\L(\Gamma_\sC)w_1^*$. This implies that there exists a projection $h \in \L(\Gamma_\sC)$ such that $t= w_1 hw_1^*$. Moreover, since $\L(\Gamma_\sC)$ is a factor there is a unitary $w_2 \in \L(\Gamma_\sC)$ so that $t = w_2 hw^*_2$. Altogether these relations show that $wt=tw$ where $w :=w_1 w_2^*$. Also we have that $\L(\Sigma)c\subseteq w \L(\Gamma_\sC)w^*$. Multiplying on both sides by $t$ we get $t \L(\Gamma_\sC)t=t\L(\Sigma)t\subseteq t w\L(\Gamma_\sC)w^* t$ and hence $tw^*t \L(\Gamma_\sC) t\subseteq t \L(\Gamma_\sC)t w^* t$. In particular, using Corollary \ref{controlquasinormalizer2} we get $w^*t =tw^*t\in \mathcal{QN}^{(1)}_{t\M t}(t\L(\Gamma_\sC)t)=t \L(\Gamma_\sC)t$ and hence $wt \in t \L(\Gamma_\sC)t$. Altogether, these relations imply that $t \L(\Sigma)t= t\L(\Gamma_\sC)t=w t \L(\Gamma_\sC)t w^*$. Since $\L(\Sigma) c\subseteq w \L(\Gamma_\sC)w^*$, we apply the moreover part in \cite[Lemma 2.6]{CI17} and derive that $\L(\Sigma) c= c \L(\Gamma_\sC)c$.

In conclusion, we showed there is a subgroup $\Lambda_{\hat v} {\rm C}_\Lambda(\Lambda_{\hat v})<{\rm QN}_\Lambda(\Lambda_{\hat v})<\Sigma<\Lambda$ such that $[\Sigma:{\rm QN}_\Lambda(\Lambda_{\hat v})]<\infty$, and ${\rm QN}^{(1)}_\Lambda(\Sigma)=\Sigma$. Moreover, there are $r\in \sP(\mathscr Z(\L(\Sigma)))$, $v_0\in \sU(\M)$, and $n\in \mathscr P(\L(\Gamma_\sC))$ so that \begin{equation}\label{eq:property.T}
v_0\L(\Sigma) rv^*_0=c\L(\Gamma_\sC)c.\end{equation}

Finally, we show that $[\Sigma: \Lambda_{\hat v} {\rm vC}_\Lambda(\Lambda_{\hat v})]<\infty$.
Let $u^*qzu=s\in u^* e \L(\Lambda_{\hat v})e z u$ be a projection and let $\C_1\subseteq \C \subseteq u^* e \L(\Lambda_{\hat v})e z u$ be a subfactor such that $u^* e \L(\Lambda_{\hat v})e z u=\langle \C, s \rangle$ is the basic construction for $\C_1\subseteq \C$. Notice that $\D_1 := \theta^{-1}(\C_1)\subseteq p\L(\Gamma_{\hat v})p$ is a finite index subfactor. Let $\Theta: \D_1 \ra \C_1s= u^*q \L(\Lambda_{\hat v})qzu$  be the isomorphism given by $\Theta(x)=\theta(x)s$ for all $x\in \D_1$. Using b) we have that $\Theta(x) sy=sy x$ for all $x\in \D_1$. Since $s$ is a Jones projection for $\C_1\subseteq \C$ one can check that $sy\neq 0$. Letting $w$ be the polar decomposition  of $sy$ we further have that $\Theta(x) w=w x$ for all $x\in \D_1$.   
By Corollary \ref{controlquasinormalizer2} we have $s\in \L(\Gamma_\sC)$, $ww^*\in u^*q\L(\Lambda_{\hat v})qzu'\cap s\M s \subset s\L(\Gamma_\sC)s$ and $w^*w \in p\D_1 p' \cap p\M p\subset  p\D_1 p' \cap pL(\Gamma_\sC) p$. Letting $w_o$ be an unitary extending $w$ and using the previous  intertwining relation we see that $u^* q\L(\Lambda_{\hat v})qzu  ww^*= w_o \D_1 w^*_o$. By taking relative commutants we obtain that $ ww^*(u^* q\L(\Lambda_{\hat v})qzu'\cap s \M s)ww^*=w_o \D_1'\cap p\L(\Gamma_\sC)p w_o^* $, and hence,  \begin{equation}\label{3.33}\begin{split} &w_o (\D_1\vee \L(\Gamma_v)p )w_o^*\subset w_o (\D_1\vee \D_1'\cap p\L(\Gamma_\sC)p )w_o^*=\\
&=ww^*(u^* q\L(\Lambda_{\hat v})qzu  \vee u^* qL(\Lambda_{\hat v})qzu'\cap s \M s)ww^* \\ &=ww^*u^* qz (\L(\Lambda_{\hat v})  \vee  \L(\Lambda_{\hat v})'\cap  \M ) qzu ww^*\subseteq ww^* \L(\Gamma_\sC)ww^*.\end{split}\end{equation} 
Since  $\Gamma_\sC$ is malnormal in $\Gamma$ the containment $w_o (\D_1\vee \L(\Gamma_v)p )w_o^*\subseteq ww^* \L(\Gamma_\sC)ww^*$ implies that $w=w_o p\in \L(\Gamma_\sC)$ and hence $ww^* \L(\Gamma_\sC)ww^*= w\L(\Gamma_\sC)w^*= w_o p\L(\Gamma_\sC)pw_o^*$. Therefore, since $\D_1\vee \L(\Gamma_v)p\subseteq p\L(\Gamma_\sC)p$ has finite index, then so does $w_o (\D_1\vee \L(\Gamma_v)p )w_o^*\subset  ww^*\L(\Gamma_\sC) ww^*$. In particular, all inclusions in \eqref{3.33} are of finite index. Hence, we have  $\L(\Gamma_\sC)\prec_\M \L(\Lambda_{\hat v})\vee\L(\Lambda_{\hat v})'\cap \M$ and since $\L(\Lambda_{\hat v})\vee \L(\Lambda_{\hat v})'\cap \M\subseteq \L(\Lambda_{\hat v} vC_\Lambda(\Lambda_{\hat v}))$ we get $\L(\Gamma_\sC)\prec_\M \L(\Lambda_{\hat v} vC_\Lambda(\Lambda_{\hat v}))$. Using \eqref{eq:property.T} this further implies that $\L(\Sigma)\prec_\M \L(\Lambda_{\hat v} {\rm vC}_\Lambda(\Lambda_{\hat v}))$. Moreover, since  $\Lambda_{\hat v}{\rm vC}_\Lambda(\Lambda_{\hat v})<\Sigma$ and ${\rm QN}^{(1)}(\Sigma)=\Sigma$ we actually have that $\L(\Sigma)\prec_{\L(\Sigma)} \L(\Lambda_{\hat v} {\rm vC}_\Lambda(\Lambda_{\hat v}))$. Using \cite[Lemma 2.5]{DHI16} this implies that $[\Sigma: \Lambda_{\hat v} {\rm vC}_\Lambda(\Lambda_{\hat v})]<\infty$, as claimed. 
\end{proof}






We continue with the following technical result which goes back to \cite[Theorem 3.2]{CI17}. The proof  goes along the same lines with the proof of \cite[Theorem 9.2]{CD-AD20} and we include all the details for completeness.

\begin{thm}\label{theorem.upgrade.inclusions}
Assume that $\Gamma= \mathscr G \{\Gamma_v\}$ is a graph product of groups such that $\mathscr G\in {\rm CC}_1$ and all vertex groups $\Gamma_v$ are icc, non-amenable. Let $\sC_1,\dots,\sC_n$ be an enumeration of its consecutive cliques and denote $\Gamma_i=\Gamma_{\sC_i}$, for any $i\in\overline{1,n}$. 

Let $\Lambda$ be an arbitrary group such that $\M=\L(\Gamma)=\L(\Lambda)$. Assume that for any $i\in \{1,\dots,n\}$ there exist $\Lambda_i<\Lambda$ with $\rm{QN_\Lambda^{(1)}}(\Lambda_i)=\Lambda_i$ and a subset $i\in J_i\subset \overline{1,n}$ satisfying:
\begin{enumerate}
    \item For any $k\in J_i$ there exists a projection $0\neq z_i^k\in \mathcal Z(\L(\Lambda_i))$ such that $\sum_{k\in J _i}z_i^k=1$;
    \item For any $k\in J_i$ there exists  $u_i^k\in \mathcal U(\M)$ such that:
        \begin{enumerate}
            \item [a.] $u_i^i \L(\Lambda_i) z_i^i (u_i^i)^*=p_i\L(\Gamma_i)p_i$, where $p_i=u_i^i z_i^i (u_i^i)^*$;
            \item [b.] $u_i^k \L(\Lambda_i) z_i^k (u_i^k)^*\subset\L(\Gamma_k)$.
        \end{enumerate}
\end{enumerate}
Then there exist a partition $T_1\sqcup \dots \sqcup T_l= \{1,...,n\}$ and a subgroup $\tilde \Lambda_i<\Lambda$ with $\rm{QN_\Lambda^{(1)}}(\tilde  \Lambda_i)=\tilde \Lambda_i$ for any $1\leq i\leq l$ such that:

\begin{enumerate}
    \item For any $k\in T_i$ there exists a projection $0\neq \tilde z_i^k\in \mathcal Z(\L(\tilde \Lambda_i))$  satisfying $\sum_{k\in T _i}\tilde z_i^k=1$;
    \item For any $k\in T_i$ there exists $\tilde u_i^k\in \mathcal U(\M)$ such that
         $\tilde u_i^k \L(\tilde \Lambda_i) \tilde z_i^k (\tilde u_i^k)^* = \tilde p_k^i\L(\Gamma_k) \tilde p_k^i$, where $\tilde p_k^i=\tilde u_i^k \tilde z_i^k (\tilde u_i^k)^*$.
        
\end{enumerate}

\end{thm}

\begin{proof} We fix an arbitrary $k\in J_1\setminus\{1\}$. Note that the assumption 2. implies that $\L(\Lambda_1)\prec_{\M} \L (\Gamma_k)$ and $\L(\Gamma_k)\prec^s_{\M} \L(\Lambda_k)$ since $\L(\Gamma_k)'\cap \M=\mathbb C1.$ Using \cite[Lemma 3.7]{Va08} we deduce that $\L(\Lambda_1)\prec_{\M} \L(\Lambda_k)$. Using \cite[Lemma 2.2]{CI17} there exists $h_k\in \Lambda$ such that $[\Lambda_1: h_k\Lambda_kh_k^{-1}\cap \Lambda_1]<\infty$. Up to replacing $\Lambda_k$ by $h_k \Lambda_k h^{-1}_k$, we may assume that $h_k=1$.

We continue by showing that for all $i\in J_1, j\in J_k$ with $i\neq j$, we have $z_1^i z_k^j=0$. By assuming the contrary, we consider some $i\in J_1,j\in J_k$ with $i\neq j$ satisfying $z_1^i z_k^j\neq 0$. Assumption 2. gives that $\L(\Lambda_k\cap \Lambda_1) z_1^i\prec_{\M}^s \L(\Gamma_i)$ and $\L(\Lambda_k\cap \Lambda_1) z_k^j\prec_{\M}^s \L(\Gamma_j)$. From \cite[Lemma 2.4]{DHI16} there exist maximal projections $a_1^i, a_k^j\in  \mathcal Z(\L(\Lambda_k\cap \Lambda_1)'\cap \M)$ such that
$\L(\Lambda_k\cap \Lambda_1) a_1^i\prec_{\M}^s \L(\Gamma_i)$ and $\L(\Lambda_k\cap \Lambda_1) a_k^j\prec_{\M}^s \L(\Gamma_j)$.
Since $z_1^i z_k^j\neq 0$, we deduce that $a_1^i a_k^j\in \mathcal Z(\L(\Lambda_k\cap \Lambda_1)'\cap \M)$ is a non-zero projection. Thus, we can apply \cite[Lemma 2.7]{Va10a} to get that $\L(\Lambda_k\cap \Lambda_1) a_1^i a_k^j\prec_{\M}^s \L(\Gamma_i\cap g \Gamma_j g^{-1})$ for some $g\in \Gamma$. Since $[\Lambda_1: \Lambda_k\cap \Lambda_1]<\infty$ we therefore deduce that 
$\L(\Gamma_1)\prec_{\M} \L(\Lambda_k\cap \Lambda_1)$ and $\L(\Lambda_k\cap \Lambda_1)\prec_{\M} \L(\Gamma_{\hat v})$ for some $v\in \sC_{i}$. By Lemma \ref{control.relative.commutants}
 we get that $\Gamma_v$ is amenable, contradiction.
 
Thus, we have $J_k=J_1$ and $z_1^i=z_k^i$ for all $i\in J_1$. Indeed, if $i\in J_1$, then there exists $l\in J_k$ such that $z_1^i z_k^l\neq 0$. The previous paragraph shows that $l=i$. This shows, in particular, that $J_1\subset J_k$. By symmetry, we also get that $J_k\subset J_1$. The previous paragraph also implies that $z_1^i z_k^{i'}= 0$ for all $i'\in J_k$ with $i'\neq i$. This implies that $z_1^i\leq z_k^i$ and the conclusion follows again by symmetry reasons. 

Next, we show that $\Lambda_k=\Lambda_1$. Assumption 2. gives that $\text{for all } i\in J_1=J_k$ we have
\begin{equation}\label{eq.a1}
u_1^i \L(\Lambda_k\cap \Lambda_1)z_1^i (u_1^i)^*\subset \L(\Gamma_i) \text{ and } u_k^i \L(\Lambda_k\cap \Lambda_1)z_k^i (u_k^i)^*\subset \L(\Gamma_i). 
\end{equation}
Fix an element $w_i\in \sC_i$ and note that
$\Gamma = \Gamma_{\sV\setminus \{w_i\}} \ast_{\Gamma_{\hat {w_i}}} \Gamma_i$, where $\Gamma=\sG\{\Gamma_v,v \in \sV\}$.
As before, Lemma \ref{control.relative.commutants} implies that  
\begin{equation}\label{eq.a2}
u_k^i \L(\Lambda_k\cap \Lambda_1)z_k^i (u_k^i)^* \nprec_{\L(\Gamma_i)} \L(\Gamma_{\hat {w_i}}).    
\end{equation}
By letting $a^i=u_1^i z_1^i (u_k^i)^*$, we note that 
\begin{equation}\label{eq.a3}
a^i u_k^i \L(\Lambda_k\cap \Lambda_1)z_k^i (u_k^i)^*= u_1^i \L(\Lambda_k\cap \Lambda_1) z_1^i (u_1^i)^* a^i    
\end{equation}
By using \eqref{eq.a1}, \eqref{eq.a2}, \eqref{eq.a3} and
by applying \cite[Theorem 1.2.1]{IPP05}, we get that $a^i\in \L(\Gamma_i)$. Assumption 2. implies that 
$$
u_1^1 \L(\Lambda_1) z_1^1 (u_1^1)^*=p_1\L(\Gamma_1)p_1 \text{ and } u_k^1 \L(\Lambda_k) z_k^1 (u_k^1)^*\subset \L(\Gamma_1).
$$
By noticing that $z_1^1=z_k^1$ and by conjugating the previous inclusion by $a^1\in \L(\Gamma_1)$, we derive that
$$
u_1^1 \L(\Lambda_1) z_1^1 (u_1^1)^*=p_1\L(\Gamma_1)p_1 \text{ and } u_1^1 \L(\Lambda_k) z_1^1 (u_1^1)^*\subset \L(\Gamma_1).
$$
Hence, $u_1^1 z_1^1 (u_1^1)^*=p_1$ and therefore, $\L(\Lambda_k\cap \Lambda_1)z_1^1\subset\L(\Lambda_k)z_1^1\subset \L(\Lambda_1)z_1^1$. Since $[\Lambda_1: \Lambda_k\cap \Lambda_1]<\infty$, we deduce that $\L(\Lambda_k\cap \Lambda_1)z_1^1\subset\L(\Lambda_k)z_1^1$ admits a finite Pimsner-Popa basis. By applying \cite[Proposition 2.6]{CdSS15} we derive that $[\Lambda_k: \Lambda_k\cap \Lambda_1]<\infty$. In combination with $[\Lambda_1: \Lambda_k\cap \Lambda_1]<\infty$, we get that
$$
\Lambda_k={\rm QN}^{(1)}_\Lambda (\Lambda_k)= {\rm QN}^{(1)}_\Lambda (\Lambda_k\cap\Lambda_1)={\rm QN}^{(1)}_\Lambda (\Lambda_1)=
\Lambda_1.
$$
Therefore, we have $u_1^1\L(\Lambda_1) z_1^1 (u_1^1)^*=p_1\L(\Gamma_1) p_1$ and $u_k^k\L(\Lambda_1) z_k^k (u_k^k)^*=p_k\L(\Gamma_k) p_k$. Since $z_1^k=z_k^k$ and all these work for all $k\in J_1$, we get the conclusion of the proof for the first element of the partition. Namely, we let $\tilde \Lambda_1=\Lambda_1$, $T_1=J_1$, $\tilde  u_1^k= u_k^k $ and $\tilde z_1^k= z_k^k$. Finally, if $T_1=\{1,\dots,n\}$, then the proof is completed. Otherwise, pick an element $t\in \{1,\dots,n\}\setminus T_1$ and repeat the above arguments starting with $J_t$.
\end{proof}

\begin{thm}\label{peripheralrec1}
   Let $\Gamma= \mathscr G \{\Gamma_v\}$ be a graph product of groups such that $\mathscr G\in {\rm CC}_1$ and  for any $v\in\mathscr V$, $\Gamma_v\in \mathcal W\mathcal R (A_v, B_v \ca I_v)$ where $A_v$ is abelian, $B_v$ is an icc  subgroup of a hyperbolic group.
   Assume that $|\sC|\neq |\sD|$ for any two distinct cliques $\sC,\sD\in {\rm cliq}(\sG)$.

Let $\Lambda$ be an arbitrary group such that $\L(\Gamma)=\L(\Lambda)$. Thus for every clique $\sC\in {\rm cliq}( \sG)$ there exist a unitary $u_\sC \in \M$  and a subgroup $\Lambda_\sC \leqslant \Lambda$ such that $u_\sC \L(\Gamma_\sC) u_\sC^*= \L(\Lambda_\sC)$.   
\end{thm}

\begin{proof}
The proof is using several techniques from \cite{CI17,CD-AD20}. First, let $\sC_1,\dots,\sC_n$ be an enumeration of its consecutive cliques and denote $\Gamma_i=\Gamma_{\sC_i}$, for any $1\leq i\leq n$. 

\begin{claim}\label{cornersidentifications} For any $1\leq i\leq n$, there exist $\Lambda_i<\Lambda$ with $\rm{QN_\Lambda^{(1)}}(\Lambda_i)=\Lambda_i$ and a subset $i\in J_i\subset \{1,...,n\}$ satisfying:
\begin{enumerate}
    \item For any $k\in J_i$ there exists a projection $0\neq z_i^k\in \mathcal Z(\L(\Lambda_i))$ such that $\sum_{k\in J _i}z_i^k=1$;
    \item For any $k\in J_i$ there exists  $u_i^k\in \mathcal U(\M)$ such that:
        \begin{enumerate}
            \item [a.] $u_i^i \L(\Lambda_i) z_i^i (u_i^i)^*=p_i\L(\Gamma_i)p_i$, where $p_i=u_i^i z_i^i (u_i^i)^*$;
            \item [b.] $u_i^k \L(\Lambda_i) z_i^k (u_i^k)^*\subset\L(\Gamma_k)$.
        \end{enumerate}
\end{enumerate}

\end{claim}

\noindent {\it Proof of the Claim \ref{cornersidentifications}.} 
Fix $1\leq i\leq n$. By applying Theorems \ref{commutingsubgroups1} and \ref{identificationcornergroups1}, there exist a subgroup $\Lambda_i<\Lambda$ with ${\rm QN_\Lambda^{(1)}}(\Lambda_i)=\Lambda_i$, a projection $0\neq z_i\in \mathcal Z(\L(\Lambda_i))$ and $u_i\in \mathscr U(\M)$ with $p_i=u_iz_i (u_i)^*\in \L (\Gamma_i)$ such that 
\begin{equation}\label{eq:}
u_i\L(\Lambda_i)z_i (u_i)^*= p_i \L(\Gamma_i) p_i.    
\end{equation}
Since $\Gamma_i$ has property (T), then \eqref{eq:} together with \cite[Lemma 2.13]{CI17} imply that $\Lambda_i$ has property (T) as well. 
Using ${\rm QN_\Lambda^{(1)}}(\Lambda_i)=\Lambda_i$ and $\L(\Lambda_i)'\cap \M\subset \L({\rm vC}_{\Lambda}(\Lambda_i))$, it follows that 
\begin{equation}\label{eq.b1}
    \L(\Lambda_i)'\cap \M=\mathscr Z (\L(\Lambda_i)).
\end{equation}

Since $\L(\Gamma_j)$ is a II$_1$ factor for any $1\leq j\leq n$, by using \eqref{eq.b1} there is a maximal projection $z_j^i\in \mathscr Z(\L(\Lambda_i))$ such that $\L(\Lambda_i)z_j^i$ can be unitarily conjugated into $\L(\Gamma_j)$; hence, we take  $u_j^i\in \mathscr U(\M)$ such that $u_j^i\L(\Lambda_i)z_j^i (u_j^i)^*\subset  \L(\Gamma_j).$ 
We continue by showing that $\bigvee_{j=1}^n z_j^i=1$. By assuming the contrary, then  $z=1-\bigvee_{j=1}^n z_j^i$ is a non-zero projection of $\mathscr Z(\L(\Lambda_i))$. Since $\L(\Lambda_i)z$ has property (T), 
Theorem \ref{proptcliques} implies that there exists $j\in\overline{1,n}$ such that $\L(\Lambda_i)z\prec_{\M} \L(\Gamma_j)$. Note that $\L(\Lambda_i)z\nprec_{\M} \L(\Gamma_{\mathscr C_j\setminus\{v\}})$, for any $v\in\mathscr C_j$. Indeed, otherwise there is $v\in\mathscr C_j$ such that $\L(\Lambda_i)\prec_{\M} \L(\Gamma_{\mathscr C_j\setminus\{v\}})$. By using \cite[Lemma 3.5]{Va08} and \eqref{eq.b1}, we get $\L(\Gamma_{v})\prec_{\M} \L(\Lambda_i)'\cap \M=\mathscr Z(\L(\Lambda_i))$, which is impossible since $\Gamma_v$ is non-amenable. Therefore, the moreover part of Theorem \ref{proptcliques} gives a non-zero projection $z_0\in \mathscr Z (\L(\Lambda_i))$ and  $u_0\in\mathscr U(\M)$ such that $u_0 \L(\Lambda_i)z_0 u_0^*\subset \L(\Gamma_j)$. Recall that $u_j^i\L(\Lambda_i)z_j^i (u_j^i)^*\subset  \L(\Gamma_j).$ Since $\L(\Gamma_j)$ is a II$_1$ factor, we can perturb $u_0$ by a different unitary and assume that $u_0z_0 u_0^*$ and $u_j^iz_j^i (u_j^i)^*$ are orthogonal projections in $\L(\Gamma_j)$. By letting $u=u_0z_0+u_j^iz_j^i$, we note that  $u^*u=z_0+z_j^i$ and let $v\in \mathscr U(\M)$ such that $vu^*u=u$. Therefore, $v \L(\Lambda_i)(z_0+z_j^i) v^*\subset \L(\Gamma_j)$, which contradicts the maximality of $z_j^i$.  This shows that $\bigvee_{j=1}^n z_j^i=1$.

Next, we prove that $z_i^i=z_i$. By assuming the contrary, the maximality of $z_i^i$ implies that $c:=z_i^i-z_i$ is a non-zero projection in $\mathscr Z(\L(\Lambda_i))$. By using \eqref{eq:} and $u_i^i\L(\Lambda_i)c (u_i^i)^*\subset  \L(\Gamma_j),$ we can proceed as in the previous paragraph and assume that $p_i$ and $p:=u_i^i c (u_i^i)^*$ are orthogonal projections of $\L(\Gamma_i)$. Hence, by letting $u=u_i^i c + u_i z_i$, we derive that $uu^*=p+p_i$, $u^*u=c+z_i$ and 
\begin{equation}\label{eq::}
    u\L(\Lambda_i)(c+z_i) u^*\subset  (p+p_i)\L(\Gamma_j)(p+p_i).
\end{equation}

Since ${\rm QN_\Lambda^{(1)}}(\Lambda_i)=\Lambda_i$, we obtain from Lemma \ref{QN1} and Lemma \ref{QN2} that 
\begin{equation}\label{eq:::}
{\rm \mathscr{QN}}^{(1)}_{(p+p_i)\L(\Gamma_j)(p+p_i)}(u\L(\Lambda_i)(c+z_i) u^*)=u\L(\Lambda_i)(c+z_i) u^*.    
\end{equation}

Relations \eqref{eq:}, \eqref{eq::} and \eqref{eq:::} and the fact that $(p+p_i)\L(\Gamma_j)(p+p_i)$ is a II$_1$ factor allow to apply the moreover part of \cite[Lemma 2.6]{CI17} and deduce that $u\L(\Lambda_i)(c+z_i) u^*=  (p+p_i)\L(\Gamma_j)(p+p_i).$ This is a contradiction, since the center of $u\L(\Lambda_i)(c+z_i) u^*$ is at least two dimensional. This shows that $z_i^i=z_i$.

Finally, since $\bigvee_{j=1}^n z_j^i=1$, for any $1\leq i\leq n$ we can eventually replace $z_j^i$ by a smaller projection in $\mathscr Z(\L(\Lambda_i))$  and assume that $\{z_j^i\}_{j=1}^n$ are mutually orthogonal.
The claim follows by letting $J_i=\{1\leq k\leq n\,:\, z_k^i\neq 0 \}$. 
\hfill$\blacksquare$

Next, we can apply Theorem \ref{theorem.upgrade.inclusions} and obtain a partition there exist a partition $T_1\sqcup \dots \sqcup T_l= \{1,...,n\}$ and a subgroup $\tilde \Lambda_i<\Lambda$ with $\rm{QN_\Lambda^{(1)}}(\tilde  \Lambda_i)=\tilde \Lambda_i$ for any $1\leq i\leq l$ such that:

\begin{enumerate}
    \item For any $k\in T_i$ there exists a projection $0\neq \tilde z_i^k\in \mathcal Z(\L(\tilde \Lambda_i))$  satisfying $\sum_{k\in T _i}\tilde z_i^k=1$;
    \item For any $k\in T_i$ there exists $\tilde u_i^k\in \mathcal U(\M)$ such that
         $\tilde u_i^k \L(\tilde \Lambda_i) \tilde z_i^k (\tilde u_i^k)^* = \tilde p_k^i\L(\Gamma_k) \tilde p_k^i$, where $\tilde p_k^i=\tilde u_i^k \tilde z_i^k (\tilde u_i^k)^*$.
        
\end{enumerate}

Note that it is enough to show that $|T_i|=1$, for all $1\leq i\leq l$. Hence, we fix an arbitrary $1\leq i\leq l$ and assume by contradiction that $|T_i|>2$. Take $k,j \in T_i$ two distinct elements such that $|\sC_k|>|\sC_j|$. Note that we have $\L(\Gamma_k)\prec_{\M} \L(\tilde \Lambda_i)$ and $\L(\tilde \Lambda_i)\prec_{\M} \L(\Gamma_j)$. Proposition \ref{lemma.counting} implies that $|\sC_k|\leq|\sC_j|$, contradiction. This shows that $|T_i|=1$, for all $1\leq i\leq l$ and the conclusion follows.
\end{proof}

\section{Superrigidity results within the category of graph product groups}

In this final section we derive a strong rigidity result for factors arising from graph product groups considered in Section 5 (see Theorem \ref{superwithincat}). Building upon the results in the prior sections we show that the graph product groups considered in Theorem \ref{peripheralrec1} are completely recognizable from the category of \emph{all} von Neumann algebras associated to any nontrivial graph product group with infinite vertex groups. Notice that unlike previous strong rigidity results \cite{IPP05,CH08} in our case we do not need any other additional a priori assumptions on the vertex groups or on the underlying graph.

To prove our theorem we first establish a result which asserts that the click subgroups of the graph products from Theorem \ref{peripheralrec1} along with various other aspects of their ``position'' in the ambient group  are features which are completely  recognizable from the von Neumann algebra framework. The result relies heavily on Theorem \ref{peripheralrec1} and for the proof we only explain how it follows from this. Although, for deriving Theorem \ref{superwithincat} we do not need all the parts of the conclusion, we decided to state it here in its complete and somewhat technical form as this could be instrumental for further attempts to establish $W^*$-superrigidity of these groups.   

\begin{thm}\label{theorem.almost.Wsuperrigidity} Let $\sG\in{\rm CC}_1$, let ${\rm cliq}(\sG)=\{ \sC_1, \ldots , \sC_n\}$ be a consecutive cliques enumeration and assume that $|\sC_i|\neq |\sC_j|$ whenever $i\neq j$.  Let $\Gamma= \sG\{\Gamma_v\}$ be a graph product group such that for any $v\in\mathscr V$, $\Gamma_v\in \mathcal W\mathcal R (A_v, B_v \ca I_v)$ where $A_v$ is abelian, $B_v$ is an icc  subgroup of a hyperbolic group and  the set $\{i\in I_v \; | \; g\cdot i\neq i\}$ is infinite for any $g\in B_v\setminus \{1\}$. Denote $\M=\L(\Gamma)$. 

Let $\Lambda$ be an arbitrary group such that $\M=\L(\Lambda)$. 
Then for any $1\leq i\leq n$ one can find a unitary $w_i\in \M$ and a subgroup $\Lambda_i<\Lambda$ satisfying the following relations: \begin{enumerate}
    \item $\mathbb T w_i \Gamma_{\sC_i} w_i^* =\mathbb T \Lambda_i$ for all $1\leq i\leq n$;
    
    \item $\mathbb T w_i \Gamma_{\sC_{i,i+1}} w_i^*= \mathbb T w_{i+1} \Gamma_{\sC_{i,i+1}} w_{i+1}^*= \mathbb T \Lambda_{i,i+1}$ where $\Lambda_{i,i+1}= \Lambda_i\cap \Lambda_{i+1}$ for all $1\leq i\leq n-1$;
    \item There is an $s\in \Lambda$ such that $\mathbb T w_n \Gamma_{\sC_{n,1}} w_n^*= \mathbb T v_s w_{1} \Gamma_{\sC_{n,1}} w_{1}^* v_{s^{-1}}= \mathbb T \Lambda_{n}\cap s \Lambda_1 s^{-1}$;
\item $w_i\L(\Gamma_{\sC_i\cup \sC_{i+1}})w_i^*=w_{i+1}\L(\Gamma_{\sC_i\cup \sC_{i+1}})w_{i+1}^* = \L(\Lambda_{i,i+1} C_\Lambda(\Lambda_{i,i+1}))$ and hence $w^*_{i+1}w_i\in \L (\Gamma_{\sC_i\cup \sC_{i+1}})$, for all  $1\leq i\leq n-1$;      
\item $\vee^{n-1}_{i=1} \left(\Lambda_{i,i+1} C_\Lambda(\Lambda_{i,i+1})\right )=\Lambda$.
\end{enumerate}

\end{thm}

\begin{proof} Using Theorem \ref{peripheralrec1} and Corollary \ref{superprod} one can find  for any $1\leq i\leq n$ a unitary $x_i\in \M$ and a subgroup $\Lambda_i<\Lambda$ satisfying $\mathbb T x_i \Gamma_{\sC_i} x_i^* =\mathbb T \Lambda_i$.  The prior relations show that $\mathbb T x_1 \Gamma_{\sC_{1,2}} x_1^* =\mathbb T \Lambda'_1$ and $\mathbb T x_2 \Gamma_{\sC_{1,2}} x_2^* =\mathbb T \Lambda'_2$ for two subproduct groups $\Lambda'_1<\Lambda_1$ and $\Lambda'_2<\Lambda_2$. Thus, for every $i=1,2$ there exist a group isomorphism $\delta_i: \Gamma_{\sC_{1,2}}\ra \Lambda_i'$ and a multiplicative character $\eta_i :\sC_{1,2}\to\mathbb T$ such that 
\begin{equation}\label{ll1}
x_i u_g x_i^*=\eta_i(g)v_{\delta_i(g)}, \text{ for all } g\in  \Gamma_{\sC_{1,2}}.     
\end{equation}
Letting $\eta(g)=\overline{\eta_2(g)}\eta_1(g)$, one can see  that altogether these relations imply that $\eta(g) x_2x_1^* v_{\delta_1(g)}= \overline{\eta_2(g)} x_2 u_g x_1^*= v_{\delta_2(g)} x_2x_1^\ast$ or all $g\in \Gamma_{\sC_{1,2}} $. Hence,  $\eta(g) x_2x_1^* = v_{\delta_2(g)} x_2x_1^*v_{\delta_1(g^{-1})}$, for all $g\in \Gamma_{\sC_{1,2}} $. 

If we consider the Fourier decomposition of $x_2x_1^*$ in $\M=\L(\Lambda)$, then basic analysis together with the previous relation show that $x_2x_1^*$ is supported on the set of those $h\in \Lambda$ for which the orbit $\{ \delta_2(g) h \delta_1(g^{-1}) \,:\, g\in \Gamma_{\sC_{1,2}}\}$ is finite. This finiteness condition yields that for any such $h$ one can find a finite index subgroup $\Sigma_h\leqslant  \Gamma_{\sC_{1,2}}$ such that $\delta_2(g) h \delta_1(g^{-1})=h$ and hence $\delta_1(g) = h^{-1}\delta_2(g)h$, for all $g\in \Sigma_h$. Thus, replacing $\delta_2$ by ${\rm ad}(h^{-1})\circ \delta_2$, $x_2$ by $v_{h^{-1}} x_2$ and $\La_2$ by $h^{-1}\La_2 h$ we can assume 
that ${\delta_1}_{|_{\Sigma_h}}= {\delta_2}_{|_{\Sigma_h}}$. Moreover, using this combined with \eqref{ll1}, we deduce that $u_g x u_{g^{-1}}= \eta(g) x$ where $x=x_1^*x_2$. However, this shows that $x$ is supported on elements $h\in \Gamma$ with finite $\Sigma_h$-conjugation orbits. Using the icc condition and the graph product structure we can see that such elements belong to $\Gamma_{\sC_1\triangle \sC_2}$, and hence, $x\in \L(\Gamma_{\sC_1\triangle \sC_2})$. Therefore, it follows from \eqref{ll1} that $\eta_1(g)v_{\delta_1(g)}=x_1 u_g x_1^*= x_2 u_g x_2^*= \eta_2(g)v_{\delta_2(g)}$ for all $g\in \Gamma_{\sC_{1,2}}$. In particular, $\delta_1=\delta_2$ on $ \Gamma_{\sC_{1,2}}$; altogether this implies the conclusion for $i=1$. Continuing by induction, after conjugating the $\Lambda_i$'s by elements in $\Lambda$ we obtain the conclusion of the first three parts.

Next, we prove 4.  First notice that 2. implies \begin{equation}\label{equalintersections}w_i\L(\Gamma_{\sC_{i,i+1}})w_i^*=w_{i+1} \L(\Gamma_{\sC_{i,i+1}}) w_{i+1}^*= \L(\La_{i,i+1}).\end{equation} 
By using the graph product structure and by taking quasinormalizers in \eqref{equalintersections}, we apply Lemma \ref{QN2} to derive that
$w_i\L(\Gamma_{\sC_i\cup \sC_{i+1}})w_i^*= w_i \L({\rm QN}_\Ga (\Ga_{\mathscr C_{i,i+1}})) w_i^*=w_{i} \mathscr {QN}_{\M}(\L(\Ga_{\mathscr C_{i,i+1}}))''w_i^*=w_{i+1} \mathscr {QN}_{\M}(\L(\Ga_{\mathscr C_{i,i+1}}))''  w_{i+1}^*= w_{i+1} (\L({\rm QN}_\Ga (\Ga_{\mathscr C_{i,i+1}}))  w_{i+1}^*= w_{i+1} \L(\Gamma_{\sC_i\cup \sC_{i+1}}) w_{i+1}^* $. The same quasinormalizer formula also implies that $w_i\L(\Gamma_{\sC_i\cup \sC_{i+1}})w_i^*=w_{i+1} \L(\Gamma_{\sC_i\cup \sC_{i+1}}) w_{i+1}^*= \L({\rm QN}_\La(\La_{i,i+1})$. In particular, $\L({\rm QN}_\La(\La_{i,i+1})$ is a II$_1$ factor. Taking relative commutants in 2.\ we also have that $w_i\L(\Gamma_{\sC_{i}\triangle \mathscr C_{i+1}})w_i^*=w_{i+1} \L(\Gamma_{\sC_{i}\triangle \mathscr C_{i+1}}) w_{i+1}^*= \L(\La_{i,i+1})'\cap \L(\Lambda)$. Altogether, these relations imply that 
$\L(\La_{i,i+1})'\cap \L(\Lambda) \vee \L(\La_{i,i+1})=\L({\rm QN}_\La(\La_{i,i+1})) $. In particular, we have ${\rm QN}_\La(\La_{i,i+1})= \La_i vC_\La(\La_{i,i+1})$.

Moreover, factoriality also shows that $(\L(\La_{i,i+1})'\cap \L(\Lambda)) \bar\otimes  \L(\La_{i,i+1})=\L({\rm QN}_\La(\La_{i,i+1})) $. As $\L(\La_{i,i+1})'\cap \L(\Lambda)\subseteq \L(vC_{\Lambda}(\La_{i,i+1}))$, then Ge's tensor splitting theorem \cite[Theorem A]{Ge95} further implies that $\L(vC_{\Lambda}(\La_{i,i+1}))=\left( \L(\La_{i,i+1})'\cap \L(\Lambda) \right)\bar \otimes B$, where $B= \L(vC_{\Lambda}(\La_{i,i+1}))\cap \L(\La_{i,i+1})$. However, we can see that $B = \L(vC_{\Lambda}(\La_{i,i+1}))\cap \L(\La_{i,i+1})=\L(vC_{\Lambda}(\La_{i,i+1})\cap  \La_{i,i+1})= \mathbb C 1$, as $\La_{i,i+1}$ is icc. In conclusion, we have shown that $\L(vC_{\Lambda}(\La_{i,i+1}))=\L(\La_{i,i+1})'\cap \L(\Lambda)$ which implies that  $vC_{\Lambda}(\La_{i,i+1})=C_\La(\La_{i,i+1})$. Hence, ${\rm QN}_\La(\La_{i,i+1})= \La_{i,i+1} C_\La(\La_{i,i+1})$. Then the remaining part of 4.\ follows from \eqref{equalintersections} and Lemma \ref{QN2} because of the graph product structure of $\sG$.

Part 5.\ follows from part 4. Indeed, since  $w_2\L(\Gamma_{\sC_1\cup \sC_{2}})w_2^*= \L(\Lambda_{1,2} C_\Lambda(\Lambda_{1,2}))$ and $w_{2}\L(\Gamma_{\sC_2\cup \sC_{3}})w_{2}^* = \L(\Lambda_{2,3} C_\Lambda(\Lambda_{2,3}))$, we further get  $w_2 \L(\Ga_{\sC_1\cup \sC_2\cup \sC_{3} }) w_2^*= \L(\Lambda_{1,2} C_\Lambda(\Lambda_{1,2})\vee \Lambda_{2,3} C_\Lambda(\Lambda_{2,3}))$. Since $w_3^*w_2\in \L (\Gamma_{\sC_2\cup \sC_3})$, we deduce $w_3 \L(\Ga_{\sC_1\cup \sC_2\cup \sC_{3} }) w_3^*= \L(\Lambda_{1,2} C_\Lambda(\Lambda_{1,2})\vee \Lambda_{2,3} C_\Lambda(\Lambda_{2,3}))$. Continuing in this fashion we get the desired conclusion by induction.
\end{proof}

\begin{rem} We observe that if in item 3. one could better control where the element $s$ is located then the statement will actually imply superrigidity of these groups. For example, this is the case if s can be picked in the subgroups $\La_1$ or $\La_n$. This however seems difficult to establish at this time. We also note that the control of the elements $s$ is closely related with relating the consecutive unitaries $w_i$'s which as we have seen in the previous results is key in the reconstruction problem.
\end{rem}
From a different perspective, if one assume a priori that the mysterious subgroup $\La$ has a non-trivial graph product structure the analysis can be enhanced and a reconstruction statement can be obtained. Moreover, in this case we do not even need to use the full conclusion of Theorem \ref{theorem.almost.Wsuperrigidity}.  Before proceeding to the result we need three elementary lemmas. 

\begin{lem}\label{finiterange} Let $A,B$ countable groups and let $C\subset \mathscr U(\L(A))$ be a countable subgroup of unitaries. Assume there exists a unitary $x\in \L(A\times B)$ such that $\mathbb T xCx^*< \mathbb T (A\times B)$. Then there exists a finite index subgroup $C_0\leqslant C$ such that  $\mathbb T xC_0x^*< \mathbb T A$.

\end{lem}
\begin{proof} From the assumption there exist group homomorphisms $\alpha: C\ra A$, $\beta: C\ra B$ and a character $\mu\in C\to\mathbb T$ such that for all $c\in C$ we have \begin{equation}\label{equaldiscrgroups}
    x c x^*= \mu(c) u_{\alpha(c)}\otimes v_{\beta(c)}.
\end{equation} 
We denoted by $\{u_a\}_{a\in A}$ and $\{v_b\}_{b\in B}$ the canonical unitaries that generate $\L(A)$ and $\L(B)$, respectively.
Letting $x= \sum_{b\in B}x_b \otimes v_b$ be the Fourier expansion, where $x_b\in \L(A),b\in B$, and using \eqref{equaldiscrgroups} we see that \begin{equation*}
    \sum_b x_b c \otimes v_b= xc= \mu(c) u_{\alpha(c)}\otimes v_{\beta(c)} x = \sum_b \mu(c) u_{\alpha(c)} x_b \otimes v_{\beta(c) b} = \sum_b \mu(c) u_{\alpha(c)} x_{\beta(c) ^{-1}b} \otimes v_{b},
\end{equation*} 
for any $c\in C$. Identifying the coefficients above we get $x_b c=  \mu(c) u_{\alpha(c)} x_{\beta(c) ^{-1}b}$ for all $b\in B,c\in C$. In particular, we have $\|x_b\|_2=  \| x_{\beta(c) ^{-1}b}\|_2$, for all $b\in B,c\in C$. Thus, if we denote the image group by $B_0 := \beta(C)<B$, we get that $\|x_b\|_2$ is constant on each left coset $B/B_0$. This implies that $B_0$ is finite and letting $C_0:= {\rm ker} (\beta )$, we get the desired conclusion. 
\end{proof}

\begin{lem}\label{lemma.elementary.groups1}
Let $\Gamma_1,\Gamma_2$ be countable group and assume $\Sigma<\Gamma_1\times\Gamma_2$ is a subgroup. 

\noindent
If there is a finite index subgroup $\Gamma_1^0<\Gamma_1$ with $\Gamma_1^0<\Sigma$, then there is a subgroup $\Gamma_2^0<\Gamma_2$ such that $\Gamma_1^0\times\Gamma_2^0<\Sigma$ has finite index. 
\end{lem}

\begin{proof}
Let $\pi:\Sigma\to \Gamma_1/\Gamma_1^0$ be the composition between the restriction to $\Sigma$ of the projection map $\Gamma_1\times\Gamma_2 \ni (g_1,g_2)\to g_1\in\Gamma_1$ with the quotient map $\Gamma_1\to\Gamma_1/\Gamma_1^0$. Since the kernel of $\pi$ equals $\Sigma\cap (\Gamma_1^0\times\Gamma_2)$ and the image of $\pi$ is finite, it follows that $[\Sigma: \Sigma\cap (\Gamma_1^0\times\Gamma_2)]<\infty$. Since $\Gamma_1^0<\Sigma$, we get $[\Sigma: \Gamma_1^0\times(\Sigma\cap \Gamma_2)]<\infty$. 
\end{proof}

For the following lemma we introduce the following ad hoc definition. We say that a countable group $\Gamma$ is {\it virtually prime} if $\Gamma$ does not admit a finite index subgroup $\Gamma_0$ which is generated by two infinite commuting subgroups. The next lemma shows that products of virtually prime groups satisfy a unique prime factorization result.

\begin{lem}\label{lemma.UPF.groups}
Let $\Gamma=\Gamma_1\times\dots\times\Gamma_n$ be a product of countable groups that are virtually prime.

\noindent If $\Sigma_1\times\Sigma_2<\Gamma$ is a finite index product subgroup, then there exist finite index subgroups $\Sigma^0_i<\Sigma_i, 1\leq i\leq 2$,  and $\Gamma^0_j<\Gamma_j$, $1\leq j\leq n$, and
a partition $T_1\sqcup T_2=\{1,\dots,n\}$
such that $\Sigma_i^0=\times_{j\in T_i}\Gamma_j^0$, for any $1\leq i\leq 2$.

\end{lem}

\begin{proof}
For any $j\in\overline{1,n}$, let $\pi_j:\Gamma\to\Gamma_j$ be the canonical projection. Since $\Gamma_j$ is virtually prime and $\pi_j(\Sigma_1)$ and $\pi_j(\Sigma_2)$ are commuting subgroups of $\Gamma_j$, it follows that either $\pi_j(\Sigma_1)$ or $\pi_j(\Sigma_2)$ is finite. Hence, there exist finite index subgroups $\Sigma_1^0<\Sigma_1$ and $ \Sigma_2^0<\Sigma_2$ and a map $\sigma:\{1,...,n\}\to\{1,2\}$ such that  $\pi_j(\Sigma_{\sigma(j)}^0)<\Gamma_j$ has finite index and $\pi_j(\Sigma^0_{\sigma(j)+1})=1$ for any $1\leq j\leq n$. Here, we use the notation that $\Sigma_{3}=\Sigma_1$. By letting $T_1=\sigma^{-1}(1), T_2=\sigma^{-1}(2)$, we derive that $\Sigma_i^0=\times_{j\in T_i} \pi_j(\Sigma^0_{i})$, for any $1\leq i\leq 2$, which ends the proof.
\end{proof}

With these preparations at hand we are now ready to state and prove our main result.

\begin{thm}\label{superwithincat} Let $\sG\in{\rm CC}_1$, let ${\rm cliq}(\sG)=\{ \sC_1, \ldots , \sC_n\}$ be a consecutive cliques enumeration and assume that $|\sC_i|\neq |\sC_j|$ whenever $i\neq j$.  Let $\Gamma= \sG\{\Gamma_v\}$, be a graph product group such that for any $v\in\mathscr V$, $\Gamma_v\in \mathcal W\mathcal R (A_v, B_v \ca I_v)$ where $A_v$ is abelian, $B_v$ is an icc  subgroup of a hyperbolic group and the set $\{i\in I_v \; | \; g\cdot i\neq i\}$ is infinite for any $g\in B_v\setminus \{1\}$.

\noindent 
Let $\theta: \L(\Gamma)\ra \L(\Lambda)$ be any $\ast$-isomorphism where
$\Lambda$ is any non-trivial graph product group whose vertex groups are infinite. 

\noindent
Then one can find a character $\eta : \Gamma\to\mathbb T$, a group isomorphism  $\delta :\Gamma \to \Lambda$, an automorphism of $\L(\Lambda)$ of the form $\phi_{a,b}$ (see the notation after equation \eqref{branchaut'} in Section 2) and a unitary $u\in \L(\Lambda)$ such that $\theta= {\rm ad}(u)\circ \phi_{a,b}\circ \Psi_{\eta, \delta }$.    

\end{thm}

\begin{proof} Assume that $\La= \La_\sH$ and let ${\rm cliq}(\sH)=\{\sD_1,\ldots ,\sD_m\}$ be the cliques of $\sH$. To simplify the notation assume that $\L(\Gamma )=\L(\La)=\M$. Then by Theorem \ref{theorem.almost.Wsuperrigidity}, for any $1\leq i\leq n$ one can find a unitary $w_i\in \M$ and a subgroup $\Lambda_i<\Lambda$ satisfying the following relations: \begin{enumerate}
    \item $\mathbb T w_i \Gamma_{\sC_i} w_i^* =\mathbb T \Lambda_i$ for all $1\leq i\leq n$; 
    \item The virtual centralizers satisfy $vC_\La(\La_i)= 1$ for all, $1\leq i\leq n$.
    \item $\mathbb T w_i \Gamma_{\sC_{i,i+1}} w_i^*= \mathbb T w_{i+1} \Gamma_{\sC_{i,i+1}} w_{i+1}^*= \mathbb T \Lambda_{i,i+1}$ where $\Lambda_{i,i+1}= \Lambda_i\cap \Lambda_{i+1}$ for all $1\leq i\leq n-1$;
    \item There is an $s\in \Lambda$ such that $\mathbb T w_n \Gamma_{\sC_{n,1}} w_n^*= \mathbb T v_s w_{1} \Gamma_{\sC_{n,1}} w_{1}^*= \mathbb T \Lambda_{n}\cap s \Lambda_1 s^{-1}$;
\end{enumerate}

Items 1., 3. and 4. are clear. We only have to justify 2. For seeing this, note that Lemma \ref{QN2} and Theorem \ref{controlquasinormalizer1} give $\L({\rm QN}_\Lambda(\Lambda_i))=\mathscr{QN}_{\M}(\L(\Lambda_i))''= w_i \mathscr{QN}_{\M}(\L(\Gamma_{\sC_i}))''  w_i^*= w_i \L(\Gamma_{\sC_i}) w_i^*=\L(\Lambda_i)$.
Since $vC_{\Lambda}(\Lambda_i) < {\rm QN}_\Lambda(\Lambda_i)$, we deduce that $vC_{\Lambda}(\Lambda_i)<\Lambda_i$, and hence, by using that $\Lambda_i$ is icc we derive  $vC_{\Lambda}(\Lambda_i)=vC_{\Lambda_i}(\Lambda_i)=1$.

\begin{claim}\label{map}
There exist a map $\sigma:\{1,\dots,n\}\to \{1,\dots,m\}$ and elements $h_1,\dots,h_n \in \La$ so that $\La_i\leqslant h_i\La_{\sD_{\sigma(i)}} h^{-1}_i$, for any $1\leq i\leq n$.
\end{claim}

\noindent {\it Proof of Claim \ref{map}.}
Let $1\leq i\leq n$ be arbitrary.
Since $\Gamma_{\sC_i}$ has property (T), then $\La_i< \La$ has property (T) as well. By applying Theorem \ref{theorem.embed.clique} one can find $1\leq k\leq m$ such that $\L(\Lambda_i)\prec_{\M} \L(\Lambda_{\sD_k}).$
Using Lemma \ref{lemma.CI17}, there is $h\in \La$ such that $[\La_i: \La_i \cap h \La_{\sD_k} h^{-1}]<\infty$. Thus, there is a finite index normal subgroup $\La_i'\lhd \La_i$ such that $\La'_i\leqslant h\La_{\sD_k} h^{-1}$. We continue by proving that $\La_i\leqslant h\La_{\sD_k} h^{-1}$. To this end, fix $t\in \La_i$. Hence, $\La_i'\leqslant  h\La_{\sD_k} h^{-1}  \cap  th\La_{\sD_k} h^{-1}t^{-1}$ and by Proposition \ref{proposition.AM10} we have that  
\begin{equation}\label{periphcont}h\La_{\sD_k} h^{-1}  \cap  th\La_{\sD_k} h^{-1}t^{-1}= h_0 h\La_\sT h^{-1} h_0^{-1}\text{ for some }\sT\subseteq \sD_k \text{ and }h_0\in h\La_{\sD_k} h^{-1}.
\end{equation} 
Since $[\Lambda_i:\Lambda_i']<\infty$, condition 2.\ implies that $vC_{\La}(\La_i')=vC_{\La}(\La_i)=1$. 
It follows from \eqref{periphcont} that $vC_{\Lambda}(h_0 h\La_\sT h^{-1} h_0^{-1})=1$, and hence, $\sT$ is a clique. This forces that $\sT=\sD_k$, and hence, \eqref{periphcont} implies that $th\La_{\sD_k} h^{-1}t^{-1} \geqslant  h\La_{\sD_k} h^{-1}$. Using Proposition \ref{proposition.AM10} we derive that  $th\La_{\sD_k} h^{-1}t^{-1} =  h\La_{\sD_k} h^{-1}$, and hence, $h^{-1}t h\in \Lambda_{\sD_k\cup {\rm link}(\sD_k)}$. Since $\sD_k$ is a clique,  we conclude that $t\in h\La_{\sD_k} h^{-1}$. \hfill$\blacksquare$


\noindent
We continue with the following claim.
\begin{claim}\label{commensurable}
For every $1\leq i\leq n$, we have  $[h_i\La_{\sD_{\sigma(i)}} h^{-1}_i: \Lambda_i]<\infty$.
\end{claim}

\noindent {\it Proof of Claim \ref{commensurable}.} 
Fix $1\leq i\leq n$ and $v\in \sC_i^{\rm int}$ and note that ${\rm star}(v)=\sC_i$.
Let $j=\sigma(i)$. Next we briefly argue that $|\mathscr D_j|\geqslant 2$. Assume by contradiction that $\mathscr D_j=\{s\}$. As $\La$ is a nontrivial graph product, there is a nontrivial group $\Sigma$ so that  $\La=\La_s \ast \Sigma$; thus $\M = \L(\La_s\ast \Sigma )$. Since $\L(\Gamma_{\mathscr C_{i,i+1}})$ is diffuse then using Claim \ref{map} and relation 1.\ we get $x \L(\Gamma_{\mathscr C_{i,i+1}}) x^*\subseteq \L(\La_s)$ where we have denoted by $x:= u_{h_i^{-1}}w_i\in \mathscr U(\M)$. Using Lemma \ref{lemma.control.qn} and the graph product structure of $\Gamma$ this further implies that  $x \L(\Gamma_{\mathscr C_{i}\cup \mathscr C_{i+1}}) x^*= \mathscr N_{\M} (x \L(\Gamma_{\mathscr C_{i,i+1}}) x^*)''\subseteq \L(\La_s)$. In particular, we have  $ x \L(\Gamma_{\mathscr C_{i+1,i+2}}) x^*\subseteq \L(\La_s)$. Thus taking the normalizer and repeating the previous argument  we get $ x \L(\Gamma_{\mathscr C_{i+1}\cup \mathscr C_{i+2}}) x^*= \mathscr N_{\M} (x \L(\Gamma_{\mathscr C_{i+1,i+2}}) x^*)''\subseteq \L(\La_s)$. Therefore, proceeding by induction we have  $ x \L(\Gamma_{\mathscr C_{k}\cup \mathscr C_{l}}) x^*\subseteq \L(\La_s)$ for all $1\leqslant k,l\leqslant m$ satisfying $\hat k-\hat l \in \{ \hat 1,\widehat{n-1}\}$; here the classes are considered in $\mathbb Z_m$.
Altogether, these relations imply that $ \M = x_i \L(\Gamma) x^*= x \L(\vee_{k,l, \hat k-\hat l \in \{\hat 1,\widehat{n-1}\}} \Gamma_{\mathscr C_{k}\cup \mathscr C_{l}}) x^*=\bigvee_{k,l, \hat k-\hat l \in \{\hat 1,\widehat{n-1}\}}x \L(\Gamma_{\mathscr C_{k}\cup \mathscr C_{l}}) x^*\subseteq \L(\La_s) $, contradicting $\Sigma\neq 1$.

Now fix an arbitrary $w\in \sD_j$. By letting $\P_1= \L(\La_{\sD_{j} \setminus \{w\}})$ and $\P_2= \L(\La_w)$, we have $\P_1\bar\otimes \P_2=\L(\La_{\sD_j})$. Since the vertex groups of $\La$ are infinite, then by \cite[Theorem 3.1]{CdSS17} one of the following must hold:

\begin{enumerate}
    \item $\P_1\vee \P_2$ is amenable relative to $ \L(\Gamma_{{\rm link}(v) })$ inside $\M$;

    \item $\P_1\vee \P_2 \prec_\M \L(\Gamma_{\sG\setminus \{v\}})$;

    \item $\P_1\vee \P_2 \prec_\M \L(\Gamma_{{\rm star}(v)})$;

    \item $\P_i \prec_\M \L(\Gamma_{{\rm link}(v) })$, for some $1 \leq i\leq 2$;    

\end{enumerate}

First, assume 1. holds. Since $\mathbb T h_i^{-1} w_i \Gamma_{\sC_i}  w_i^*  h_i=\mathbb T  h_i^{-1} \Lambda_i h_i\subset \P_1\vee \P_2$, we derive that $h_i^{-1} w_i \L(\Gamma_{\sC_i})  w_i^*  h_i$ is amenable relative to $ \L(\Gamma_{{\rm link}(v) })$ inside $\M$. Since $\L(\Gamma_{\sC_i})$ has property (T), then  $h_i^{-1} w_i \L(\Gamma_{\sC_i})  w_i^*  h_i\prec_{\M} \L (\Gamma_{{\rm link}(v)})$. By applying Corollary \ref{corollary.graph.CI17}, it follows that $\sC_i\subset {\rm  link}(v)$, contradiction.

Now, assume 2. holds. Since $\mathbb T h_i^{-1} w_i \Gamma_{\sC_i}  w_i^*  h_i=\mathbb T  h_i^{-1} \Lambda_i h_i\subset \P_1\vee \P_2$, we derive that $\L(\Gamma_{\sC_i})\prec_\M \L (\Gamma_{\sG\setminus \{v\}})$.
As in the previous case, we apply Corollary \ref{corollary.graph.CI17} and derive that $\sC_i\subset \sG \setminus\{v\}$, contradiction.

Next, assume 3. holds. As ${\rm star}(v)=\sC_i$ and $\mathbb T w_i \Gamma_{\sC_i}  w_i^* =\mathbb T\Lambda_i $, we derive that $\L (\Lambda_{\sD_j})\prec_\M \L (h_i^{-1} \Lambda_i  h_i)$. 
This implies that there exist projections $p\in \L (\Lambda_{\sD_j}), q\in \L (h_i^{-1} \Lambda_i  h_i)$, a non-zero partial isometry $w\in q \M p$ and a $*$-homomorphism $\theta: p\L (\Lambda_{\sD_j})p\to q\L (h_i^{-1} \Lambda_i  h_i)q$ such that $\theta(x)w=wx$. Since $\La_i\leqslant h_i\La_{\sD_{j}} h^{-1}_i$, it follows that $w p\L (\Lambda_{\sD_j})p\subset  q \L (\Lambda_{\sD_j})qw$. Since $\D_j$ is a clique, it follows from Lemma \ref{lemma.control.qn} that $w\in \L (\Lambda_{\sD_j})$.
This shows that $\L (\Lambda_{\sD_j})\prec_{\L (\Lambda_{\sD_j})} \L (h_i^{-1} \Lambda_i  h_i)$, which further implies that $[\Lambda_{\sD_j}: h_i^{-1} \Lambda_i  h_i]<\infty$. This proves the claim.

Before completing Claim \ref{commensurable} by
 assuming that 4. holds, we first establish the following notation and prove the following subclaim. Since $\mathbb T w_i \Gamma_{\sC_i} w_i^*=\mathbb T \Lambda_i$, we define the subgroups $\Lambda^i_v, \Lambda^i_{\hat v}<\Lambda_i$ by  $\mathbb T w_i \Gamma_{v} w_i^*=\mathbb T \Lambda^i_{ v}$ and $\mathbb T w_i \Gamma_{{\rm link}(v)} w_i^*=\mathbb T \Lambda^i_{\hat v}$, respectively. Note that $\Lambda_i= \Lambda^i_{v}\times\Lambda^i_{\hat v}$ since $v\in \sC_i^{\rm int}$.

\noindent
{\bf Subclaim 1.} For any subset $\sD_0\subset \sD_j$  satisfying $\L(\Lambda_{\sD_0})\prec_{\M} \L (\Gamma_{{\rm link}(v)})$, the subgroup $(\Lambda_0)_{\sD_0}:= h_i \Lambda_{\sD_0} h_i^{-1}\cap \Lambda^i_{\hat v} $ satisfies  $[h_i \Lambda_{\sD_0} h_i^{-1}:(\Lambda_0)_{\sD_0}]<\infty$.

\noindent {\it Proof of Subclaim 1.}
For proving the subclaim, we consider a subset $\sD_0\subset \sD_j$ satisfying $\L(\Lambda_{\sD_0})\prec_{\M} \L (\Gamma_{{\rm link}(v)})$. By applying Lemma \ref{lemma.CI17} there is $k\in\Lambda$ such that
\begin{equation}\label{d1}
[h_i \Lambda_{\sD_0} h_i^{-1}: h_i \Lambda_{\sD_0} h_i^{-1}\cap k \Lambda^i_{\hat v} k^{-1}]<\infty.    
\end{equation}

Set $\Lambda_0= k^{-1}h_i \Lambda_{\sD_0} h_i^{-1}k\cap  \Lambda^i_{\hat v}$.
We continue by showing that $k$ normalizes $h_i \Lambda_{\sD_0}  h_i^{-1}$. Recall that $\Lambda^i_{\hat v}<\Lambda_i< h_i \Lambda_{\sD_j} h_i^{-1}$. From \eqref{d1} we get that $[h_i \Lambda_{\sD_0} h_i^{-1}: k \Lambda_0 k^{-1}]<\infty$, and consequently, there is a finite subset $S\subset \Lambda$ satisfying $(h_i \Lambda_{\sD_0} h_i^{-1})k\subset \cup_{s\in S} sk \Lambda_0 \subset \cup_{s\in S} sk h_i \Lambda_{\sD_j}h_i^{-1}$. By applying Theorem \ref{controlquasinormalizer1} we derive that $h_i^{-1}k h_i\in \Lambda_{\sD_j\cup {\rm link}(\sD_0)}$, which implies that $k$ is normalizing $h_i \Lambda_{\sD_0}  h_i^{-1}$. In combination with \eqref{d1}, this shows that $\Lambda_0=(\Lambda_0)_{\sD_0}$ is a finite index subgroup of $h_i \Lambda_{\sD_0} h_i^{-1}$. This proves the subclaim.
\hfill$\blacksquare$



Now, we can assume that 4. holds for all 
$w\in \sD_j$.
Namely, one of the following holds: a) $\P_1 \prec_\M \L(\Gamma_{{\rm link}(v) })$ or  b) $\P_2 \prec_\M \L(\Gamma_{{\rm link}(v) })$. We continue by considering the following cases.

\noindent
{\bf Case 1.} Suppose there are $w_1\neq w_2\in \sD_j$ such that a) holds. The subclaim implies that $[h_i \Lambda_{\sD_j\setminus\{w_1\}} h_i^{-1}:(\Lambda_0)_{\sD_j\setminus\{w_1\}}]<\infty$ and $[h_i \Lambda_{\sD_j\setminus\{w_2\}} h_i^{-1}:(\Lambda_0)_{\sD_j\setminus\{w_2\}}]<\infty$. 
Since $w_1\in \sD_j\setminus\{w_2\}$, it follows that $[h_i\Lambda_{w_1}h_i^{-1}: h_i\Lambda_{w_1}h_i^{-1}\cap \Lambda^i_{\hat v}]<\infty$. Consequently, the inclusion $(h_i\Lambda_{w_1}h_i^{-1}\cap \Lambda^i_{\hat v})\vee (\Lambda_0)_{\sD_j\setminus\{w_1\}}<\Lambda_i<h_i \sD_j h_i^{-1}$
is of finite index. This proves the claim.

\noindent
{\bf Case 2.} Next, we suppose that $b)$ holds for any $w\in \sD_j$. Using the subclaim we get that $(\Lambda_0)_{w}<h_i \Lambda_{w}  h_i^{-1}$ has finite index for any $w\in\sD_j$. Hence, the inclusion $\vee_{w\in\sD_j}(\Lambda_0)_{w}<\Lambda_i<\vee_{w\in\sD_j}h_i \Lambda_{w}  h_i^{-1}=h_i \Lambda_{\sD_j} h_i ^{-1}$ has finite index as well, which again proves the claim.

\noindent
{\bf Case 3.} Finally, assume that there is $w_0\in \sD_j$ such that a) holds and for any $w\in \sD_j\setminus\{w_0\}$ we have that b) holds. We note in passing that in this situation b) does not provide any additional information besides a). 
Note that Subclaim 1 provides  
\begin{equation}\label{r1}
    [h_i \Lambda_{\sD_j\setminus\{w_0\}} h_i^{-1}:(\Lambda_0)_{\sD_j\setminus\{w_0\}}]<\infty
\end{equation}

\noindent We continue by showing the following.

\noindent {\bf Subclaim 2.} The group $(\Lambda_0)_{\sD_j\setminus\{w_0\}}$ is  icc and $ vC_{h_i\La_{\mathscr D_j \setminus \{w_0\}}h_i^{-1}  }((\Lambda_0)_{\sD_j\setminus\{w_0\}})=1$. Also, there exists a subgroup $\Upsilon < h_i \La_{w_0} h_i^{-1}$ satisfying  $[\La_i: (\Lambda_0)_{\sD_j\setminus\{w_0\}}\times \Upsilon]<\infty$.

\noindent {\it Proof of Subclaim 2.}
Using  $(\Lambda_0)_{\sD_j\setminus\{w_0\}}<\Lambda^i_{\hat v}<h_i \Lambda_{\sD_j\setminus\{w_0\}} h_i^{-1}\times h_i \Lambda_{w_0} h_i^{-1}$, we apply Lemma \ref{lemma.elementary.groups1} and deduce that there is a subgroup $\Lambda'< h_i \Lambda_{w_0} h_i^{-1}\cap\Lambda_{\hat v}^i$ such that 
\begin{equation}\label{w1}
[\Lambda_{\hat v}^i: (\Lambda_0)_{\sD_j\setminus\{w_0\}}\times \Lambda']<\infty.    
\end{equation}
Since $\Lambda_{\hat v}^i$ is icc, \eqref{w1} implies that $(\Lambda_0)_{\sD_j\setminus\{w_0\}}$ is icc as well. 
Next, we argue that there is a finite index subgroup $\La'' \leqslant h_i \Lambda_{w_0} h_i^{-1}\cap \La_v^i$ such that 
\begin{equation}\label{w2}
[\Lambda^i_v: \Lambda'']<\infty
\end{equation}
Indeed, as $(\Lambda_0)_{\sD_j\setminus\{w_0\}}$ is icc, one can see from \eqref{r1} that $C_{h_i \Lambda_{\sD_j\setminus\{w_0\}} h_i^{-1}}((\Lambda_0)_{\sD_j\setminus\{w_0\}})=1$, and thus,
$\La_v^i\leqslant C_{h_i\La_{\mathscr D_j}h_i^{-1}}(\La^i_{\hat v} ) \leqslant C_{h_i \La_{\mathscr D_j}h_i^{-1}}((\Lambda_0)_{\sD_j\setminus\{w_0\}} )  \leqslant h_i \La_{w_0}h_i^{-1}$. This shows that we can actually choose $\Lambda''=\Lambda_v^i$ to satisfy \eqref{w2}.

Altogether, \eqref{w1} and \eqref{w2} imply the existence of $\Upsilon := \La'\times \La''< h_i \La_{w_0} h_i^{-1}$ such that  $[\La_i: (\Lambda_0)_{\sD_j\setminus\{w_0\}}\times \Upsilon]<\infty $. Together with condition 2., we deduce that
$1= vC_{\La}(\La_i)= vC_{\La }((\Lambda_0)_{\sD_j\setminus\{w_0\}}\times \Upsilon)\geq vC_{h_i\La_{\mathscr D_j \setminus \{w_0\}}h_i^{-1}  }((\Lambda_0)_{\sD_j\setminus\{w_0\}})\times vC_{h_i\La_{w_0}h_i^{-1} }( \Upsilon)$.
In particular, this shows that $ vC_{h_i\La_{\mathscr D_j \setminus \{w_0\}}h_i^{-1}  }((\Lambda_0)_{\sD_j\setminus\{w_0\}})=1.$
\hfill$\blacksquare$

We continue by applying Lemma \ref{lemma.UPF.groups} and find a subgraph $v\in \mathscr C^1_i\subset \mathscr C_i$, a finite index subgroup $\Upsilon_0\leqslant\Upsilon$ and finite index subgroups $\Ga'_w\leqslant \Ga_w$, for any $w\in \mathscr C^1_i$, such that if $\Ga'_{\mathscr C_i^1}:= \times_{w\in \mathscr C_i^1} \Ga'_w$, then  
\begin{equation}\label{preimage}\mathbb T w_i \Ga'_{\mathscr C_i^1} w_i^*=\mathbb T \Upsilon_0. 
\end{equation}  

\noindent By Lemma \ref{lemma.UPF.groups} we consider some finite index subgroups $\Ga'_{\mathscr C_i \setminus \mathscr C_i^1}\leqslant \Ga_{\mathscr C_i\setminus \mathscr C_i^1}$ and $(\Lambda^0_0)_{\sD_j\setminus\{w_0\}}< (\Lambda_0)_{\sD_j\setminus\{w_0\}}$ such that 
\begin{equation}\label{complementpreimage}
  \mathbb T w_i \Ga'_{\mathscr C_i \setminus\mathscr C_i^1} w_i^*=\mathbb T (\Lambda^0_0)_{\sD_j\setminus\{w_0\}}.  
\end{equation}

Since $\Ga'_{\mathscr C_i^1}\leqslant \Ga_{\mathscr C_i^1}$ has finite index and $v\in \sC_i^1\cap \sC_i^{\rm int}$, we obtain from Theorem \ref{controlquasinormalizer1} that  $\L(\Ga_{\mathscr C_i \setminus \mathscr C_i^1})= \L(\Ga'_{\mathscr C_i^1})' \cap \M$. Hence, by taking relative commutants in \eqref{preimage} and using $\Upsilon_0<\Upsilon< h_i\Lambda_{w_0}h_i^{-1}$, we deduce that $w_i \L(\Ga_{\mathscr C_i \setminus \mathscr C_i^1}) w_i^*= w_i \L(\Ga'_{\mathscr C_i^1})' w_i^*\cap \M = \L(\Upsilon_0)'\cap \M \supseteq \L(h_i\La_{{\rm link}(w_0)}h_i^{-1}) $. Since $(\Lambda_0)_{\sD_j\setminus\{w_0\}}\leqslant  h_i \Lambda_{\sD_j\setminus\{w_0\}} h_i^{-1}$, then relation \eqref{complementpreimage} entails $w_i \L(\Ga'_{\mathscr C_i \setminus \mathscr C_i^1}) w_i^*\subseteq \L(h_i \La_{\sD_j }h_i^{-1})$, and hence,  by combining this with the prior containment, it follows that $\L(h_i\Lambda_{{\rm link}(w_0)}h_i^{-1})\prec_{\M} \L(h_i \La_{\sD_j }h_i^{-1})$. Using
Corollary \ref{corollary.graph.CI17} we further deduce that ${\rm link}(w_0)\subset \mathscr D_j$. 

Next, we claim that 
\begin{equation}\label{int}\mathscr C^1_i\subseteq \sC_i^{\rm int}.\end{equation} 
Assume by contradiction that \eqref{int} does not hold. Without loss of generality,
we may assume there is $v_0\in  \mathscr C^1_i\cap \mathscr C_{i+1}$. Note that $\mathbb T w_i \Ga'_{v_0}
w_i^*= \mathbb T \Upsilon_1$ for some subgroup  $\Upsilon_1< \Upsilon_0< h_i \La_{w_o}h_i^{-1}$, and hence, $ w_i\L(\Ga'_{v_0})
w_i^*= \L(\Upsilon_1)$. By taking relative commutants and using  Corollary \ref{controlintertw5} and ${\rm link}(w_0)=\sD_j\setminus\{w_0\}$, we see that \begin{equation}\label{eqcomm}\begin{split} 
\Q:=w_i\L(\Ga_{{\rm link}(v_0) })w_i^*&= (w_i\L(\Ga'_{v_0})
w_i^*)'\cap \M = \L(\Upsilon_1)'\cap \M\\
&= \L(h_i\La_{\sD_j\setminus \{w_0\}} h_i^{-1})\bar\otimes (\L(\Upsilon_1)'\cap \L(h_i \La_{w_0}h_i^{-1}))\end{split}
\end{equation} 
Since ${\rm link}(v_0)=(\sC_i\cup \sC_{i+1})\setminus\{v_0\}$,
we note that  
$$
\Q=w_i\L(\Ga_{{\rm link}(v_0) })w_i^*=   w_i\L(\Ga_{\mathscr{C}_{i,i+1}\setminus \{v_0\} })w_i^* \bar\otimes (w_i\L(\Ga_{\mathscr C_i \setminus \mathscr{C}_{i,i+1}})w_i^* \ast w_i \L(\Ga_{\mathscr{C}_{i+1} \setminus \mathscr{C}_{i,i+1}}) w_i^*).$$ 
Since $ \L (h_i\La_{\mathscr  D_j\setminus \{w_0\}} h_i^{-1})$ is a regular property (T) von Neumann subalgebra of $\Q$, 
it follows from the main technical result of \cite{IPP05} (see also \cite[Theorem 5.4]{PV09}) that
$\L (h_i\La_{\mathscr  D_j\setminus \{w_0\}} h_i^{-1})\prec_\Q w_i\L(\Ga_{\mathscr C_{i,i+1}\setminus \{v_0\} })w_i^*$. Thus, relation \eqref{complementpreimage} and Corollary \ref{corollary.graph.CI17} entail 
\begin{equation}\label{pp1}
\mathscr{C}_{i} \setminus \mathscr{C}^{1}_{i}\subseteq \mathscr{C}_{i,i+1}\setminus \{v_0\}.    
\end{equation}
Using \eqref{complementpreimage} we have $w_i \L(\Ga'_{\mathscr{C}_{i} \setminus\mathscr{C}_{i}^{1}} ) w_i^*=\L((\Lambda_0^0)_{\sD_j\setminus\{w_0\}})$. By taking relative commutants in \eqref{eqcomm} inside $\Q$ and using Subclaim 2 and \eqref{pp1}, we further get 
\begin{equation}\label{pp11}
w_i\L(\Ga_{\mathscr{C}_{i}\cup \mathscr{C}_{i+1} \setminus(\mathscr{C}_{i}\setminus \mathscr {C}^{1}_{i} )})w_i^*=\L(\Upsilon_1)'\cap \L(h_i \La_{w_0}h_i^{-1})\subset \L(h_i \La_{w_0}h_i^{-1}).    
\end{equation}
 
In particular, we have 
\begin{equation}\label{gremb1} w_i\L(\Ga_{\mathscr C_{i+1} \setminus \mathscr C_{i,i+1} })w_i^*\subset \L(h_i \La_{w_0}h_i^{-1}).\end{equation} 
Since we also have  $\mathbb T w_{i+1}\Ga_{\mathscr C_{i+1} \setminus\mathscr C_{i,i+1} }w_{i+1}^*< \mathbb T h_{i+1}\La_{\mathscr D_{\sigma(i+1)}} h_{i+1}^{-1} $, we obtain that $w_0\in \mathscr D_{\sigma(i+1)}$. Indeed, by \cite[Lemma 2.7]{Va10a}, there is $\lambda\in \Lambda$ such that $\L(\Ga_{\mathscr C_{i+1} \setminus\mathscr C_{i,i+1} })\prec_{\M}^s \L (\lambda \Lambda_{w_0} \lambda^{-1} \cap \Lambda_{\sD_{\sigma(i+1)}} )$, ane hence, $\lambda \Lambda_{w_0} \lambda^{-1} \cap \Lambda_{\sD_{\sigma(i+1)}}$ is an infinite group. By Proposition \ref{proposition.AM10} we get that $w_0\in \sD_{\sigma(i+1)}$.
Now, using that ${\rm link}(w_0)\subset \sD_j$, we get $\sigma(i+1)=j$. In particular, we have that $\mathbb T w_{i+1}\Ga_{\mathscr C_{i+1} \setminus\mathscr C_{i,i+1} }w_{i+1}^*=\mathbb T \La_{i+1}< \mathbb T h_{i+1}\La_{\mathscr D_{j}} h_{i+1}^{-1}.$ This further implies that  \begin{equation}\label{gremb2}\mathbb T h_i h^{-1}_{i+1}w_{i+1} w_i^*( w_i\Ga_{\mathscr C_{i+1} \setminus\mathscr C_{i,i+1} } w_i^*) w_i w_{i+1}^* h_{i+1} h_i^{-1}< \mathbb T h_{i}\La_{\mathscr D_{j}} h_{i}^{-1}.
\end{equation}

Using  Theorem \ref{controlquasinormalizer1}, \eqref{gremb1} and \eqref{gremb2} further imply that  $x:=h_i h^{-1}_{i+1}w_{i+1} w_i^*\in \mathscr U(\L(h_i \La_{\mathscr D_j} h_i^{-1}))$. Thus, using Lemma \ref{finiterange}, \eqref{gremb1} and \eqref{gremb2} further entail the existence of a finite index subgroup $\Ga''_{\mathscr C_{i+1} \setminus\mathscr C_{i,i+1}}\leqslant\Ga_{\mathscr{C}_{i+1} \setminus\mathscr{C}_{i,i+1}}$ such that  $\mathbb T x( w_i\Ga''_{\mathscr C_{i+1} \setminus\mathscr C_{i,i+1} } w_i^*) x^* < \mathbb T h_{i}\La_{w_0} h_{i}^{-1}$. In particular, there exist a finite index subgroup $\Ga''_{\mathscr C_{i+1,i+2}}\leqslant\Ga_{\mathscr C_{i+1,i+2}}$ and a subgroup $\Lambda_2<h_i \Lambda_{w_0}  h_i^{-1}$  such that 
$\mathbb T x( w_i\Ga''_{\mathscr C_{i+1,i+2}} w_i^*) x^* =\mathbb T \La_2$. Taking relative commutants and using Corollary \ref{controlintertw5} we get \begin{equation}\label{equalcomm3}x\L(\Ga_{\mathscr C_{i+1}\cup \mathscr C_{i+2} \setminus \mathscr C_{i+1,i+2}})x^*= \L(h_i \La_{\mathscr D_j\setminus \{w_0\}} h_i^{-1})\bar \otimes( \L(\La_2)'\cap \L(h_i\La_{w_0}h_i^{-1})).
\end{equation} 
However, we canonically have the free product decomposition 
\begin{equation}\label{pp111}
x\L(\Ga_{\mathscr{C}_{i+1}\cup \mathscr{C}_{i+2} \setminus \mathscr {C}_{i+1,i+2}})x^*= x\L(\Ga_{\mathscr{C}_{i+1}\setminus \mathscr{C}_{i+1,i+2}}) x^*\ast x\L(\Ga_{\mathscr{C}_{i+2} \setminus \mathscr{C}_{i+1,i+2}})x^*    
\end{equation}
From \eqref{equalcomm3} we derive that $\L(h_i \La_{\mathscr D_j\setminus \{w_0\}} h_i^{-1})$ is a regular property (T) von Neumann subalgebra of $x\L(\Ga_{\mathscr{C}_{i+1}\cup \mathscr{C}_{i+2} \setminus \mathscr {C}_{i+1,i+2}})x^*$. Hence, we obtain a contradiction by applying   \cite[Theorem 5.4]{PV09} (see also \cite{IPP05}) to relation \eqref{pp111}.
Thus, relation \eqref{int} holds.

Using \eqref{int} and taking relative commutants in  \eqref{complementpreimage} we obtain that $w_i \L(\Ga_{\mathscr C_i^1})w_i^*= \L (h_i \La_{w_0}h_i^{-1})$. Combining this with \eqref{preimage} we conclude that $[h_i \La_{w_0} h_i^{-1}: \Upsilon_0]<\infty$. Thus, $[h_i\Lambda_{\sD_j} h_i^{-1}: (\Lambda_0)_{\sD_j\setminus\{w_0\}}\times\Upsilon_0 ]<\infty$. Since $ (\Lambda_0)_{\sD_j\setminus\{w_0\}}\times\Upsilon_0 <\Lambda_i< h_i\Lambda_{\sD_j} h_i^{-1}$, we conclude the proof of our claim in case 3. 
\hfill$\blacksquare$

\noindent Next we prove the following claim.

\begin{claim}\label{fullidentification}
For any $1\leq i\leq n$,
there exists a unitary $y_i \in \M$ such that $\mathbb T y_i \Gamma_{\mathscr C_i} y_i^*= \mathbb T \La_{\mathscr D_{\sigma (i)}}$.

\end{claim}

\noindent \emph{Proof of Claim \ref{fullidentification}}. Fix $1\leq i\leq n$. From Claim \ref{commensurable}  we have  $[h_i\La_{\sD_{\sigma(i)}} h^{-1}_i: \Lambda_i]<\infty$. Recall also that 
$\mathbb T w_i \Ga_{\sC_i} w_i^* =\mathbb T \La_i$, which  clearly implies that $w_i \L(\Ga_{\sC_i}) w_i^* = \L(\La_i)$. Taking quasinormalizers of these algebras and using successively part 2) in Lemma \ref{QN2} we have
\begin{equation*} \begin{split}w_i \L(\Ga_{\mathscr C_i}) w_i^*&= w_i \L({\rm QN}_\Ga(\Ga_{\mathscr C_i})) w_i^*=  w_i \mathscr Q\mathscr N (\L(\Ga_{\sC_i}))'' w_i^*\\&= \mathscr Q\mathscr N (w_i\L(\Ga_{\sC_i})w_i^*)''=  \mathscr{QN}_\M(\L(\La_i))''=   \L({\rm QN}_\La(\La_i)) \\& =\L({\rm QN}_\La(h_i \La_{\mathscr D_{\sigma(i)}}h_i^{-1}))= \L({\rm QN}_\La(h_i \La_{\mathscr D_{\sigma(i)}}h_i^{-1}))=\L(h_i \La_{\mathscr D_{\sigma(i)}}h_i^{-1})\\&= v_{h_i}\L(\La_{\mathscr D_{\sigma(i)}}) v_{h_i^{-1}}.\end{split}\end{equation*} 
Thus the claim follows from the $W^*$-superrigidity of $\Ga_{\mathscr C_i}$, see Corollary \ref{superprod}. $\hfill\blacksquare$
\vskip 0.07in
\noindent In particular, Claim \ref{fullidentification} implies that $y_i\L(\Ga_{\mathscr C_i})y_i^*= \L(\La_{\mathscr D_{\sigma(i)}})$. Then using Theorem \ref{cyclerel1} and the same argument from the proof of \cite[Theorem 7.9]{CDD22} one can find a unitary $u\in \M$ such that $u\L(\Ga_{\mathscr C_i})u^*= \L(\La_{\mathscr D_{\sigma(i)}})$ for all $i$. Thus without any loss of generality we can assume that $\L(\Ga_{\mathscr C_i})= \L(\La_{\mathscr D_{\sigma(i)}})=:\M_i$ for all $i$.  Since $\Ga_{\mathscr C_i}$ is $W^*$-superrigid we can find unitaries $x_i\in \M_i$ such that \begin{equation}\label{wsuper8}\mathbb T x_i \Gamma_{\mathscr C_i} x_i^*= \mathbb T \La_{\mathscr D_{\sigma (i)}}\text{ for all  }i.\end{equation}

These relations already show that $\La = \vee_i \La_{\mathscr D_{\sigma(i)}}$ and also $\sigma: {\rm cliq}(\mathscr G) \ra {\rm cliq}(\mathscr H)$ is a bijection. Moreover, one can see that the graph $\mathscr H\in {\rm CC}_1$. This and relations \eqref{wsuper8} show that $\La$ can be represented as a graph product $\La_{\mathscr H'}$, where the vertex groups $\La_w$ are property (T) wreath-like products as in Theorem \ref{prodrigid1} and the underlying graph $\mathscr H'$ is isometric to $\mathscr G$. Hence, the desired conclusion follows from \cite[Theorem 7.10]{CDD22}. \end{proof}


\begin{thebibliography}{99}\footnotesize

\bibitem[Ag13]{Ag13} I. Agol, \textit{The virtual Haken conjecture}, Doc. Math. \textbf{18} (2013), 1045--1087. With an appendix by I. Agol, D. Groves, J. Manning.

\bibitem[AM10]{AM10} Y. Antol\'in, A. Minasyan,  \textit{Tits alternatives for graph products}, J. Reine Angew. Math. \textbf{704} (2015), 55--83.


\bibitem[Be14]{Be14} M. Berbec, {\it W$^*$-superrigidity for wreath products with groups having positive first $\ell^2$-Betti number}, Internat. J. Math. {\bf 26} (2015), no. 1, 1550003, 27.
\bibitem[BV12]{BV12} M. Berbec, S. Vaes, \textit{$W^*$-superrigidity for group von Neumann algebras of left-right wreath products}, Proc. Lond. Math. Soc. {\bf 108} (2014), no. 5, 1116--1152.



\bibitem[BKKO14]{BKKO14} E. Breuillard, M. Kalantar, M. Kennedy, N. Ozawa, {\it $C^*$-simplicity and the unique trace property for discrete groups}, Publ. Math. Inst. Hautes \'Etudes Sci. {\bf126} (2017), 35--71.

\bibitem[BO08]{BO08} N. P. Brown, N. Ozawa, \textit{$\mathrm{C}^\ast$-algebras and finite-dimensional approximations}, Graduate Studies in Mathematics, \textbf{88}. American Mathematical Society, Providence, RI, 2008.



\bibitem[Ca16]{Ca16}   M. Caspers, \textit{Absence of Cartan subalgebras for right-angled Hecke von Neumann
algebras}, Anal. PDE \textbf{13} (2020), no. 1, 1--28.
\bibitem[CF14]{CF14} M. Caspers, P. Fima, \textit{Graph products of operator algebras}, J. Noncommut. Geom.
\textbf{11} (2017), no. 1, 367--411.

\bibitem[CDHK20]{CDHK20} I. Chifan, S. Das, C. Houdayer, K. Khan, \textit{Examples of property (T) factors with trivial fundamental group}, to appear in Amer. J. Math, arXiv:2003.11729.

\bibitem[CDD22]{CDD22} I. Chifan, M. Davis, D. Drimbe, \textit{Rigidity for von Neumann algebras of  graph product groups. I. Structure of automorphisms}, Preprint 2022, arXiv:2209.12996.



\bibitem[CdSS15]{CdSS15} I. Chifan, R. de Santiago, T. Sinclair, {\it $W^{*}$
-rigidity for the von Neumann algebras of products of hyperbolic
groups}, Geom. Funct. Anal. {\bf 26} (2016), no. 1, 136--159. 

\bibitem[CdSS17]{CdSS17} I. Chifan, R. de Santiago, W. Sucpikarnon, {\it Tensor product decompositions of II$_{1}$ factors arising from
extensions of amalgamated free product groups}, Comm. Math. Phys. {\bf 364} (2018), 1163--1194.
		


\bibitem[CD-AD20]{CD-AD20} I. Chifan, A. Diaz-Arias, D. Drimbe, {\it New examples of $W^*$ and $C^*$
-superrigid groups}, Adv. Math. \textbf{412} (2023), Paper no. 108797, 57pp.
 
  \bibitem[CD-AD21]{CD-AD21} I. Chifan, A. Diaz-Arias, D. Drimbe, \textit{$W^*$  and $C^*$-superrigidity results for coinduced groups}, J. Funct. Anal. \textbf{284} (2023), no.1, 109730.


 
 

 
 












\bibitem[CH08]{CH08} I. Chifan, C. Houdayer, {\it Bass-Serre rigidity results in von Neumann algebras}, Duke Math. J. {\bf 153} (2010),
no. 1, 23--54.



\bibitem[CI17]{CI17} I. Chifan, A. Ioana, \textit{Amalgamated free product rigidity for group von Neumann algebras}, Adv. Math. \textbf{329} (2018), 819--850.

\bibitem[CIK13]{CIK13} I. Chifan, A. Ioana, Y. Kida, \textit{$W^*$-superrigidity for arbitrary actions of central quotients of braid groups}, Math. Ann. \textbf{361} (2015), 563--582.



\bibitem[CIOS21]{CIOS21} I. Chifan, A. Ioana, D. Osin, B. Sun, {\it Wreath-like product groups and rigidity of their von
Neumann algebras}, Ann. of Math., {\bf 198} (2023), no. 3, 1261--1303.


\bibitem[CIOS23]{CIOS23} I. Chifan, A. Ioana, D. Osin, B. Sun, {\it Uncountable families of $W^*$ and $C^*$-superrigid Kazhdan groups}, Preprint 2023.
  
 
\bibitem[CKP14]{CKP14} I. Chifan, Y. Kida, S. Pant, \textit{Primeness results for von Neumann algebras associated with surface braid groups}, Int. Math. Res. Not. IMRN 2016, no. 16, 4807--4848.

\bibitem[CK-E21]{CK-E21} I. Chifan, S. Kunnawalkam-Elyavalli,  \textit{Cartan Subalgebras in von Neumann Algebras Associated with Graph Product Groups}, to appear in Groups, Geom. Dyn. arXiv:2107.04710.
		


\bibitem[Co76]{Co76} A. Connes, {\it Classification of injective factors. Cases II$_{1}$, II$_{\infty}$, III$_{\lambda}$, $\lambda\neq 1$} Ann. of Math. (2) {\bf 104} (1976), no. 1, 73--115.

\bibitem[Co82]{Co82} A. Connes,  {\it Classification des facteurs. Operator algebras and applications, Part $2$},
(Kingston, Ont., 1980), 43--109, Proc. Sympos. Pure Math., 38, Amer. Math. Soc.,
Providence, R.I., 1982.


\bibitem[CU18]{CU18} I. Chifan, B. Udrea, \textit{Some rigidity results for II$_1$ factors arising from wreath products of property (T) groups}, J. Funct. Anal. {\bf 278} (2020), no. 7, 108419, 32pp.





		


		
 
















\bibitem[DHI16]{DHI16} D. Drimbe, D. Hoff, A. Ioana, {\it Prime II$_1$ factors arising from irreducible lattices in products of rank one simple Lie groups}, J. Reine Angew. Math. {\bf 757} (2019), 197--246.  

\bibitem[DK-E21]{DK-E21} C. Ding, S. Kunnawalkam Elayavalli, Proper proximality for groups acting on
trees, Preprint 2021, arXiv:2107.02917.





\bibitem[Dr19a]{Dr19a} D. Drimbe, {\it Prime II$_1$ factors arising from actions of product groups}, J. Funct. Anal. {\bf 278} (2020), no. 5, 108366, 23.
		
\bibitem[Dr19b]{Dr19b} D. Drimbe, {\it Orbit equivalence rigidity for product actions}, Comm. Math. Phys. {\bf 379} (2020), no. 1, 41–59.
		
\bibitem[Dr20]{Dr20} D. Drimbe, {\it Product rigidity in von Neumann and C$^*$-algebras via s-malleable deformations}, Comm. Math. Phys. {\bf 388} (2021), no. 1, 329--349. 


		
\bibitem[FGS10]{FGS10} J. Fang, S. Gao, R. Smith, \textit{The Relative Weak Asymptotic Homomorphism Property for Inclusions of Finite von Neumann Algebras}, Internat. J. Math. {\bf 22} (2011), no. 7, 991--1011.




\bibitem[Ge95]{Ge95} L. Ge: {\it On maximal injective subalgebras of factors}, Adv. Math. {\bf 118} (1996), no. 1, 34-70.

\bibitem[Gr90]{Gr90}E. Green, {\it Graph Products of Groups}, PhD Thesis, The University of Leeds, 1990, \url{http://etheses.whiterose.ac.uk/236/}.

\bibitem[GM19]{GM19}  A. Genevois, A. Martin, {\it Automorphisms of graph products of groups from a geometric perspective}, Proc. Lond. Math. Soc. (3) \textbf{119} (2019), no. 6, 1745--1779

\bibitem[HW08]{HW08} F. Haglund, D. Wise, {\it Special cube complexes}, Geom. Funct. Anal. \textbf{17} (2008), no. 5, 1551--1620.

\bibitem[HH20]{HH20} C. Horbez, J. Huang, \textit{Measure equivalence classification of transvection-free right-angled Artin groups}, J. \'Ec. polytech. Math. \textbf{9} (2022), 1021--1067


\bibitem[HH21]{HH21} C. Horbez, J. Huang, \textit{Orbit equivalence rigidity of irreducible actions of right-angled Artin groups}, Preprint 2021, arXiv:2110.04141.

\bibitem[IPP05]{IPP05} A. Ioana, J. Peterson, S. Popa, \textit{Amalgamated free products of weakly rigid factors and calculation of their symmetry groups}. Acta Math. {\bf 200} (2008), no. 1, 85--153. 
 
 



\bibitem[Io10]{Io10} A. Ioana, \textit{$W^*$-superrigidity for Bernoulli actions of property (T) groups}, J. Amer. Math. Soc. \textbf{24} (2011), no. 4, 1175--1226.


\bibitem[Io11]{Io11} A. Ioana, \textit{Uniqueness of the group measure space decomposition for Popa's $\mathcal H\mathcal J$ factors}, Geom. Funct. Anal.  \textbf{22} (2012), no. 3, 699--732. 

 


\bibitem [Io12]{Io12} A. Ioana, {\it Classification and rigidity for von Neumann algebras}, European Congress of Mathematics, EMS (2013), 601-625.

\bibitem[Io12b]{Io12b} A. Ioana, {\it Cartan subalgebras of amalgamated free product II$_1$ factors}, with an appendix by Adrian Ioana and Stefaan Vaes, Ann. Sci. \'Ec. Norm. Sup\'er. (4) {\bf 48} (2015), no. 1, 71--130.

\bibitem[Io17]{Io17} A. Ioana, {\it Rigidity for von Neumann algebras}, Proceedings of the International Congress of Mathematicians-Rio de Janeiro 2018. Vol. III. Invited lectures, 1639-1672, World Sci. Publ., Hackensack, NJ, 2018.

 
\bibitem[IPV10]{IPV10} A. Ioana, S. Popa, S. Vaes, \textit{A class of superrigid group von Neumann algebras}, Ann. of Math. (2) {\bf 178} (2013), no. 1, 231--286.

\bibitem[Jo81]{Jo81} V. F. R. Jones, \textit{Index for subfactors}, Invent. Math. {\bf 72} (1983), no. 1, 1--25.


\bibitem[Jo00]{Jo00} V. F. R. Jones, \textit{Ten problems. In Mathematics: frontiers and perspectives} (ed. by V. Arnold, M. Atiyah, P. Lax and B. Mazur), 79-91, Amer. Math. Soc., Providence, RI, 2000.

\bibitem[KV15]{KV15} A. S. Krogager, S. Vaes, \textit{A class of \text{ }$\rm II_1$\text{ }factors with exactly two group measure space decompositions}, J. Math. Pures Appl. (9) \textbf{108} (2017), no. 1, 88--110.

	


\bibitem[MO13]{MO13}  A. Minasyan, D. Osin, \textit{Acylindrical hyperbolicity of groups acting on trees}, Math. Ann. $\textbf{362}$ (2015), no. 3-4, 1055--1105.


\bibitem[MvN43]{MvN43} F.J. Murray, J. von Neumann, {\it On rings of operators} IV, Ann. of Math (2) {\bf 44} (1943), 716--808.

 \bibitem[MY02]{MY02} I. Mineyev, G. Yu, {\it The Baum-Connes conjecture for hyperbolic groups}, Invent. Math. {\bf 149} (2002), no. 1, 97--122.
\bibitem[O-O98]{O-O98} H. Oyono-Oyono. \textit{La conjecture de Baum-Connes pour les groupes agissant sur
les arbres.} C. R. Acad. Sci. Paris S\'er. I Math. $\textbf{326}$ (1998), no. 7, 799--804.

\bibitem[O-O01b]{O-O01b} H. Oyono-Oyono, \textit{Baum-Connes Conjecture and extensions}, J. Reine Angew. Math. \textbf{532} (2001) 133--149.




\bibitem[OP03]{OP03} N. Ozawa and S. Popa, {\it Some prime factorization results for type II$_{1}$ factors}, Invent. Math. {\bf 156} (2004), no. 2, 223--234.

\bibitem[OP07]{OP07}  N. Ozawa, S. Popa, {\it On a class of II$_1$ factors with at most one Cartan subalgebra}, Ann. of
Math (2) {\bf172} (2010), no. 1, 713--749.

\bibitem[Po99]{Po99} S. Popa, \textit{Some properties of the symmetric enveloping algebra of a factor, with applications to amenability
			and property (T)}, Doc. Math. \textbf{4} (1999), 665-744.

\bibitem[Po01]{Po01} S. Popa, \textit{On a class of type II$_1$ factors with Betti numbers invariants}, Ann. of Math (2) \textbf{163} (2006), 809--899.

\bibitem[Po03]{Po03} S. Popa, \textit{Strong rigidity of $\textrm{II}_1$ factors arising from malleable actions of $w$-rigid groups I}, Invent. Math. \textbf{165}  (2006), no. 2, 369--408.

\bibitem[Po04]{Po04} S. Popa, \textit{Strong rigidity of $\textrm{II}_1$ factors arising from malleable actions of $w$-rigid groups II}, Invent. Math. {\bf 165} (2006), no. 2, 409--452.

\bibitem[Po06]{Po06} S. Popa, \textit{Deformation and rigidity for group actions and von Neumann algebras}, International Congress of Mathematicians. Vol. I, 445--477, Eur. Math. Soc., Z\"urich, 2007.

\bibitem[PP86]{PP86} M. Pimsner, S. Popa, {\it Entropy and index for subfactors}, Ann. Sci. \'Ec. Norm. Sup\'er {\bf 19} (1986), 57–106.


\bibitem[PV09]{PV09} S. Popa, S. Vaes, {\it Group measure space decomposition of II$_1$ factors and W$^*$–superrigidity}, Invent. Math. {\bf 182} (2010), no. 2, 371-417

\bibitem[PV11]{PV11} S. Popa, S. Vaes, {\it Unique Cartan decomposition for $\rm II_1$ factors arising from arbitrary actions of free groups}, Acta Math. \textbf{212 }(2014), no. 1, 141--198.
		
\bibitem[PV12]{PV12} S. Popa, S. Vaes, \textit{Unique Cartan decomposition for $\rm II_1$ factors arising from arbitrary actions of hyperbolic groups}, J. Reine Angew. Math. \textbf{694} (2014), 215--239.

\bibitem[SS10]{SS10} A. Sinclair, R. Smith, {\it Finite von Neumann algebras and MASAs}, London Mathematical Society Lecture Note Series, {\bf351}, Cambridge University Press, Cambridge, (2008). 
		
\bibitem[Va07]{Va07} S. Vaes, {\it Rigidity results for Bernoulli actions and their von Neumann algebras (after Sorin Popa)}, Séminaire Bourbaki, Vol. 2005/2006. Astérisque {\bf 311} (2007), no. 961, viii, 237--294.

\bibitem[Va08]{Va08} S. Vaes, {\it Explicit computations of all finite index bimodules for a family of II$_1$ factors}, Ann. Sci. \'Ec. Norm. Sup\'er (4) {\bf41} (2008), no. 5, 743--788.



\bibitem[Va10a]{Va10a} S. Vaes, {\it One-cohomology and the uniqueness of the group measure space decomposition of a II$_1$ factor}, Math. Ann. {\bf355} (2013), no. 2, 661--696.

\bibitem[Va10b]{Va10b} S. Vaes, \textit{Rigidity for von Neumann algebras and their invariants}, Proceedings of the International Congress of Mathematicians. Volume III, 1624-1650, Hindustan Book Agency, New Delhi, 2010.

\bibitem[Wi11]{Wi11} D. T. Wise, \textit{Research announcement: the structure of groups with a quasiconvex hierarchy}, Electron. Res.
Announc. Math. Sci. $\textbf{16}$ (2009), 44--55.

\end{thebibliography}
\end{document}